\documentclass[11pt,reqno]{amsart}
\usepackage{amsmath,amssymb,extarrows}
\usepackage{hyperref}
\usepackage{url}
\usepackage{tikz,enumerate}
\usepackage{diagbox}
\usepackage{appendix}
\usepackage{epic}
\allowdisplaybreaks[4]

\newtheorem{theorem}{Theorem}
\newtheorem{corollary}[theorem]{Corollary}

\newtheorem{lemma}[theorem]{Lemma}
\newtheorem{prop}[theorem]{Proposition}
\theoremstyle{definition}

\theoremstyle{remark}
\newtheorem{rem}{Remark}
\numberwithin{equation}{section}
\numberwithin{theorem}{section}
\numberwithin{defn}{section}

\DeclareMathOperator{\sg}{sg}

\begin{document}
\title[Representations of mock theta functions]
 {Representations of mock theta functions}

\author{Dandan Chen and Liuquan Wang}
\address{School of Mathematical Sciences, East China Normal University, Shanghai, People's Republic of China}
\email{ddchen@stu.ecnu.edu.cn}

\address{School of Mathematics and Statistics, Wuhan University, Wuhan 430072, Hubei, People's Republic of China}
\email{wanglq@whu.edu.cn;mathlqwang@163.com}




\begin{abstract}
Motivated by the works of Liu, we provide a unified approach to find Appell-Lerch series and Hecke-type series representations for mock theta functions. We establish a number of parameterized identities with two parameters $a$ and $b$. Specializing the choices of $(a,b)$, we not only give various known and new representations for the mock theta functions of orders 2, 3, 5, 6 and 8, but also present many other interesting identities. We find that some mock theta functions of different orders are related to each other, in the sense that their representations can be deduced from the same $(a,b)$-parameterized identity. Furthermore, we introduce the concept of false Appell-Lerch series. We then express the Appell-Lerch series, false Appell-Lerch series and Hecke-type series in this paper using the building blocks $m(x,q,z)$ and $f_{a,b,c}(x,y,q)$ introduced by Hickerson and Mortenson, as well as $\overline{m}(x,q,z)$ and $\overline{f}_{a,b,c}(x,y,q)$ introduced in this paper. We also show the equivalences of our new representations for several mock theta functions and the known representations.
\end{abstract}

\maketitle
\tableofcontents

\section{Introduction}\label{sec:1}
In his last letter to Hardy dated on January 12, 1920, Ramanujan gave a list of 17 functions which he called ``mock theta functions''. He defined each function as a $q$-series in Eulerian form and separated them into four classes: one class of third order, two of fifth order, and one of seventh order. Ramanujan also stated identities satisfied by mock theta functions of the same order. In his lost notebook \cite{lostnotebook}, identities for mock theta functions of the sixth and tenth orders were recorded. Since then, mock theta functions have attracted the attention of many mathematicians.

It was some time before researchers understood the modularity of mock theta functions. From the Eulerian forms of mock theta functions, it is difficult to observe any significant transformation properties. Therefore, finding alternative representations for mock theta functions becomes the first task for studying their modular behaviors. With the contribution of many works,  Watson \cite{Watson,Watson-2}, Andrews \cite{Andrews-TAMS}, Andrews and Hickerson \cite{Andrews-Hickerson}, Berndt and Chan \cite{Berndt-Chan}, Choi \cite{Choi-1,Choi-2,Choi-3,Choi-4}, Garvan \cite{Garvan-arXiv}, Gordon and McIntosh \cite{Gordon-McIntosh}, Hickerson \cite{Hickerson}, and Zwegers \cite{Zwegers-Rama}, to name a few, we now know that mock theta functions usually admit at least one of the two kinds of representations: Appell-Lerch series or Hecke-type series. A complete list of Appell-Lerch series representations for classical mock theta functions can be found in  \cite[Section 5]{Hickerson-Mortenson}.

Appell-Lerch series are series of the form
\begin{align}
\sum_{n=-\infty}^{\infty}\frac{(-1)^{\ell n}q^{\ell n(n+1)/2}b^n}{1-aq^n}. \label{Appell-defn}
\end{align}
Here and throughout this paper, we assume that $|q|<1$. After multiplying the series \eqref{Appell-defn} by the factor $a^{\ell/2}$ and viewing it as function in the variables $a,b$ and $q$, it is also usually refereed as a level $\ell$ Appell function. This kind of series was first studied by Appell \cite{Appell} and Lerch \cite{Lerch}.

A series is of Hecke-type if it has the following form:
\begin{eqnarray*}
\sum_{(m,n)\in D}(-1)^{H(m,n)}q^{Q(m,n)+L(m,n)},
\end{eqnarray*}
where $H$ and $L$ are linear forms, $Q$ is a quadratic form, and $D$ is some subset of $ \mathbb{Z}\times \mathbb{Z}$ such that $Q(m,n)\geq 0$ for every $(m,n)\in D$. Historically $Q(m,n)$ was assumed to be indefinite (see \cite{Andrews-Hecke} for example). Here we allow $Q(m,n)$ to be definite as well. The following classical identity of Jacobi expresses an infinite product as a Hecke-type series
\cite[Eq.\ (3.15)]{Andrews-Hecke}:
\begin{align*}
(q;q)_{\infty}^{3}=\sum_{n=-\infty}^\infty\sum_{m\geq|n|}(-1)^mq^{(m^2+m)/2}.
\end{align*}
Here and later we use the standard $q$-series notation:
\begin{align*}
&(x;q)_0:=1, \quad (x;q)_n:=\prod_{k=0}^{n-1}(1-xq^k), \\
&(x)_\infty=(x;q)_\infty:=\prod_{k=0}^\infty(1-xq^k), \quad |q|<1.
\end{align*}

Motivated by Jacobi's identity, Hecke \cite{hecke59} systematically investigated theta series related to indefinite quadratic forms. For instance, Hecke \cite[p. 425]{hecke59} found that
\begin{eqnarray*}
\sum_{n=-\infty}^\infty\sum_{|m|\leq n/2}(-1)^{m+n}q^{(n^2-3m^2)/2+(n+m)/2}=\prod_{n=1}^\infty(1-q^n)^2,
\end{eqnarray*}
which is originally due to Rogers \cite[p. 323]{Rogers94}. Kac and Peterson \cite{Kac} illustrated ways for proving  Hecke-type identities using affine Lie algebra.

Appell-Lerch series and Hecke-type series played important roles in $q$-series. It serves as bridges between mock theta functions and the theory of modular forms. For example, for the third order mock theta function \footnote{Throughout this paper, to avoid confusion, we use a superscript $(n)$ to indicate that a mock theta function is of order $n$.}
\begin{align}
f^{(3)}(q)=\sum_{n=0}^{\infty}\frac{q^{n^2}}{(-q;q)_{n}^2},
\end{align}
Watson \cite{Watson} found the following Appell-Lerch series representation:
\begin{align}
f^{(3)}(q)=\frac{2}{(q;q)_{\infty}}\sum_{n=-\infty}^{\infty}\frac{(-1)^nq^{\frac{3}{2}n^2+\frac{1}{2}n}}{1+q^n}. \label{intro-f(q)}
\end{align}
For the fifth order mock theta function
\begin{align}
f_0^{(5)}(q)&=\sum_{n=0}^{\infty}\frac{q^{n^2}}{(-q;q)_{n}}, \label{f0-defn}
\end{align}
Andrews \cite{Andrews-TAMS} showed that $f_0^{(5)}(q)$ has a Hecke-type series expression as:
\begin{align}
f_0^{(5)}(q)&=\frac{1}{(q;q)_{\infty}}\sum_{n=0}^{\infty}\sum_{|j|\leq n}(-1)^jq^{\frac{5}{2}n^2+\frac{1}{2}n-j^2}(1-q^{4n+2}). \label{intro-f0}
\end{align}
By using the Appell-Lerch series or Hecke-type series expressions of mock theta functions such as \eqref{intro-f(q)} and \eqref{intro-f0}, Zwegers \cite{Zwegers} successfully fit mock theta functions into the theory of modular forms.  For more detailed introduction to the developments of mock theta functions, we refer the reader to the survey of Gordon and McIntosh \cite{Gordon-McIntosh-Survey}, the paper of Hickerson and Mortenson \cite{Hickerson-Mortenson}, the recent books of Andrews and Berndt \cite{lost-notebook5}, and Bringmann et al. \cite{BFOR} as well as the references listed there.

There are several ways for establishing Appell-Lerch series or Hecke-type series expressions for mock theta functions. To deduce Appell-Lerch series representations of third order mock theta functions, Watson used a transformation formula connecting a terminated ${}_{8}\phi_{7}$ series to a terminated ${}_{4}\phi_{3}$ series.  Andrews \cite{Andrews-TAMS} used Bailey chain theory to produce Hecke-type series expressions for the fifth and seventh order mock theta functions. Liu \cite{Liu2002} derived some $q$-series expansion formulas and gave new proofs to \eqref{intro-f0}. In a series of works, Liu \cite{Liu2013Rama,Liu2013IJNT} established some transformation formulas for $q$-hypergeometric series and thereby proved many interesting Hecke-type identities such as (see \cite[Proposition 1.11]{Liu2013IJNT})
\begin{align}\label{intro-Liu-eq}
\sum_{n=0}^{\infty}\frac{(q;q^2)_{n}q^n}{(q^2;q^2)_{n}}=\frac{(q;q^2)_{\infty}}{(q^2;q^2)_{\infty}}\sum_{n=0}^{\infty}\sum_{j=-n}^{n}(-1)^{n+j}q^{n^2+n-j^2}.
\end{align}

Among a number of nice transformation formulas of Liu, the following one has shown its power in establishing Hecke-type identities (see \cite[Theorem 1.7]{Liu2013Rama} or \cite[p.\ 2089]{Liu2013IJNT}).
\begin{theorem}\label{thm-main}
For $\max \{|uab/q|, |ua|, |ub|, |c|, |d|\}<1$, we have \footnote{For the convergence, we only need to assume that $|uab/q|<1$. We require $\max\{|ua|,|ub|,|c|,|d|\}<1$ so that the denominator of each term appearing in the identity does not vanish.}
\begin{align}
&\frac{(uq,uab/q;q)_\infty}{(ua,ub;q)_\infty}{}_{3}\phi_{2}\bigg(\genfrac{}{}{0pt}{}{q/a,q/b,v}{ c, d};q,\frac{uab}{q}\bigg) \nonumber \\
=&\sum_{n=0}^{\infty}\frac{(1-uq^{2n})(u,q/a,q/b;q)_n}{(1-u)(q,ua,ub;q)_n}(-uab)^nq^{(n^2-3n)/2}\times {}_{3}\phi_{2}\bigg(\genfrac{}{}{0pt}{}{q^{-n},uq^{n},v}{c,d};q,q\bigg). \label{thm-main-eq}
\end{align}
\end{theorem}
Utilizing Theorem \ref{thm-main}, several new Hecke-type identities have been found by Wang and Yee \cite{WangThesis,Wang-Yee}.
In particular, they proved that \cite[Theorem 1.1]{Wang-Yee}
\begin{align}
\sum_{n=1}^{\infty}\frac{q^n(q;q^2)_{n}}{(-q;q^2)_{n}(1+q^{2n})}=\sum_{n=1}^{\infty}\sum_{|j|\leq n}(-1)^jq^{n^2+j^2}-\sum_{n=1}^{\infty}(-1)^nq^{2n^2}. \label{W-Y-eq}
\end{align}
Note that this is a Hecke-type identity associated with a definite theta series, which is quite rare in the literature.
Furthermore, Wang \cite{WangRama} used Theorem \ref{thm-main} to give new proofs for five false theta function identities of Ramanujan. In a recent work, Chan and Liu \cite{Chan-Liu} used Theorem \ref{thm-main} to establish three Hecke-type identities such as
\begin{align}\label{Chan-Liu-eq}
\sum_{n=1}^\infty\frac{(q;q)_n}{(-q;q)_n}(-1)^nq^{n(n-1)/2}=\sum_{n=1}^\infty\sum_{j=-n+1}^n(1-q^n)^2(-1)^{n+j+1}q^{2n^2-n-j^2}.
\end{align}

In this paper, we continue to employ  Theorem \ref{thm-main} to illustrate a systematic way to establish various representations for $q$-series in the Eulerian form. We first generalize Theorem \ref{thm-main} to the following form.
\begin{theorem} \label{thm-key}
Suppose $\max \{|\alpha abz/q|, |\alpha a|,|\alpha b|,|\alpha c_1|,\cdots,|\alpha c_m|\}<1$ and $m$ is a nonnegative integer. We have
\begin{align}
&\frac{(\alpha q,\alpha ab/q;q)_\infty}{(\alpha a,\alpha b;q)_\infty}
{}_{m+2}\phi_{m+1}\bigg(\genfrac{}{}{0pt}{}{q/a,q/b,\alpha b_1,\cdots,\alpha b_m}
{0,\alpha c_1,\cdots,\alpha c_m};q,\alpha abz/q\bigg) \nonumber\\
&=\sum_{n=0}^\infty\frac{(1-\alpha q^{2n})(\alpha,q/a,q/b;q)_n(-\alpha  ab/q)^n
  q^{n(n-1)/2}}{(1-\alpha)(q,\alpha a,\alpha b;q)_n} \\
\nonumber
&\quad \times{}_{m+2}\phi_{m+1}\bigg(\genfrac{}{}{0pt}{}{q^{-n},\alpha q^n,\alpha b_1,\cdots,\alpha b_m}
{0,\alpha c_1,\cdots,\alpha c_m};q,zq\bigg).
\end{align}
\end{theorem}

Theorems \ref{thm-main} and \ref{thm-key} provide an elegant way for finding alternative representations for basic hypergeometric series. Indeed, these theorems allow us to write a series in Eulerian form as a sum involving truncated $_{m+2}\phi_{m+1}$ series. In suitable situations, this truncated sum may be further simplified and in turn we obtain very nice representations of the original series. Such nice representations involve Appell-Lerch series or Hecke-type series. For example, using Theorem \ref{thm-main} we establish the following $(a,b)$-parameterized identity (see Theorem \ref{thm-ab-5-8}): for $\max \{|ab|, |aq^2|, |bq^2|\}<1$,
\begin{align}
&\frac{(q^2,ab;q^2)_\infty}{(aq^2,bq^2;q^2)_\infty}
{}_3\phi_2\bigg(\genfrac{}{}{0pt}{}{q^2/a,q^2/b,q^2}{0,q^3};q^2,ab\bigg) \nonumber\\
&=(1-q)\sum_{n=0}^{\infty}(1+q^{2n+1})\frac{(q^2/a,q^2/b;q^2)_n}{(aq^2,bq^2;q^2)_n}(-ab)^nq^{3n^2+2n}\sum_{j=-n}^{n}q^{-2j^2-j}. \label{example}
\end{align}
It turns out that by choosing suitable values for $(a,b)$ in this identity, we get Hecke-type series representations for four mock theta functions of orders 5, 6 and 8. Namely, we find that
\begin{align}
F_1^{(5)}(q):=&~\sum_{n=0}^{\infty} \frac{q^{2n(n+1)}}{(q;q^2)_{n+1}}    \nonumber \\
=&~\frac{1}{(q^2;q^2)_{\infty}}\sum_{n= 0}^\infty \sum_{j=0}^{2n}(-1)^nq^{5n^2+4n-\binom{j+1}{2}}(1+q^{2n+1}),\\
\psi_{-}^{(6)}(q):=&~\sum_{n=1}^{\infty}\frac{q^{n}(-q;q)_{2n-2}}{(q;q^2)_{n}}  \nonumber \\
=&~q\frac{(-q;q)_{\infty}}{(q;q)_{\infty}}\sum_{n=0}^{\infty}(-1)^nq^{3n^3+3n}\sum_{j=-n}^{n}q^{-2j^2-j},  \\
T_1^{(8)}(q):=&~\sum_{n=0}^{\infty}\frac{q^{n(n+1)}(-q^2;q^2)_{n}}{(-q;q^2)_{n+1}}  \nonumber \\
=&~\frac{(-q^2;q^2)_{\infty}}{(q^2;q^2)_{\infty}}\sum_{n=0}^{\infty}q^{4n^2+3n}(1-q^{2n+1})\sum_{j=-n}^{n}(-1)^jq^{-2j^2-j},  \\
V_1^{(8)}(q):=&~\sum_{n=0}^{\infty}\frac{q^{(n+1)^2}(-q;q^2)_n}{(q;q^2)_{n+1}}   \nonumber \\
=&~q\frac{(-q;q^2)_{\infty}}{(q^2;q^2)_{\infty}}\sum_{n=0}^{\infty}(-1)^nq^{4n^2+4n}\sum_{j=-n}^{n}q^{-2j^2-j}. \label{intro-V1-1}
\end{align}

Furthermore, the parameterized identity \eqref{example} also generates new Hecke-type identities. For instance, if we take $(a,b)\rightarrow (0,1)$ in \eqref{example}, we get the following identity which seems to be new:
\begin{align}
\sum_{n=0}^\infty(-1)^n\frac{(q^2;q^2)_n}{(q;q^2)_{n+1}}q^{n(n+1)}
=\sum_{n=0}^\infty\sum_{j=-n}^n(1+q^{2n+1})q^{4n^2+3n-2j^2-j}. \label{example-T8}
\end{align}

By establishing different $(a,b)$-parameterized identities, we are able to provide different representations for the same mock theta function. An  example is that using two parameterized identities other than \eqref{example}, we find two Appell-Lerch series representations for $V_1^{(8)}(q)$. Namely,
\begin{align}
V_1^{(8)}(q)=&~\frac{(-q^4;q^4)_\infty}{(q^4;q^4)_\infty}\sum_{n=-\infty}^{\infty}\frac{(-1)^nq^{(2n+1)^2}}{1-q^{4n+1}} \label{intro-V1-2}\\
=&~q\frac{(-q;q)_\infty}{(q;q)_\infty}\sum_{n=-\infty}^\infty \frac{(-1)^nq^{n^2+2n}}{1+q^{4n+2}}. \label{intro-V1-3}
\end{align}
The formula \eqref{intro-V1-1} was found by Cui, Gu and Hao \cite{CGH} using Bailey pairs, and \eqref{intro-V1-2} is due to Gordon and McIntosh \cite{Gordon-McIntosh}. The representation \eqref{intro-V1-3} appears to be new, and in fact we will show that it is equivalent to \eqref{intro-V1-2} (see Section \ref{subsec-mock-V18}).

We will present 30 parameterized identities like \eqref{example}. Three of them were discovered by Liu \cite{Liu2013Rama} and the others are new to the best of our knowledge. By choosing suitable values for the parameters, we provide new proofs for most of the known Appell-Lerch series or Hecke-type series representations for mock theta functions of orders 2, 3, 5, 6 and 8. Meanwhile, we will also show many new Hecke-type identities associated to definite or indefinite quadratic forms.

In order to write Appell-Lerch series or Hecke-type series representations in a  standard way and thus having clearer understanding of their modular properties, Hickerson and Mortenson \cite{Hickerson-Mortenson} introduced two functions $m(x,q,z)$ and $f_{a,b,c}(x,y,q)$ (see \eqref{m-defn} and \eqref{f-defn} for definition). They  serve as building blocks for  Appell-Lerch series and a large family of Hecke-type series, respectively. Moreover, their modular properties  have been well studied by Zwegers \cite{Zwegers}. Hickerson and Mortenson \cite{Hickerson-Mortenson} also developed an approach to express a Hecke-type series in terms of Appell-Lerch series. In particular, they \cite[Section 5]{Hickerson-Mortenson} gave representations for all classical mock theta functions in terms of $m(x,q,z)$.

In this paper, we will follow \cite{Hickerson-Mortenson} and write the Appell-Lerch series and Hecke-type series we obtained in terms of these building blocks. Along this process, we find that there are certain series whose shape is similar to Appell-Lerch series but cannot be expressed in terms of $m(x,q,z)$. For example, in \eqref{2-1-cor-4} we establish the following identity
\begin{align}
\sum_{n=0}^{\infty}\frac{(-1)^nq^{n(n+1)}(q^2;q^2)_n}{(q;q^2)_{n+1}^2}=\sum_{n=0}^{\infty}\frac{q^{2n^2+2n}}{1-q^{2n+1}}
-\sum_{n=-\infty}^{-1}\frac{q^{2n^2+2n}}{1-q^{2n+1}}.
\end{align}
This does not satisfy the definition of Appell-Lerch series in \eqref{Appell-defn}.
Similarly, there are some Hecke-type series which seems not expressible by $f_{a,b,c}(x,y,q)$. An example is \eqref{W-Y-eq} (see \eqref{W-Y-eq-new}). Therefore, we introduce two new functions $\overline{m}(x,q,z)$ and $\overline{f}_{a,b,c}(x,y,q)$ (see \eqref{barm-defn} and \eqref{barf-defn}). The definitions of these functions closely resemble that of $m(x,q,z)$ and $f_{a,b,c}(x,y,q)$. However, their modular properties are unclear to us and deserve future research. Since the only difference between $\overline{m}(x,q,z)$ and $m(x,q,z)$ is the signs in their summands, we call $\overline{m}(x,q,z)$ a \emph{false Appell-Lerch series}. We will express almost all the Appell-Lerch series, false Appell-Lerch series and Hecke-type series in this paper using these building blocks $m(x,q,z)$, $f_{a,b,c}(x,y,q)$, $\overline{m}(x,q,z)$ and $\overline{f}_{a,b,c}(x,y,q)$. In particular, we use the method in \cite{Hickerson-Mortenson} to convert a representation in terms of $f_{a,b,c}(x,y,q)$ to a representation in terms of $m(x,q,z)$.

Besides seeing the modularity of a series clearly, there is another advantage for expressing Appell-Lerch series and Hecke-type series using $m(x,q,z)$ or $f_{a,b,c}(x,y,q)$. That is, we can use properties of these building blocks to transform between different forms. By doing this, we can see if different series representations are equivalent or not. For example, after writing \eqref{intro-V1-2} and \eqref{intro-V1-3} in terms of $m(x,q,z)$ and using properties of $m(x,q,z)$ established in \cite{Hickerson-Mortenson}, we find that they are in fact equivalent (see Section \ref{subsec-mock-V18}). We also show that  our new Hecke-type series representations of the third order mock theta functions $\psi^{(3}(q)$ and $\nu^{(3)}(q)$ are equivalent to the known representations found by Andrews \cite{A12} or Mortenson \cite{Mortenson-2013}.

The paper is organized as follows. In Section \ref{sec-pre}, we first  recall some formulas from the theory of basic hypergeometric series. We also discuss some limiting cases of Watson's $q$-analog of Whipple's theorem. The formulas listed in Section \ref{subsec-basic} will be used in evaluating certain terminated ${}_{3}\phi_{2}$ series, which are fundamental for establishing parameterized identities. Then in Section \ref{subsec-build} we give the definitions and useful properties of the building blocks of Appell-Lerch series, false Appell-Lerch series and Hecke-type series.  As examples, we will rewrite the identities \eqref{intro-Liu-eq}, \eqref{W-Y-eq} and \eqref{Chan-Liu-eq} using these building blocks.   In Section \ref{sec:proof}, we shall first prove Theorems \ref{thm-main} and \ref{thm-key}. Then as first examples of these theorems, we give several parameterized identities and new Hecke-type identities.  We continue to apply Theorem \ref{thm-main} in Sections \ref{sec-order-2}-\ref{sec-order-8}, where we discuss mock theta functions of orders 2, 3, 5, 6 and 8, respectively. For each mock theta function, we correspondingly establish some parameterized identities. Each of these identities gives us a representation for the mock theta function and produces new interesting identities.

\section{Preliminaries}\label{sec-pre}
In this section, we first collect some useful identities on basic hypergeometric series. Then we introduce several building blocks for expressing Appell-Lerch series and Hecke-type series and some formulas for simplifying such expressions.

Throughout this paper, we denote $\zeta_n:=e^{2\pi i/n}$. For convenience, we adopt the following compact notation:
\begin{align*}
(a_1,a_2,\cdots,a_m;q)_n&=(a_1;q)_n(a_2;q)_n\cdots(a_m;q)_n,\\
(a_1,a_2,\cdots,a_m;q)_\infty&=(a_1;q)_\infty(a_2;q)_\infty\cdots(a_m;q)_\infty.
\end{align*}

We recall the Jacobi's triple product identity:
\begin{align*}
j(x;q):=(x)_\infty (q/x)_\infty (q)_\infty =\sum_{n=-\infty}^{\infty} (-1)^nq^{\binom{n}{2}}x^n.
\end{align*}
Again for convenience, we denote
\begin{align}
j(x_1,x_2,\dots, x_n;q):=j(x_1;q)j(x_2;q)\cdots j(x_n;q).
\end{align}
As some special cases, we let $a$ and $m$ be rational numbers with $m$ positive and define
\begin{align*}
J_{a,m}:=j(q^a;q^m), \quad \overline{J}_{a,m}:=j(-q^a;q^m) \,\, \textrm{and} \,\, J_{m}:=J_{m,3m}=(q^m;q^m)_\infty.
\end{align*}

We will also use the following identities without mention (see \cite[Section 2]{Hickerson-Mortenson}).
\begin{align*}
&\overline{J}_{0,1}=2\overline{J}_{1,4}=2\frac{J_{2}^2}{J_{1}}, \quad \overline{J}_{1,2}=\frac{J_{2}^5}{J_{1}^2J_{4}^2}, \quad {J}_{1,2}=\frac{J_{1}^2}{J_{2}}, \quad \overline{J}_{1,3}=\frac{J_2J_{3}^2}{J_1J_6}, \\
& {J}_{1,4}=\frac{J_1J_4}{J_2}, \quad J_{1,6}=\frac{J_1J_6^2}{J_2J_3}, \quad \overline{J}_{1,6}=\frac{J_2^2J_3J_{12}}{J_1J_4J_6}.
\end{align*}
The following identities follow from the definition of $j(x;q)$ and can be found in \cite[Eqs.\ (2.2a), (2.2b)]{Hickerson-Mortenson}:
\begin{align}
&j(q^nx;q)=(-1)^nq^{-\binom{n}{2}}x^{-n}j(x;q), \quad n\in \mathbb{Z}, \label{j-id-1} \\
&j(x;q)=j(q/x;q)=-xj(x^{-1};q). \label{j-id-2}
\end{align}
We also recall the classical partial fraction expansion for the reciprocal of Jacobi's theta product (see \cite[p.\ 1]{lostnotebook} or \cite[p.\ 136]{Tannery}):
\begin{align}
\sum_{n=-\infty}^\infty \frac{(-1)^nq^{\binom{n+1}{2}}}{1-q^nz}=\frac{J_1^3}{j(z;q)}. \label{reciprocal-Jacobi}
\end{align}
Here $z$ is not an integral power of $q$. This formula will be used several times for simplifying expressions.

\subsection{Useful basic hypergeometric series identities}\label{subsec-basic}
In this subsection, we collect some useful identities, which will be important in deducing Appell-Lerch series or Hecke-type series representations from $q$-series of the Eulerian form.

The basic hypergeometric series ${}_r\phi_s$ is defined as \cite[Eq.\ (1.2.22)]{Gasper}
\begin{align*}
{}_r\phi_s \bigg(\genfrac{}{}{0pt}{}{a_1,  \cdots,  a_r}{b_1,   \dots,  b_s}; q,z \bigg)
=\sum_{n=0}^\infty\frac{(a_1,\cdots,a_r;q)_n}{(q,b_1,\cdots,b_s;q)_n}((-1)^nq^{n(n-1)/2})^{1+s-r}z^n.
\end{align*}
From \cite[Eq.\ (3.14)]{Liu2013IJNT}, we find
\begin{align}\label{sears:32}
{}_3\phi_2\bigg(\genfrac{}{}{0pt}{}{q^{-n},\alpha q^n,\beta}{c,d};q,q\bigg)
=(-c)^nq^{n(n-1)/2}\frac{(q\alpha/c;q)_n}{(c;q)_n} {}_3\phi_2\bigg(\genfrac{}{}{0pt}{}{q^{-n},\alpha q^n,d/\beta}{d,q\alpha/c};q,q\beta/c\bigg).
\end{align}
We also need the following formula \cite[p.\ 28]{Gasper} which relates a ${}_{3}\phi_{2}$ series to a ${}_2\phi_1$ series:
\begin{align}\label{hyper2}
{}_3\phi_2\bigg(\genfrac{}{}{0pt}{}{q^{-n},b,bzq^{-n}/c}{0, bq^{1-n}/c};q,q\bigg)
=\frac{(c;q)_n}{(c/b;q)_n}{}_2\phi_1\bigg(\genfrac{}{}{0pt}{}{q^{-n},b,}c;q,z\bigg).
\end{align}

Watson's $q$-analog of Whipple's theorem (see, for example, \cite[Eq.\ (2.5.1)]{Gasper}; \cite[Theorem 5]{Liu2002}) can be stated in the following form.
\begin{lemma}[Watson] \label{watson}
Let $n$ be a nonnegative integer. We have
\begin{align}
&\frac{(\alpha q,\alpha ab/q;q)_n}{(\alpha a,\alpha b;q)_n}
{}_4\phi_3\bigg(\genfrac{}{}{0pt}{}{q^{-n},q/a,q/b,\alpha cd/q }
{\alpha c,\alpha d,q^2/{\alpha abq^n}};q,q\bigg)\\
\nonumber
=&{}_8\phi_7\bigg(\genfrac{}{}{0pt}{}{q^{-n},q\sqrt a,-q\sqrt a,\alpha ,q/a,q/b,q/c,q/d}{\sqrt\alpha ,-\sqrt\alpha ,\alpha a,\alpha b,\alpha c,\alpha d,\alpha q^{n+1}};q,\alpha^2abcdq^{n-2}\bigg).
\end{align}
\end{lemma}
In the rest of this subsection, we discuss some consequences of Lemma \ref{watson}, which will be used frequently.

From \cite[Proposition 2.2]{Liu2013IJNT} we find
\begin{align}
&{}_3\phi_2\bigg(\genfrac{}{}{0pt}{}{q^{-n},\alpha q^{n+1},\alpha cd/q}{\alpha c,\alpha d};q,q\bigg) \nonumber \\
=&(-\alpha)^nq^{n(n+1)/2}\frac{(q;q)_n}{(q\alpha;q)_n}
\sum_{j=0}^n(-1)^j\frac{(1-\alpha q^{2j})(\alpha,q/c,q/d;q)_j}{(1-\alpha)(q,\alpha c,\alpha d;q)_j}(cd/q)^j
q^{-j(j+1)/2}. \label{Liu-eq-313}
\end{align}
\begin{lemma} \label{meq:1}
For any nonnegative integer $n$, we have
\begin{align}
&{}_3\phi_2\bigg(\genfrac{}{}{0pt}{}{q^{-n},\alpha q^{n+1},q/c}{\alpha d,q^2/c};q,d\bigg) \nonumber\\
=&(q/c)^n\frac{(\alpha c,q;q)_n}{(q^2/c,\alpha q;q)_n}
\sum_{j=0}^n(-1)^j\frac{(1-\alpha q^{2j})(\alpha,q/c,q/d;q)_j}{(1-\alpha)(q,\alpha c,\alpha d;q)_j}(cd/q)^jq^{-j(j+1)/2}.
\end{align}
\end{lemma}
\begin{proof}
Replacing $\alpha$ by $\alpha q$, $\beta$ by $\alpha cd/q$, $c$ by $\alpha c$ and $d$ by $\alpha d$ in \eqref{sears:32}, we have
\begin{align}
&{}_3\phi_2\bigg(\genfrac{}{}{0pt}{}{q^{-n},\alpha q^{n+1},\alpha cd/q}{\alpha c,\alpha d};q,q\bigg)  \nonumber\\
=&(-\alpha c)^nq^{n(n-1)/2}\frac{(q^2/c;q)_n}{(\alpha c;q)_n}
{}_3\phi_2\bigg(\genfrac{}{}{0pt}{}{q^{-n},\alpha q^{n+1},q/c}{\alpha d, q^2/c};q,d\bigg).
\end{align}
Combining this with \eqref{Liu-eq-313}, we get the desired identity.
\end{proof}
Letting $\alpha\rightarrow q^{-1}$ in Lemma \ref{meq:1}, we get the following result.
\begin{lemma}\label{lem-limit}
For any nonnegative integer $n$, we have
\begin{align}
&{}_3\phi_2\bigg(\genfrac{}{}{0pt}{}{q^{-n},q^{n}, q/c}{d/q,q^2/c};q,d\bigg)  \nonumber\\
=&(q/c)^n(1-q^n)\frac{(c/q;q)_n}{(q^2/c;q)_n}\nonumber\\
&\times\bigg(\frac{(c+d)q+cd(q-2)-q^3}{(c-q)(d-q)(1-q)}+\sum_{j=2}^{n}\frac{(-1)^j(1-q^{2j-1})}{(1-q^j)(1-q^{j-1})}
\frac{(q/c,q/d;q)_j}{(c/q,d/q;q)_j}(cd)^jq^{-j(j+3)/2}\bigg).
\end{align}
\end{lemma}

Letting $d\rightarrow\infty$ in \eqref{Liu-eq-313}, we find the following transformation formula, which appears as Proposition 2.3 in \cite{Liu2013IJNT}.
\begin{corollary}
For any nonnegative integer $n$, we have
\begin{align}\label{eq:1}
&{}_2\phi_1\bigg(\genfrac{}{}{0pt}{}{q^{-n},\alpha q^{n+1}}{\alpha c};q,c\bigg) \nonumber\\
=&(-\alpha)^nq^{n(n+1)/2}\frac{(q;q)_n}{(q\alpha;q)_n}
\sum_{j=0}^n\frac{(1-\alpha q^{2j})(\alpha,q/c;q)_j}{(1-\alpha)(q,\alpha c;q)_j}(c/\alpha)^j
q^{-j(j+1)}.
\end{align}
\end{corollary}

\begin{lemma} \label{meq:2}
For any nonnegative integer $n$, we have
\begin{align}
&{}_3\phi_2\bigg(\genfrac{}{}{0pt}{}{q^{-n},\alpha q^{n+1},q}{0,q^2/c};q,q\bigg)  \nonumber\\
=&\frac{(\alpha c,q;q)_n}{(q^2/c,\alpha q;q)_n}(\alpha/c)^nq^{n^2+2n}
\sum_{j=0}^n\frac{(1-\alpha q^{2j})(\alpha,q/c;q)_j}{(1-\alpha)(q,\alpha c;q)_j}c^j\alpha^{-j}q^{-j^2-j}.
\end{align}
\end{lemma}
\begin{proof}
Replacing $(b, c)$ by $(\alpha b, \alpha c)$ in \eqref{hyper2}, we have
\begin{align}\label{Tchen1}
{}_3\phi_2\bigg(\genfrac{}{}{0pt}{}{q^{-n},\alpha b,bzq^{-n}/c}{0, bq^{1-n}/c};q,q\bigg)
=\frac{(\alpha c;q)_n}{( c/b;q)_n}{}_2\phi_1\bigg(\genfrac{}{}{0pt}{}{q^{-n},\alpha b}{\alpha c};q,z\bigg).
\end{align}
Now,  setting $(b,z)=(q^{n+1}, c)$ in (\ref{Tchen1}), we obtain
\begin{align}\label{3chen9}
{}_3\phi_2\bigg(\genfrac{}{}{0pt}{}{q^{-n},\alpha q^{n+1},q}{0, q^2/c};q,q\bigg)
=\frac{(\alpha c;q)_n}{( cq^{-n-1};q)_n}{}_2\phi_1\bigg(\genfrac{}{}{0pt}{}{q^{-n},\alpha q^{n+1}}{\alpha c};q,c\bigg).
\end{align}
Substituting \eqref{eq:1} into \eqref{3chen9}, we complete the proof of Lemma \ref{meq:2}.
\end{proof}

Taking $\alpha\rightarrow q^{-1}$ in Lemma \ref{meq:2}, we obtain
\begin{lemma}\label{lem-limit-2}
For any nonnegative integer $n$, we have
\begin{align}
&{}_3\phi_2\bigg(\genfrac{}{}{0pt}{}{q^{-n},q^{n},q}{0,q^2/c};q,q\bigg)  \nonumber\\
=&(1-q^n)\frac{(c/q;q)_n}{(q^2/c;q)_n}c^{-n}q^{n^2+n}\nonumber\\
&\times\bigg(\frac{q-2c+cq}{(1-q)(c-q)}+\sum_{j=2}^{n}
\frac{1-q^{2j-1}}{(1-q^{j-1})(1-q^j)}\frac{(q/c;q)_j}{(c/q;q)_j}c^jq^{-j^2}\bigg).
\end{align}
\end{lemma}

From \cite[Proposition 2.5]{Liu2013IJNT} we find \footnote{There is a typo in \cite{Liu2013IJNT} that $q^{n(n+1)/2}$ should be replaced by $(-1)^nq^{n(n+1)/2}$.}
\begin{align}\label{eq:pro2.5}
&(-1)^n\frac{(\alpha q;q)_n}{(q;q)_n}q^{n(n+2)/2}
{}_2\phi_1\bigg(\genfrac{}{}{0pt}{}{q^{-n},\alpha q^{n+1}}{\alpha c};q,1\bigg)\nonumber\\
=& \sum_{j=0}^{n}\frac{(1-\alpha q^{2j})(\alpha,q/c;q)_j}{(1-\alpha)(q,\alpha c;q)_j}q^{j^2-j}(\alpha c)^j.
\end{align}
\begin{lemma}\label{meq:3}
For any nonnegative integer $n$, we have
\begin{align}
&{}_3\phi_2\bigg(\genfrac{}{}{0pt}{}{q^{-n},\alpha q^{n+1},q/c}{0,q^2/c};q,q\bigg)\nonumber\\
=&\frac{(\alpha c,q;q)_n}{(q^2/c,\alpha q;q)_n}(q/c)^n
\sum_{j=0}^n\frac{(1-\alpha q^{2j})(\alpha,q/c;q)_j}{(1-\alpha)(q,\alpha c;q)_j}{(\alpha c)}^jq^{j^2-j}.
\end{align}
\end{lemma}
\begin{proof}
Taking $(b,c,z)=(\alpha q^{n+1}, \alpha c, 1)$ in \eqref{hyper2}, we obtain
\begin{align}
{}_3\phi_2\bigg(\genfrac{}{}{0pt}{}{q^{-n},\alpha q^{n+1},q/c}{0,q^2/c};q,q\bigg)
=\frac{(\alpha c;q)_n}{(cq^{-n-1};q)_n}
{}_2\phi_1\bigg(\genfrac{}{}{0pt}{}{q^{-n},\alpha q^{n+1}}{\alpha c};q,1\bigg).
\end{align}
Together with \eqref{eq:pro2.5}, we complete the proof of Lemma \ref{meq:3}.
\end{proof}

We will also need the $q$-Pfaff--Saalsch\"utz summation formula \cite[p.\ 40, Eq.\ (2.2.1)]{Gasper}:
\begin{align}
{_3}\phi_2 \bigg(\genfrac{}{}{0pt}{} {q^{-n},  aq^n,  aq/bc} {aq/b,  aq/c}; q,  q \bigg) &= \frac{(b,c;q)_n}{(aq/b, aq/c ;q)_n}  \left(\frac{aq}{bc}\right)^n.  \qquad \label{Pfaff}
\end{align}

\subsection{Building blocks for Appell-Lerch series, false Appell-Lerch series and Hecke-type series}\label{subsec-build}
Following Hickerson and Mortenson \cite{Hickerson-Mortenson}, we define
\begin{align}\label{m-defn}
m(x,q,z):=\frac{1}{j(z;q)}\sum_{r=-\infty}^{\infty} \frac{(-1)^rq^{\binom{r}{2}}z^r}{1-q^{r-1}xz},
\end{align}
where $x,z\in \mathbb{C}^{*}:=\mathbb{C} \backslash \{0\}$ with neither $z$ nor $xz$ an integral power of $q$. Let
\begin{align}\label{f-defn}
f_{a,b,c}(x,y,q):=\sum_{\sg(r)=\sg(s)}\sg(r)(-1)^{r+s}x^ry^sq^{a\binom{r}{2}+brs+c\binom{s}{2}}.
\end{align}
Here $x,y\in \mathbb{C}^{*}$ and $\sg(r):=1$ for $r\geq 0$ and $\sg(r):=-1$ for $r<0$.

In terms of the above building blocks, identities \eqref{intro-f(q)} and \eqref{intro-f0} can be rewritten
(see \cite[Eq.\ (5.4)]{Hickerson-Mortenson}) as
\begin{align}
f^{(3)}(q)=2m(-q,q^3,q)+2m(-q,q^3,q^2)=4m(-q,q^3,q)+\frac{J_{3,6}^2}{J_{1}} \label{intro-f(q)-m}
\end{align}
and \cite[Eq.\ (8.12)]{Hickerson-Mortenson}
\begin{align}
f_{0}^{(5)}(q)=\frac{1}{J_1}\left(f_{3,7,3}(q^2,q^2,q)+q^3f_{3,7,3}(q^7,q^7,q) \right)=\frac{1}{J_1}f_{3,7,3}(q^{5/8},-q^{5/8},-q^{1/4}). \label{intro-f0-f}
\end{align}

Hickerson and Mortenson were also able to convert \eqref{intro-f0-f} into an expression in terms of $m(x,q,z)$ (see \cite[Corollary 1.12]{Hickerson-Mortenson}).
Since the modular properties for $m(x,q,z)$ and $f_{a,b,c}(x,y,q)$ have been well studied in \cite{Zwegers}, it will be easier for us to understand the modular properties of an Appell-Lerch series or Hecke-type series after expressing them using these building blocks.

In this paper, we will also express the Appell-Lerch series or Hecke-type series we encountered in terms of these building blocks. In some cases, the expressions in $m(x,q,z)$ or $f_{a,b,c}(x,y,q)$ may be further simplified. For example, a closer examination on Liu's identity \eqref{intro-Liu-eq} shows that the right side can be simplified to $\frac{J_{2}^2}{J_1}$. For this we need the following Kronecker-type identity \cite[Eq.\ (1.2)]{Mortenson-2017}:
\begin{align}\label{Kronecker}
\sum_{\sg(r)=\sg(s)}\sg(r)x^ry^sq^{rs}=\frac{J_{1}^3j(xy;q)}{j(x;q)j(y;q)}.
\end{align}
We may rewrite the sum in the right side of \eqref{intro-Liu-eq} as
\begin{align}
&\sum_{n=0}^\infty \sum_{j=-n}^n (-1)^{n+j}q^{n^2+n-j^2} \nonumber \\
=& \frac{1}{2}\left(\sum_{n=0}^{\infty}\sum_{j=-n}^{n}(-1)^{n+j}q^{n^2+n-j^2}+\sum_{n=-\infty}^{-1}\sum_{j=n+1}^{-n-1}(-1)^{n+1+j}q^{n^2+n-j^2} \right) \nonumber \\
=&\frac{1}{2}\left(\sum_{\substack{n+j\geq 0\\ n-j\geq 0}}-\sum_{\substack{n+j<0\\ n-j<0}} \right)(-1)^{n+j}q^{n^2+n-j^2} \quad \textrm{(set $n+j=r, n-j=s$)}\nonumber \\
=&\frac{1}{2}\sum_{\substack{\sg(r)=\sg(s) \\r\equiv s \pmod{2}}}\sg(r)(-1)^rq^{rs+\frac{1}{2}r+\frac{1}{2}s}  \nonumber\\
=&\frac{1}{2}\left(\sum_{\sg(r)=\sg(s)}\sg(r)q^{4rs+r+s}-\sum_{\sg(r)=\sg(s)}\sg(r)q^{4rs+3r+3s+2}  \right) \label{intro-Liu-eq-H} \\
=&\frac{1}{2}\left(\frac{J_{4}^3j(q^2;q^4)}{j(q;q^4)^2}-q^2\frac{J_{4}^3j(q^6;q^4)}{j(q^3;q^4)^2}  \right) \quad \textrm{(by \eqref{Kronecker})} \nonumber \\
=&\frac{J_{2}^4}{J_{1}^2}.
\end{align}
Therefore,  we can rewrite \eqref{intro-Liu-eq} as
\begin{align}\label{Liu-eq-simplify}
\sum_{n=0}^{\infty}\frac{(q;q^2)_nq^n}{(q^2;q^2)_n}=\frac{J_{2}^2}{J_1}=\sum_{n=0}^\infty q^{n(n+1)/2},
\end{align}
which is one of Ramanujan's theta function. If we do not express the right side to the form \eqref{intro-Liu-eq-H}, i.e., $\frac{1}{2}\left(f_{0,1,0}(-q,-q,q^4)-q^2f_{0,1,0}(-q^3,-q^3,q^4) \right)$ and using \eqref{Kronecker} to further simplify, it would be difficult to guess that the right side is such a simple infinite product.

We collect the following properties for $m(x,q,z)$ and $f_{a,b,c}(x,y,q)$, which will be helpful for simplifying the final expressions. Following \cite{Hickerson-Mortenson}, the term ``generic'' will be used to mean that the parameters do not cause poles in the Appell-Lerch sums or in the quotients of theta functions.
\begin{lemma}\label{lem-m-prop}
(Cf.\ \cite[Proposition\ 3.1]{Hickerson-Mortenson}.) For generic $x,z\in \mathbb{C}^{*}$
\begin{align}
m(x,q,z)=m(x,q,qz), \label{m-id-1} \\
m(x,q,z)=x^{-1}m(x^{-1},q,z^{-1}), \label{m-id-2} \\
m(qx,q,z)=1-xm(x,q,z). \label{m-id-3}
\end{align}
\end{lemma}

\begin{lemma}\label{lem-m-minus}
(Cf.\ \cite[Theorem 3.3]{Hickerson-Mortenson}.) For generic $x,z_0,z_1\in \mathbb{C}^{*}$
\begin{align}
m(x,q,z_1)-m(x,q,z_0)=\frac{z_0J_1^3j(z_1/z_0;q)j(xz_0z_1;q)}{j(z_0;q)j(z_1;q)j(xz_0;q)j(xz_1;q)}. \label{m-minus}
\end{align}
\end{lemma}

\begin{lemma}\label{lem-m-decompose}
(Cf.\ \cite[Theorem 3.5]{Hickerson-Mortenson}.) For generic $x,z,z'\in \mathbb{C}^{*}$
\begin{align}
m(x,q,z)=& \sum_{r=0}^{n-1}q^{-\binom{r+1}{2}}(-x)^rm(-q^{\binom{n}{2}-nr}(-x)^n,q^{n^2},z')+\nonumber \\
& \frac{z'J_{n}^3}{j(xz;q)j(z';q^{n^2})}\sum_{r=0}^{n-1}\frac{q^{\binom{r}{2}}(-xz)^rj(-q^{\binom{n}{2}+r}(-x)^nzz';q^n)j(q^{nr}z^n/z';q^{n^2})}{j(-q^{\binom{n}{2}}(-x)^nz',q^rz;q^n)}.
\end{align}
\end{lemma}
As a corollary of this powerful formula, we have the following result.
\begin{lemma}\label{lem-m-add}
(Cf.\ \cite[Corollary 3.7]{Hickerson-Mortenson}.) For generic $x,z\in \mathbb{C}^{*}$
\begin{align}
m(x,q,z)=m(-qx^2,q^4,z^4)-\frac{x}{q}m(-\frac{x^2}{q},q^4,z^4)-\frac{J_2J_4j(-xz^2;q)j(-xz^3;q)}{xj(xz;q)j(z^4;q^4)j(-qx^2z^4;q^2)}.
\end{align}
\end{lemma}
The following lemma contains useful formulas for simplifying expressions in terms of $f_{a,b,c}(x,y,q)$.
\begin{lemma}\label{lem-f-prop}
(Cf.\ \cite[Propositions\ 6.1-6.2, Corollary 6.4]{Hickerson-Mortenson}.) For $x,y \in \mathbb{C}^{*}$
\begin{align}
&f_{a,b,c}(x,y,q)=f_{a,b,c}(-x^2q^a,-y^2q^c,q^4)-xf_{a,b,c}(-x^2q^{3a},-y^2q^{c+2b},q^4)\nonumber \\
&\quad \quad \quad  -yf_{a,b,c}(-x^2q^{a+2b},-y^2q^{3c},q^4)+xyq^{b}f_{a,b,c}(-x^2q^{3a+2b},-y^2q^{3c+2b},q^4), \label{f-id-1} \\
&f_{a,b,c}(x,y,q)=-\frac{q^{a+b+c}}{xy}f_{a,b,c}(q^{2a+b}/x,q^{2c+b}/y,q), \label{f-id-2} \\
&f_{a,b,c}(x,y,q)=-yf_{a,b,c}(q^bx,q^cy,q)+j(x;q^a), \label{f-id-3} \\
&f_{a,b,c}(x,y,q)=-xf_{a,b,c}(q^ax,q^by,q)+j(y;q^c). \label{f-id-4}
\end{align}
\end{lemma}

In \cite[Eq.\ (4.6)]{Hickerson-Mortenson} it was defined for $x$ being neither 0 nor an integral power of $q$ that
\begin{align}\label{h-defn}
h(x,q):=\frac{1}{j(q;q^2)}\sum_{n=-\infty}^\infty \frac{(-1)^nq^{n(n+1)}}{1-q^nx}.
\end{align}
It was also proved that for generic $x\in \mathbb{C}^{*}$ \cite[Proposition 4.4]{Hickerson-Mortenson}
\begin{align}\label{h-m}
h(x,q)=-x^{-1}m(x^{-2}q,q^2,x).
\end{align}
Meanwhile, Hickerson and Mortenson \cite[Eq.\ (4.8)]{Hickerson-Mortenson} also defined for $x^2$ being neither zero nor an integral power of $q^2$ that
\begin{align}
k(x,q):=\frac{1}{xj(-q;q^4)}\sum_{n=-\infty}^\infty \frac{q^{n(2n+1)}}{1-q^{2n}x^2}. \label{k-defn}
\end{align}
It was proved that \cite[Eq.\ (4.10)]{Hickerson-Mortenson}
\begin{align}
xk(x,q)=m(-x^2,q,x^{-2})+\frac{J_1^4}{2J_2^2j(x^2;q)}. \label{k-id}
\end{align}
We will use \eqref{h-m} and \eqref{k-id} later.

While $m(x,q,z)$ and $f_{a,b,c}(x,y,q)$ can be used to express Appell-Lerch series and a big family of Hecke-type series we encountered, there are still some series which are unlikely to be represented by them. For these series we have to introduce two new building blocks.

We define for $x,z \in \mathbb{C}^{*}$ with $xz$ not being an integral power of $q$ that
\begin{align}
\overline{m}(x,q,z)=\sum_{r=-\infty}^{\infty}\frac{(-1)^r\sg(r)q^{\binom{r}{2}}z^r}{1-q^{r-1}xz}. \label{barm-defn}
\end{align}
This closely resembles \eqref{m-defn} except that we do not have a factor of infinite product and we have introduced $\sg(r)$ in the summand.

In a way similar to \eqref{f-defn}, we define
\begin{align}\label{barf-defn}
\overline{f}_{a,b,c}(x,y,q):=\sum_{\sg(r)=\sg(s)}(-1)^{r+s}x^ry^sq^{a\binom{r}{2}+brs+c\binom{s}{2}}.
\end{align}
The difference between this function and $f_{a,b,c}(x,y,q)$ is that we have removed  $\sg(r)$ from the summand.

We have  some formulas similar to Lemma \ref{lem-f-prop} for $\overline{f}_{a,b,c}(x,y,q)$.
\begin{lemma}\label{lem-barf}
We have
\begin{align}
&\overline{f}_{a,b,c}(x,y,q)=\overline{f}_{a,b,c}(-x^2q^a,-y^2q^c,q^4)-x\overline{f}_{a,b,c}(-x^2q^{3a},-y^2q^{c+2b},q^4)\nonumber \\
&\quad \quad \quad  -y\overline{f}_{a,b,c}(-x^2q^{a+2b},-y^2q^{3c},q^4)+xyq^{b}\overline{f}_{a,b,c}(-x^2q^{3a+2b},-y^2q^{3c+2b},q^4), \label{barf-id-1} \\
&\overline{f}_{a,b,c}(x,y,q)=-\frac{q^{a+b+c}}{xy}\overline{f}_{a,b,c}(q^{2a+b}/x,q^{2c+b}/y,q). \label{barf-id-2}
\end{align}
\end{lemma}
\begin{proof}
Decompose the definition \eqref{barf-defn} depending on the parity of $r$ and $s$, we get \eqref{barf-id-1}. Replacing $(r,s)$ by $(-r-1,-s-1)$ we get \eqref{barf-id-2}.
\end{proof}

At this moment, the modular properties of  $\overline{m}(x,q,z)$ and $\overline{f}_{a,b,c}(x,y,q)$ are unclear and deserve investigation in the future. In particular, $\overline{m}(x,q,z)$ does not meet the definition of Appell-Lerch series given in \eqref{Appell-defn}. However, since the only difference between $m(x,q,z)$ and $\overline{m}(x,q,z)$ is the signs for their summands, which is similar to the case of false theta functions, we would like to call $\overline{m}(x,q,z)$ as  a \emph{false Appell-Lerch series}.

Special cases of the functions $\overline{m}(x,q,z)$ and $\overline{f}_{a,b,c}(x,y,q)$ exist in the literature. For example, Wang and Yee
\cite[Eqs.\ (5.1), (5.5)]{Wang-Yee} established the following identity
\begin{align}
\sum_{n=1}^\infty \frac{(-1)^{n-1}q^n(q^2;q^2)_{n-1}}{(-q^2;q^2)_n}=\sum_{n=1}^\infty \frac{q^{n(n+1)/2}}{1+q^{2n}}-\sum_{n=-\infty}^{-1}\frac{q^{n(n+1)/2}}{1+q^{2n}}. \label{W-Y-m-eq}
\end{align}
Upon decomposing the sums on the right side according to the parity of $n$, we can write the right side as
\begin{align}
&\sum_{n=1}^\infty \frac{q^{n(n+1)/2}}{1+q^{2n}}-\sum_{n=-\infty}^{-1}\frac{q^{n(n+1)/2}}{1+q^{2n}} \nonumber \\
= &\sum_{n=0}^\infty \frac{q^{2n^2+n}}{1+q^{4n}}+\sum_{n=0}^\infty \frac{q^{2n^2+3n+1}}{1+q^{4n+2}}-\frac{1}{2} -\sum_{n=-\infty}^{-1}\frac{q^{2n^2+n}}{1+q^{4n}}-\sum_{n=-\infty}^{-1} \frac{q^{2n^2+3n+1}}{1+q^{4n+2}} \nonumber \\
=& \overline{m}(q,q^4,-q^3)+q\overline{m}(q,q^4,-q^5)-\frac{1}{2}. \label{pre-example-1}
\end{align}

Another example is the following identity  proved by  Andrews, Dyson and Hickerson \cite[Theorem 1]{ADH}:
\begin{align}
\sigma(q):=& \sum_{n=0}^\infty \frac{q^{n(n+1)/2}}{(-q;q)_n} \nonumber \\
=&\sum_{n=0}^\infty \sum_{j=-n}^{n}(-1)^{n+j}q^{n(3n+1)/2-j^2}(1-q^{2n+1}). \label{ADH-id}
\end{align}
We find that we can express the Hecke-type series on the right side by $\overline{f}_{a,b,c}(x,y,q)$, but we are not able to express it by $f_{a,b,c}(x,y,q)$.
\begin{prop}
We have
\begin{align}
\sigma(q)=\overline{f}_{1,5,1}(q^{3/8},-q^{3/8},q^{1/4}). \label{ADH-id-barf}
\end{align}
\end{prop}
\begin{proof}
We have
\begin{align}
\sigma(q)=&\sum_{n=0}^{\infty}\sum_{j=-n}^{n}(-1)^{n+j}q^{n(3n+1)/2-j^2}-\sum_{n=0}^{\infty} \sum_{j=-n}^{n}(-1)^{n+j}q^{n(3n+1)/2+2n+1-j^2} \nonumber \\
=& \left(\sum_{n=0}^{\infty}\sum_{j=-n}^{n}+\sum_{n=-\infty}^{-1}\sum_{j=n+1}^{-n-1}\right)(-1)^{n+j}q^{n(3n+1)/2-j^2} \nonumber \\
& \quad \quad \quad \quad \quad \textrm{(replace $n$ by $-n-1$ in the second sum)}\nonumber\\
=& \left(\sum_{\substack{n+j\geq 0 \\n-j\geq 0}}+\sum_{\substack{n+j< 0 \\n-j< 0}}\right) (-1)^{n+j}q^{n(3n+1)/2-j^2} \quad \textrm{(set $n+j=r$, $n-j=s$)}\nonumber\\
=& \sum_{\substack{\sg(r)=\sg(s) \\r\equiv s \pmod{2}}} (-1)^rq^{\frac{1}{8}r^2+\frac{5}{4}rs+\frac{1}{8}s^2+\frac{1}{4}r+\frac{1}{4}s} \nonumber \\
=&\sum_{\sg(r)=\sg(s)}q^{\frac{1}{2}r^2+5rs+\frac{1}{2}s^2+\frac{1}{2}r+\frac{1}{2}s}-\sum_{\sg(r)=\sg(s)}q^{\frac{1}{2}r^2+5rs+\frac{1}{2}s^2+\frac{7}{2}r+\frac{7}{2}s+2} \nonumber \\
=&\overline{f}_{1,5,1}(-q,-q,q)-q^2\overline{f}_{1,5,1}(-q^4,-q^4,q) \nonumber \\
=&\overline{f}_{1,5,1}(q^{3/8},-q^{3/8},q^{1/4}).
\end{align}
Here the last equality follows by setting $(a,b,c,x,y,q)\rightarrow (1,5,1,q^{3/8},-q^{3/8},q^{1/4})$ in \eqref{barf-id-1}.
\end{proof}

In the same way, we can rewrite \eqref{W-Y-eq} and \eqref{Chan-Liu-eq} as
\begin{align}
\sum_{n=1}^{\infty}\frac{q^n(q;q^2)_{n}}{(-q;q^2)_{n}(1+q^{2n})}=\frac{1}{2}\left(\bar{f}_{1,0,1}(q^2,q^2,q^4)+q\bar{f}_{1,0,1}(q^4,q^4,q^4)-1 \right) \label{W-Y-eq-new}
\end{align}
and
\begin{align}
&\sum_{n=1}^{\infty}\frac{(q;q)_n}{(-q;q)_n}(-1)^nq^{n(n-1)/2}\nonumber \\
=& 1+\overline{f}_{1,3,1}(-q,-q,q^2)-\overline{f}_{1,3,1}(-1,-1,q^2)\nonumber \\
&  +q\overline{f}_{1,3,1}(-q^4,-q^4,q^2)-q^2\overline{f}_{1,3,1}(-q^5,-q^5,q^2) \nonumber \\
=&1+\overline{f}_{1,3,1}(q^{1/4},-q^{1/4},q^{1/2})-\overline{f}_{1,3,1}(q^{-1/4},-q^{-1/4},q^{1/2}). \label{Chan-Liu-barf}
\end{align}

For false Appell-Lerch series and those Hecke-type series which cannot be expressed by $f_{a,b,c}(x,y,q)$, we will express them using  $\overline{m}(x,q,z)$ and $\overline{f}_{a,b,c}(x,y,q)$, respectively.
In most of the time, this process is quite routine and just follow the tricks and steps in the deduction of \eqref{Liu-eq-simplify}, \eqref{pre-example-1} and \eqref{ADH-id-barf}. Therefore, we will omit the deduction process and simply write the final expressions. There are only a few cases which require special treatments, and for them we will give more details.

Hickerson and Mortenson \cite{Hickerson-Mortenson} also gave a way to convert $f_{a,b,c}(x,y,q)$ to a form in terms of $m(x,q,z)$ and infinite products. Following  \cite{Hickerson-Mortenson}, we define
\begin{align}
&g_{a,b,c}(x,y,q,z_1,z_0)\nonumber \\
:=&\sum_{t=0}^{a-1}(-y)^tq^{c\binom{t}{2}}j(q^{bt}x;q^a)m\left(-q^{a\binom{b+1}{2}-c\binom{a+1}{2}-t(b^2-ac)}\frac{(-y)^a}{(-x)^b},q^{a(b^2-ac)},z_0\right) \nonumber \\
& +\sum_{t=0}^{c-1}(-x)^tq^{a\binom{t}{2}}j(q^{bt}y;q^c)m\left(-q^{c\binom{b+1}{2}-a\binom{c+1}{2}-t(b^2-ac)}\frac{(-x)^c}{(-y)^b},q^{c(b^2-ac)},z_1 \right). \label{gabc-defn}
\end{align}
Among a number of results,  they established the following theorems, which will be used in this paper as well.
\begin{theorem}\label{thm-fg-1}
(Cf.\ \cite[Theorem 1.3]{Hickerson-Mortenson}.) Let $n$ and $p$ be positive integers with $(n,p)=1$. For generic $x,y\in \mathbb{C}^{*}$
\begin{align}
f_{n,n+p,n}(x,y,q)=g_{n,n+p,n}(x,y,q,-1,-1)+\Phi_{n,p}(x,y,q),
\end{align}
where
\begin{align}
&\Phi_{n,p}(x,y,q):=\nonumber \\
& \frac{J_{p^2(2n+p)}^3}{\overline{J}_{0,np(2n+p)}}\sum_{r^{*}=0}^{p-1}\sum_{s^{*}=0}^{p-1}q^{n\binom{r-(n-1)/2}{2}+(n+p)(r-(n-1)/2)(s+(n+1)/2)+n\binom{s+(n+1)/2}{2}}(-x)^{r-(n-1)/2} \nonumber \\
& \frac{(-y)^{s+(n+1)/2}j(-q^{np(s-r)}x^n/y^n;q^{np^2})j(q^{p(2n+p)(r+s)+p(n+p)}x^py^p;q^{p^2(2n+p)})    }{j(q^{p(2n+p)r+p(n+p)/2}(-y)^{n+p}/(-x)^n,q^{p(2n+p)s+p(n+p)/2}(-x)^{n+p}/(-y)^n;q^{p^2(2n+p)}) }.\label{Phi-defn}
\end{align}
Here, $r:=r^{*}+\{(n-1)/2 \}$ and $s:=s^{*}+\{(n-1)/2\}$, with $0\leq \{\alpha\}<1$ denoting the fractional part of $\alpha$.
\end{theorem}

\begin{theorem}\label{thm-fg-2}
(Cf.\ \cite[Theorem 1.9]{Hickerson-Mortenson}.) Let $n$ be a positive odd integer. For generic $x,y\in \mathbb{C}^{*}$
\begin{align}
f_{n,n+2,n}(x,y,q)=g_{n,n+2,n}(x,y,q,y^n/x^n,x^n/y^n)-\Theta_{n,2}(x,y,q),
\end{align}
where
\begin{align}
\Theta_{n,2}(x,y,q):=\frac{y^{(n+1)/2}J_{2n,4n}J_{4(n+1),8(n+1)}j(y/x,q^{n+2}xy;q^{4(n+1)})j(q^{2n}/x^2y^2;q^{8(n+1)})  }{q^{(n^2-3)/2}x^{(n-3)/2}j(y^n/x^n;q^{4n(n+1)})j(-q^{n+2}x^2,-q^{n+2}y^2;q^{4(n+1)})  }. \label{Theta-defn}
\end{align}
\end{theorem}

Using the above two theorems, we will convert those Hecke-type series expressible by $f_{a,b,c}(x,y,q)$ appeared in this paper to a form in terms of $m(x,q,z)$. Again, since this process is routine, we will omit the details unless there are some special treatments.

\section{Some $(a,b)$-parameterized identities}\label{sec:proof}
Transformation formulas stated in Theorem \ref{thm-key} can be regarded as a consequence of the following result of Liu.
\begin{lemma}\label{Liu-An-lem}
(Cf.\ \cite[Theorem 4.1]{Liu2013IJNT}.) If $A_n$ is a complex sequence, then, under suitable convergence conditions, we have
\begin{align}\label{liu1}
&\frac{(\alpha q,\alpha ab/q;q)_\infty}{(\alpha a,\alpha b;q)_\infty}\sum_{n=0}^\infty A_n(q/a;q)_n(\alpha a)^n\\
\nonumber
=&\sum_{n=0}^\infty\frac{(1-\alpha q^{2n})(\alpha,q/a,q/b;q)_n(-\alpha ab/q)^nq^{n(n-1)/2}}{(1-\alpha)(q,\alpha a,\alpha b;q)_n}\\
\nonumber
&\times\sum_{k=0}^n\frac{(q^{-n},\alpha q^n;q)_k(q^2/b)^k}{(q/b;q)_k}A_k.
\end{align}
\end{lemma}

Now we give proofs to Theorems \ref{thm-main} and \ref{thm-key}.
\begin{proof}[Proof of Theorem \ref{thm-key}]
The desired identity follows by setting
\begin{align}
A_{n}=\frac{(q/b, \alpha b_1,\dots, \alpha b_m;q)_{n}}{(q,\alpha c_1, \dots, \alpha c_m;q)_{n}}(\frac{bz}{q})^n
\end{align}
in Lemma \ref{Liu-An-lem}.
\end{proof}
\begin{proof}[Proof of Theorem \ref{thm-main}]
If $u=0$, the desired identity is clearly true. Now we suppose $u\ne 0$. The theorem follows from Theorem \ref{thm-key} after taking $(m, \alpha, b_1, b_2, c_1, c_2, z)\rightarrow (2, u, v/u, q/u, c/u, d/u, 1)$.
\end{proof}

Using Theorems \ref{thm-main} and \ref{thm-key}, we are able to establish numerous identities with two parameters $a$ and $b$. These parameterized identities will generate interesting Appell-Lerch series or Hecke-type series after specializing the choices of $a$ and $b$. We now illustrate several examples.

\begin{theorem}\label{meq:1.1}
For $\max\{|ab|, |aq^2|, |bq^2|\}<1$, we have
\begin{align}
&\frac{(q^2,ab;q^2)_\infty}{(aq^2,bq^2;q^2)_\infty}{}_3\phi_2\bigg(\genfrac{}{}{0pt}{}{q^2/a,q^2/b,q}{-q^2,q^3};q^2,-ab\bigg) \nonumber\\
=&(1-q)\sum_{n=0}^\infty\sum_{j=-n}^n(-1)^n(1+q^{2n+1})\frac{(q^2/a,q^2/b;q^2)_n}{(aq^2,bq^2;q^2)_n}(ab)^nq^{n^2-j^2}. \label{thm-meq:1.1-eq}
\end{align}
\end{theorem}
\begin{proof}
Taking $(q, m,\alpha,b_1,b_2,c_1,c_2,z) \rightarrow (q^2, 2, q^2,q^{-1},0,-1,q,-1)$ in Theorem \ref{thm-key}, we obtain
\begin{align}
&\frac{(q^2,ab;q^2)_\infty}{(aq^2,bq^2;q^2)_\infty}{}_3\phi_2\bigg(\genfrac{}{}{0pt}{}{q^2/a,q^2/b,q}{-q^2,q^3};q^2,-ab\bigg) \nonumber \\
=&\sum_{n=0}^\infty(-1)^n(1-q^{4n+2})\frac{(q^2/a,q^2/b;q^2)_n}{(aq^2,bq^2;q^2)_n}(ab)^nq^{n(n-1)}
{}_3\phi_2\bigg(\genfrac{}{}{0pt}{}{q^{-2n},q^{2n+2},q}{-q^2,q^3};q^2,-q^2\bigg). \label{proof-eq1}
\end{align}
Taking $(q, \alpha,c,d) \rightarrow (q^2, 1,q,-q^2)$ in Lemma \ref{meq:1}, we obtain
\begin{align}
{}_3\phi_2\bigg(\genfrac{}{}{0pt}{}{q^{-2n},q^{2n+2},q}{-q^2,q^3};q^2,-q^2\bigg)
=\frac{(1-q)q^n}{1-q^{2n+1}}\sum_{j=-n}^nq^{-j^2}. \label{proof-eq2}
\end{align}
Substituting \eqref{proof-eq2} into \eqref{proof-eq1}, we arrive at \eqref{thm-meq:1.1-eq}.
\end{proof}
\begin{corollary}
We have
\begin{align}
&\sum_{n=0}^\infty(-1)^n\frac{q^{2n(n+1)}}{(q^4;q^4)_n(1-q^{2n+1})} \nonumber \\
=& \frac{1}{(q^2;q^2)_\infty}\sum_{n=0}^\infty\sum_{j=-n}^n(-1)^n(1+q^{2n+1})q^{3n^2+2n-j^2} \label{T1} \\
=& \frac{1}{J_2}\left(f_{1,2,1}(q^4,q^4,q^4)-q^5f_{1,2,1}(q^{10},q^{10},q^4) \right) \label{T1-H} \\
=&\frac{1}{J_2}f_{1,2,1}(q^{3/2},-q^{3/2},-q) \label{T1-f-exp} \\
=&q^{-1/2}\frac{J_2}{J_4}\left(m(-q^{1/2},-q^3,-1)-m(q^{1/2},-q^3,-1) \right), \label{T1-A} \\
&\sum_{n=0}^\infty\frac{q^{n(n+1)}}{(-q^2;q^2)_{n}(1-q^{2n+1})} \nonumber \\
=&\sum_{n=0}^\infty\sum_{j=-n}^n(1+q^{2n+1})q^{2n^2+n-j^2} \label{T2} \\
=&\overline{f}_{1,3,1}(-q^2,-q^2,q^2)+q^3\overline{f}_{1,3,1}(-q^6,-q^6,q^2). \label{T2-barf-exp}
\end{align}
\end{corollary}
\begin{proof}
Taking $(a,b)\rightarrow (0,0)$ and $(0,1)$ in Theorem \ref{meq:1.1}, we obtain \eqref{T1} and \eqref{T2}, respectively.

As mentioned in Section \ref{subsec-build}, deducing \eqref{T1-H} from \eqref{T1} is a routine exercise. To deduce \eqref{T1-f-exp} from \eqref{T1-H}, we set $(a,b,c,x,y,q)\rightarrow (1,2,1,q^{3/2},-q^{3/2},-q)$ in \eqref{f-id-1}. To get \eqref{T1-A} from \eqref{T1-f-exp}, we apply Theorem \ref{thm-fg-1} with $(n,p)=(1,1)$. The proof of \eqref{T2-barf-exp} is again routine and omitted.
\end{proof}

\begin{theorem}\label{meq:1.2}
For $\max\{|ab/q|, |aq^2|, |bq^2|\}<1$, we have
\begin{align}
&\frac{(q^2,ab;q^2)_\infty}{(aq^2,bq^2;q^2)_\infty}{}_3\phi_2\bigg(\genfrac{}{}{0pt}{}{q^2/a,q^2/b,-1}{q,-q^2};q^2,ab/q\bigg) \nonumber\\
=&\sum_{n=0}^\infty\sum_{j=-n}^n(1-q^{4n+2})\frac{(q^2/a,q^2/b;q^2)_n}{(aq^2,bq^2;q^2)_n}(ab)^nq^{n^2-n-j^2}. \label{thm-meq:1.2-eq}
\end{align}
\end{theorem}
\begin{proof}
Taking $(q,m,\alpha,b_1,b_2,c_1,c_2,z)\rightarrow(q^2,2,q^2,-q^{-2},0,q^{-1},-1,q^{-1})$ in Theorem \ref{thm-key},
we obtain
\begin{align}
&\frac{(q^2,ab;q^2)_\infty}{(aq^2,bq^2;q^2)_\infty}{}_3\phi_2\bigg(\genfrac{}{}{0pt}{}{q^2/a,q^2/b,-1}{q,-q^2};q^2, ab/q\bigg) \nonumber\\
=&\sum_{n=0}^\infty(-1)^n(1-q^{4n+2})\frac{(q^2/a,q^2/b;q^2)_n}{(aq^2,bq^2;q^2)_n}(ab)^nq^{n(n-1)}
{}_3\phi_2\bigg(\genfrac{}{}{0pt}{}{q^{-2n},q^{2n+2},-1}{q,-q^2};q^2,q\bigg). \label{revise-meq1.2-add-1}
\end{align}
Taking $(q, \alpha,c,d)\rightarrow(q^2, 1,-q^2,q)$ in Lemma \ref{meq:1}, we obtain
\begin{align}
{}_3\phi_2\bigg(\genfrac{}{}{0pt}{}{q^{-2n},q^{2n+2},-1}{q,-q^2};q^2,q\bigg)
=(-1)^n\sum_{j=-n}^nq^{-j^2}. \label{revise-meq1.2-add-2}
\end{align}
Substituting \eqref{revise-meq1.2-add-2} into \eqref{revise-meq1.2-add-1}, we arrive at \eqref{thm-meq:1.2-eq}.
\end{proof}
\begin{corollary}
We have
\begin{align}
&1+2\sum_{n=1}^\infty\frac{q^{2n^2+n}}{(q;q)_{2n}(1+q^{2n})} \nonumber \\
=& \frac{1}{(q^2;q^2)_{\infty}}\sum_{n=0}^\infty\sum_{j=-n}^n(1-q^{4n+2})q^{3n^2+n-j^2} \label{T3} \\
=& \frac{1}{J_2}\left(f_{1,2,1}(-q^3,-q^3,q^4)+q^4f_{1,2,1}(-q^9,-q^9,q^4) \right) \label{T3-f-exp} \\
=& 2\frac{J_2}{J_1}\left(m(-q^5,q^{12},-1)-q^{-1}m(-q,q^{12},-1) \right)+q^{-1}\frac{J_{2}^3J_{6}^2J_{8}^2}{J_{1}^2J_{4}^2J_{24}^2}, \label{T3-A} \\
&1+2\sum_{n=1}^\infty\frac{(-1)^nq^{n^2}}{(q;q^2)_{n}(1+q^{2n})} \nonumber \\
=& \sum_{n=0}^\infty\sum_{j=-n}^n(-1)^n(1-q^{4n+2})q^{2n^2-j^2} \label{T4} \\
=& \overline{f}_{1,3,1}(q,q,q^2)-q^2\overline{f}_{1,3,1}(q^5,q^5,q^2). \label{T4-barf-exp}
\end{align}
\end{corollary}
\begin{proof}
Taking $(a,b)\rightarrow (0,0)$ and $(0,1)$ in Theorem \ref{meq:1.2}, we obtain \eqref{T3} and \eqref{T4}, respectively.

To get \eqref{T3-A} from \eqref{T3-f-exp}, we apply Theorem \ref{thm-fg-1} with $(n,p)=(1,1)$.
\end{proof}

\begin{theorem}\label{meq:3.1}
For  $\max \{|ab|, |aq|,|bq|\}<1$, we have
\begin{align}
&\frac{(q,ab;q)_\infty}{(aq,bq;q)_\infty}{}_3\phi_2\bigg(\genfrac{}{}{0pt}{}{q/a,q/b,-1}{0,-q};q,ab\bigg)\\
\nonumber
=&\sum_{n=0}^\infty\sum_{j=-n}^n(-1)^{j}(1-q^{2n+1})\frac{(q/a,q/b;q)_n}{(aq,bq;q)_n}(ab)^nq^{n(n-1)/2+j^2}.
\end{align}
\end{theorem}
\begin{proof}
Taking $(c,d,u,v) \rightarrow (0,-q,q,-1)$ in Theorem \ref{thm-main}, we deduce that
\begin{align}
&\frac{(q,ab;q)_\infty}{(aq,bq;q)_\infty}{}_3\phi_2\bigg(\genfrac{}{}{0pt}{}{q/a,q/b,-1}{0,-q};q,ab\bigg) \nonumber\\
=&\sum_{n=0}^\infty(1-q^{2n+1})\frac{(q/a,q/b;q)_n}{(aq,bq;q)_n}(-ab)^nq^{n(n-1)/2}
{}_3\phi_2\bigg(\genfrac{}{}{0pt}{}{q^{-n},q^{n+1},-1}{0,-q};q,q\bigg). \label{proof-add-1}
\end{align}
Setting $(\alpha, c)=(1,-q)$ in  Lemma \ref{meq:3}, we obtain
\begin{align}
{}_3\phi_2\bigg(\genfrac{}{}{0pt}{}{q^{-n}, q^{n+1},-1}{0,-q};q,q\bigg)
=(-1)^n\sum_{j=-n}^n(-1)^jq^{j^2}. \label{proof-add-2}
\end{align}
The theorem follows after substituting \eqref{proof-add-2} into \eqref{proof-add-1}.
\end{proof}
\begin{corollary}
We have
\begin{align}
&\sum_{n=0}^\infty\frac{q^{n^2+n}}{(q;q)_{n}(1+q^{n})} \nonumber \\
=& \frac{1}{2(q;q)_{\infty}}\sum_{n=0}^\infty\sum_{j=-n}^n(-1)^j(1-q^{2n+1})q^{3n^2/2+n/2+j^2} \label{T9} \\
=&\frac{1}{2J_1}\left(f_{5,1,5}(q^3,q^3,q)+q^2f_{5,1,5}(q^6,q^6,q)\right), \label{T9-f-exp} \\
&\sum_{n=0}^\infty(-1)^n\frac{q^{n(n+1)/2}}{1+q^n} \nonumber \\
=& \frac{1}{2}\sum_{n=0}^\infty\sum_{j=-n}^n(-1)^{n+j}(1-q^{2n+1})q^{n^2+j^2} \label{T10} \\
=& \frac{1}{2}\left(\overline{f}_{1,0,1}(-q^2,-q^2,q^4)-q\overline{f}_{1,0,1}(-q^4,-q^4,q^4)\right) \label{T10-barf-exp} \\
=& \frac{1}{4}\left(\frac{J_{1}^4}{J_{2}^2} +1\right). \label{T10-simplify}
\end{align}
\end{corollary}
\begin{proof}
Taking $(a,b)\rightarrow (0,0)$ and $(0,1)$ in Theorem \ref{meq:3.1}, we obtain \eqref{T9} and \eqref{T10}, respectively.

Note that
\begin{align}
\sum_{n=0}^\infty(-1)^n\frac{q^{n(n+1)/2}}{1+q^n}=\frac{1}{2}\sum_{n=-\infty}^\infty(-1)^n\frac{q^{n(n+1)/2}}{1+q^n}+\frac{1}{4}. \label{reciprocal-apply-2}
\end{align}
Setting $z=-1$ in \eqref{reciprocal-Jacobi} we get \eqref{T10-simplify} immediately.
\end{proof}
\begin{rem}
The Hecke-type series in \eqref{T10} looks quite similar to the following identity of Liu \cite[Eq.\ (4.12)]{Liu2013IJNT}:
\begin{align}\label{Liu-Eq412}
\sum_{n=0}^{\infty}\frac{q^{n^2}}{(q^2;q^2)_{n}}=\frac{1}{(q;q)_{\infty}}\sum_{n=0}^\infty\sum_{j=-n}^n(-1)^{j}(1-q^{2n+1})q^{n^2+j^2}.
\end{align}
Note that there is typo in \cite[Eq.\ (4.12)]{Liu2013IJNT} that $(-1)^{n+j}$ should be $(-1)^j$. We can actually rewrite the above identity as
\begin{align}
\sum_{n=0}^{\infty}\frac{q^{n^2}}{(q^2;q^2)_{n}}=\frac{1}{J_1}\left(f_{1,0,1}(q^2,q^2,q^4)+qf_{1,0,1}(q^4,q^4,q^4) \right)=\frac{J_{2}^2}{J_{1}J_4}. \label{Liu-Eq412-simplify}
\end{align}
Here the first equality follows by \eqref{Liu-Eq412} and routine process. For the second equality, we note that
\begin{align}
f_{1,0,1}(q^2,q^2,q^4)=& \sum_{r,s\geq 0}(-1)^{r+s}q^{2r^2+2s^2}-\sum_{r,s<0}(-1)^{r+s}q^{2r^2+2s^2} \nonumber \\
=& \left(\sum_{r=0,s>0}+\sum_{r>0,s=0}+\sum_{r=s=0} \right)(-1)^{r+s}q^{2r^2+2s^2} \nonumber \\
=&1+2\sum_{r=1}^\infty (-1)^rq^{2r^2} \nonumber \\
=&\frac{J_2^2}{J_4}.
\end{align}
Similarly,
\begin{align}
f_{1,0,1}(q^4,q^4,q^4)=\sum_{r,s\geq 0}(-1)^{r+s}q^{2r^2+2r+2s^2+2s}-\sum_{r,s<0}(-1)^{r+s}q^{2r^2+2r+2s^2+2s}=0.
\end{align}
Here the last equality follows by replacing $(r,s)$ by $(-r-1,-s-1)$ in the second sum. Therefore \eqref{Liu-Eq412-simplify} holds.
\end{rem}
\begin{rem}
At this moment we are not able to  convert $f_{5,1,5}(x,y,q)$ to an expression in terms of $m(x,q,z)$, which might be possible. See \eqref{T11-H} for another Hecke-type series of this kind.
\end{rem}

\begin{theorem}\label{meq:3.2}
For  $\max \{|ab|, |aq^2|, |bq^2|\}<1$, we have
\begin{align}
&\frac{(q^2,ab;q^2)_\infty}{(aq^2,bq^2;q^2)_\infty}{}_3\phi_2\bigg(\genfrac{}{}{0pt}{}{q^2/a,q^2/b,q}{0,q^3};q^2,ab\bigg)\\
\nonumber
=&(1-q)\sum_{n=0}^\infty\sum_{j=-n}^n(-1)^n(1+q^{2n+1})\frac{(q^2/a,q^2/b;q^2)_n}{(aq^2,bq^2;q^2)_n}(ab)^nq^{n^2+2j^2-j}.
\end{align}
\end{theorem}
\begin{proof}
Taking $(q,c,d,u,v)\rightarrow (q^2, 0, q^3, q^2, q)$ in Theorem \ref{thm-main}, we deduce that
\begin{align}
&\frac{(q^2,ab;q^2)_\infty}{(aq^2,bq^2;q^2)_\infty}{}_3\phi_2\bigg(\genfrac{}{}{0pt}{}{q^2/a,q^2/b,q}{0,q^3};q^2,ab\bigg) \nonumber \\
=&\sum_{n=0}^\infty(1-q^{4n+2})\frac{(q^2/a,q^2/b;q^2)_n}{(aq^2,bq^2;q^2)_n}(-ab)^nq^{n(n-1)}
{}_3\phi_2\bigg(\genfrac{}{}{0pt}{}{q^{-2n}, q^{2n+2},q}{0,q^3};q^2,q^2\bigg). \label{proof-new-2}
\end{align}
Taking $(q, \alpha, c) \rightarrow (q^2, 1, q)$ in Lemma \ref{meq:3}, we obtain
\begin{align}
{}_3\phi_2\bigg(\genfrac{}{}{0pt}{}{q^{-2n}, q^{2n+2},q}{0,q^3};q^2,q^2\bigg)
=\frac{(1-q)q^n}{1-q^{2n+1}}\sum_{j=-n}^nq^{2j^2-j}. \label{proof-new-1}
\end{align}
The theorem follows after substituting \eqref{proof-new-1} into \eqref{proof-new-2}.
\end{proof}
\begin{corollary}
We have
\begin{align}
&\sum_{n=0}^\infty\frac{q^{2n^2+2n}}{(1-q^{2n+1})(q^2;q^2)_{n}} \nonumber \\
=&\frac{1}{(q^2;q^2)_{\infty}}\sum_{n=0}^\infty\sum_{j=-n}^n(-1)^n(1+q^{2n+1})q^{3n^2+2n+2j^2-j} \label{T11} \\
=& \frac{1}{J_2}\left(f_{5,1,5}(q^6,q^8,q^2)-q^5f_{5,1,5}(q^{12},q^{14},q^2) \right), \label{T11-H} \\
&\sum_{n=0}^\infty(-1)^n\frac{q^{n(n+1)}}{1-q^{2n+1}} \nonumber \\
=&\sum_{n=0}^\infty\sum_{j=-n}^n(1+q^{2n+1})q^{2n^2+n+2j^2-j} \label{T12} \\
=&\overline{f}_{1,0,1}(-q^4,-q^6,q^8)+q^3\overline{f}_{1,0,1}(-q^8,-q^{10},q^8). \label{T12-bar-H}
\end{align}
\end{corollary}
\begin{proof}
Taking $(a,b)\rightarrow (0,0)$ and $(0,1)$ in Theorem \ref{meq:3.2}, we obtain identities \eqref{T11} and \eqref{T12}, respectively.
\end{proof}

\begin{theorem}\label{meq:3.3}
For  $\max\{|ab|, |aq^2|, |bq^2|\}<1$, we have
\begin{align}
&\frac{(q^2,ab;q^2)_\infty}{(aq^2,bq^2;q^2)_\infty}{}_3\phi_2\bigg(\genfrac{}{}{0pt}{}{q^2/a,q^2/b,q^2}{-q^2,-q^3};q^2,ab\bigg)\\
\nonumber
=&(1+q)\sum_{n=0}^\infty\sum_{j=-n}^n(-1)^j(1-q^{2n+1})\frac{(q^2/a,q^2/b;q^2)_n}{(aq^2,bq^2;q^2)_n}(ab)^nq^{2n^2+n-j^2}.
\end{align}
\end{theorem}
\begin{proof}
Taking $(q, c, d, u, v)\rightarrow (q^2, -q^2, -q^3, q^2, q^2)$ in Theorem \ref{thm-main}, we deduce that
\begin{align}\label{eq:w1}
&\frac{(q^2,ab;q^2)_\infty}{(aq^2,bq^2;q^2)}{}_3\phi_2\bigg(\genfrac{}{}{0pt}{}{q^2/a,q^2/b,q^2}{-q^2,-q^3};q^2,ab\bigg) \nonumber \\
=&\sum_{n=0}^{\infty}(1-q^{4n+2})\frac{(q^2/a,q^2/b;q^2)_n}{(aq^2,bq^2;q^2)_n}(-ab)^nq^{n(n-1)}{}_3\phi_2\bigg(\genfrac{}{}{0pt}{}{q^{-2n},q^{2n+2},q^2}{-q^2,-q^3};q^2,q^2\bigg).
\end{align}
From \cite[Eq.\ (6.11)]{A12} we find
\begin{align}
\frac{1+q^{2n+1}}{1+q}{}_3\phi_2\bigg(\genfrac{}{}{0pt}{}{q^{2n},q^{2n+2},q^2}{-q^2,-q^3};q^2,q^2\bigg)
=(-1)^nq^{n^2+2n}\sum_{j=-n}^{n}(-1)^jq^{-j^2}. \label{proof-add-andrews}
\end{align}
The theorem follows after substituting \eqref{proof-add-andrews} into \eqref{eq:w1}.
\end{proof}
\begin{corollary}
We have
\begin{align}
&\sum_{n=0}^{\infty}\frac{q^{2n^2+2n}}{(-q;q)_{2n+1}}\nonumber \\
=&\frac{1}{(q^2;q^2)_{\infty}}\sum_{n=0}^{\infty}\sum_{j=-n}^{n}(-1)^j(1-q^{2n+1})q^{4n^2+3n-j^2} \label{T13} \\
=&\frac{1}{J_2}\left(f_{3,5,3}(q^6,q^6,q^2)+q^7f_{3,5,3}(q^{14},q^{14},q^2) \right) \label{T13-H-pre} \\
=&\frac{1}{J_2}f_{3,5,3}(q^{9/4},-q^{9/4},-q^{1/2}) \label{T13-H} \\
=&2q^{-1}m(-q^7,q^{24},-1)+2q^{-3}m(-q,q^{24},-1)+\frac{1}{J_2}\Phi_{3,2}(q^{9/4},-q^{9/4},-q^{1/2}), \label{T13-A}\\
&\sum_{n=0}^{\infty}(-1)^n\frac{(q^2;q^2)_{n}}{(-q;q)_{2n+1}}q^{n^2+n}\nonumber \\
=&\sum_{n=0}^{\infty}\sum_{j=-n}^{n}(-1)^{n+j}(1-q^{2n+1})q^{3n^2+2n-j^2} \label{T14} \\
=&\overline{f}_{1,2,1}(-q^4,-q^4,q^4)-q^5\overline{f}_{1,2,1}(-q^{10},-q^{10},q^4) \label{T14-barf-pre} \\
=&\overline{f}_{1,2,1}(q^{3/2},-q^{3/2},q). \label{T14-barf}
\end{align}
\end{corollary}
\begin{proof}
Taking $(a,b)\rightarrow (0,0)$ and $(0,1)$ in Theorem \ref{meq:3.3}, we obtain \eqref{T13} and \eqref{T14}, respectively.

Using \eqref{f-id-1} with $(a,b,c,x,y,q)\rightarrow (3,5,3,q^{9/4},-q^{9/4},-q^{1/2})$, we get  \eqref{T13-H} from \eqref{T13-H-pre}. Using \eqref{barf-id-1} with $(a,b,c,x,y,q)\rightarrow (1,2,1,q^{3/2},-q^{3/2},q)$, we get \eqref{T14-barf} from \eqref{T14-barf-pre}.

The expression \eqref{T13-A} follows from \eqref{T13-H} after applying Theorem \ref{thm-fg-1} with $(n,p)=(3,2)$ and simplifications.
\end{proof}
\begin{rem}
The identity \eqref{T13} was found by Andrews \cite[Eq.\ (1.15)]{A12} and also reproved by Mortenson \cite[p.\ 303]{Mortenson-2013}. Mortenson \cite[p.\ 303]{Mortenson-2013} also gave the expression \eqref{T13-H}.
\end{rem}

\begin{theorem}\label{T:4.8}
For $\max\{|ab|, |aq^2|, |bq^2|\}<1$, we have
\begin{align}
&\frac{(q^2,ab;q^2)_\infty}{(aq^2,bq^2;q^2)_\infty}{}_3\phi_2\bigg(\genfrac{}{}{0pt}{}{q^2/a,q^2/b,-q}{q,-q^2};q^2,ab\bigg)\\
\nonumber
=&\sum_{n=0}^\infty\sum_{j=-n}^n(1-q^{4n+2})\frac{(q^2/a,q^2/b;q^2)_n}{(aq^2,bq^2;q^2)_n}(ab)^nq^{2n^2-j^2}.
\end{align}
\end{theorem}
\begin{proof}
Taking $(q,c,d,u,v)\rightarrow (q^2, q, -q^2, q^2, -q)$ in Theorem \ref{thm-main}, we deduce that
\begin{align}
&\frac{(q^2,ab;q^2)_\infty}{(aq^2,bq^2;q^2)_\infty}{}_3\phi_2\bigg(\genfrac{}{}{0pt}{}{q^2/a,q^2/b,-q}{q,-q^2};q^2,ab\bigg) \nonumber\\
=&\sum_{n=0}^\infty (1-q^{4n+2})\frac{(q^2/a,q^2/b;q^2)_n}{(aq^2,bq^2;q^2)_n}(-ab)^nq^{n^2-n}{}_3\phi_2\bigg(\genfrac{}{}{0pt}{}{q^{-2n},q^{2n+2},-q}{q,-q^2};q^2,q^2\bigg). \label{T4.8-proof}
\end{align}
Taking $(q, \alpha, c, d)\rightarrow (q^2, 1, q, -q^2)$ in \eqref{Liu-eq-313}, we obtain
\begin{align}
{}_3\phi_2\bigg(\genfrac{}{}{0pt}{}{q^{-2n},q^{2n+2},-q}{q,-q^2};q^2,q^2\bigg)=(-1)^nq^{n^2+n}\sum_{j=-n}^{n}q^{-j^2}. \label{T4.8-id}
\end{align}
The theorem follows after substituting \eqref{T4.8-id} into \eqref{T4.8-proof}.
\end{proof}
\begin{corollary}
We have
\begin{align}
&\sum_{n=0}^{\infty}\frac{(-q;q^2)_nq^{2n^2+2n}}{(q,q^2)_n(q^4;q^4)_n}\nonumber \\
=&\frac{1}{(q^2;q^2)_\infty}\sum_{n=0}^{\infty}\sum_{j=-n}^{n}(1-q^{4n+2})q^{4n^2+2n-j^2} \label{T4.8-cor-1} \\
=&\frac{1}{J_2}\left(f_{3,5,3}(-q^5,-q^5,q^2)+q^6f_{3,5,3}(-q^{13},-q^{13},q^2) \right) \label{T4.8-cor-1-H} \\
=&2\frac{\overline{J}_{1,6}}{J_2}\left(m(-q^{44},q^{96},-1)+q^{-2}m(-q^{28},q^{96},-1)-q^{-4}m(-q^{20},q^{96},-1)\right.\nonumber \\
& \left. -q^{-10}m(-q^4,q^{96},-1) \right) -2\frac{\overline{J}_{3,6}}{J_2}\left(q^{-1}m(-q^{36},q^{96},-1)-q^{-7}m(-q^{12},q^{96},-1) \right)\nonumber \\
& +\frac{1}{J_2}\left(\Phi_{3,2}(-q^5,-q^5,q^2)+q^6\Phi_{3,2}(-q^{13},-q^{13},q^2) \right), \label{T4.8-cor-1-A}
\end{align}
\begin{align}
&\sum_{n=0}^{\infty}\frac{(q;q^2)_n q^{2n^2+2n}}{(-q;q)_{2n}(q^2;q^2)_n} \nonumber \\
=&\frac{1}{(q^2;q^2)_\infty}\sum_{n=0}^{\infty}\sum_{j=-n}^{n}(1-q^{4n+2})(-1)^jq^{4n^2+2n-j^2} \label{T4.8-cor-2} \\
=&\frac{1}{J_2}\left(f_{3,5,3}(q^5,q^5,q^2)+q^6f_{3,5,3}(q^{13},q^{13},q^2) \right) \label{T4.8-cor-2-H-pre} \\
=&\frac{1}{J_2}f_{3,5,3}(q^{7/4},-q^{7/4},-q^{1/2}) \label{T4.8-cor-2-H} \\
=&2\frac{J_{1,6}}{J_2}\left(m(-q^{10},q^{24},-1)+q^{-2}m(-q^2,q^{24},-1)  \right)+2q^{-1}\frac{J_{3,6}}{J_2}m(-q^6,q^{24},-1)\nonumber \\
&\quad +\frac{1}{J_2}\Phi_{3,2}(q^{7/4},-q^{7/4},-q^{1/2}), \label{T4.8-cor-2-A} \\
&\sum_{n=0}^{\infty}\frac{(-1)^n(-q;q^2)_nq^{n^2+n}}{(q;q^2)_n(-q^2;q^2)_n}\nonumber \\
=&\sum_{n=0}^{\infty}\sum_{j=-n}^{n}(1-q^{4n+2})(-1)^nq^{3n^2+n-j^2} \label{T4.8-cor-3} \\
=&\overline{f}_{1,2,1}(q^3,q^3,q^4)-q^4\overline{f}_{1,2,1}(q^9,q^9,q^4), \label{T4.8-cor-3-barf-exp} \\
&\sum_{n=0}^{\infty}\frac{(-1)^n(q;q^2)_nq^{n^2+n}}{(-q;q)_{2n}}\nonumber \\
=&\sum_{n=0}^{\infty}\sum_{j=-n}^{n}(1-q^{4n+2})(-1)^{n+j}q^{3n^2+n-j^2} \label{T4.8-cor-4} \\
=&\overline{f}_{1,2,1}(-q^3,-q^3,q^4)-q^4\overline{f}_{1,2,1}(-q^9,-q^9,q^4). \label{T4.8-cor-4-barf-exp}
\end{align}
\end{corollary}
\begin{proof}
Taking $(a,b)\rightarrow (0,0)$ and $(0,1)$ in Theorem \ref{T:4.8}, we obtain \eqref{T4.8-cor-1} and \eqref{T4.8-cor-3}, respectively.

Replacing $q$ by $-q$ in \eqref{T4.8-cor-1} and \eqref{T4.8-cor-3}, we obtain \eqref{T4.8-cor-2} and \eqref{T4.8-cor-4}, respectively.

Using \eqref{f-id-1} with $(a,b,c,x,y,q)\rightarrow (3,5,3,q^{7/4},-q^{7/4},-q^{1/2})$, we get \eqref{T4.8-cor-2-H} from \eqref{T4.8-cor-2-H-pre}. The expression \eqref{T4.8-cor-1-A} (resp.\ \eqref{T4.8-cor-2-A}) follows from \eqref{T4.8-cor-1-H} (resp.\ \eqref{T4.8-cor-2-H}) after applying Theorem \ref{thm-fg-1} with $(n,p)=(3,2)$ and simplifications.
\end{proof}
\begin{rem}
The identity \eqref{T4.8-cor-2} was found by Andrews \cite[Eq.\ (1.16)]{A12} using a different method.
\end{rem}
\section{Mock theta functions of order 2}\label{sec-order-2}
In \cite[p.\ 120, Eq.\ (5.1)]{Gordon-McIntosh-Survey} and \cite{McIntosh-2018Rama}  three mock theta functions of order 2 were defined. They are
\begin{align}
&A^{(2)}(q):=\sum_{n=0}^{\infty}\frac{q^{n+1}(-q^2;q^2)_{n}}{(q;q^2)_{n+1}}=\sum_{n=0}^{\infty}\frac{q^{(n+1)^2}(-q;q^2)_{n}}{(q;q^2)_{n+1}^2}, \label{mock-2-1-defn} \\
&B^{(2)}(q):=\sum_{n=0}^{\infty}\frac{q^{n}(-q;q^2)_{n}}{(q;q^2)_{n+1}}=\sum_{n=0}^{\infty}\frac{q^{n^2+n}(-q^2;q^2)_{n}}{(q;q^2)_{n+1}^2}, \label{mock-2-2-defn} \\
&\mu^{(2)}(q):=\sum_{n=0}^{\infty}\frac{(-1)^n(q;q^2)_{n}q^{n^2}}{(-q^2;q^2)_{n}^2}. \label{mock-2-3-defn}
\end{align}
These functions  appeared in Ramanujan's Lost Notebook \cite{lostnotebook}. In particular, the function $B^{(2)}(q)$ arises in the modular transformation formulas for $A^{(2)}(q)$ and $\mu^{(2)}(q)$.

McIntosh \cite{McIntosh-2007Canad} provided the following Appell-Lerch series representations for these functions.
\begin{theorem}\label{thm-ord-2}
We have
\begin{align}
&A^{(2)}(q)=q\frac{(-q;q^2)_{\infty}}{(q^2;q^2)_{\infty}}\sum_{n=0}^{\infty}\frac{(-1)^nq^{2n^2+3n}}{1-q^{2n+1}}, \label{mock-2-1} \\
&B^{(2)}(q)=\frac{(-q^2;q^2)_{\infty}}{(q^2;q^2)_{\infty}}\sum_{n=-\infty}^{\infty}\frac{(-1)^nq^{2n^2+2n}}{1-q^{2n+1}}, \label{mock-2-2} \\
&\mu^{(2)}(q)=2\frac{(q;q^2)_{\infty}}{(q^2;q^2)_{\infty}}\sum_{n=-\infty}^{\infty}\frac{q^{2n^2+n}}{1+q^{2n}}. \label{mock-2-3}
\end{align}
\end{theorem}
Note that \eqref{mock-2-1} appeared  in \cite[Eq.\ (3.7)$_{R}$]{Andrews-Mordell} and \cite[p.\ 8]{lostnotebook}. Identity \eqref{mock-2-2} can also be found in \cite{lostnotebook} and \cite[Eq.\ (4.3)]{Andrews-Mordell}. Identity \eqref{mock-2-3} appeared on pages 8 and 29 in \cite{lostnotebook}.

Using Bailey pairs, Cui, Gu and Hao \cite{CGH} provided Hecke-type series representations for mock theta functions of order 2. They also express these Hecke-type series in terms of $f_{a,b,c}(x,y,q)$. See \cite[Theorem 1.3]{CGH}.

We will reprove the identities in Theorem \ref{thm-ord-2} by regarding them as corollaries of some general $(a,b)$-parameterized identities. As byproducts, we find new representations for $A^{(2)}(q)$ and $B^{(2)}(q)$ (see \eqref{cor-A} and \eqref{cor-B}).

\subsection{Representations for $A^{(2)}(q)$}

\begin{theorem}\label{thm-ab-2-1-1st}
For $\max\{|ab|, |aq^2|, |bq^2|\}<1$, we have
\begin{align}
&\frac{(q^2,ab;q^2)_\infty}{(aq^2,bq^2;q^2)_\infty}
{}_3\phi_2\bigg(\genfrac{}{}{0pt}{}{q^2/a,q^2/b,q^2}{q^3,q^3};q^2,ab\bigg) \nonumber\\
=& (1-q)^2\sum_{n=0}^{\infty}\frac{1+q^{2n+1}}{1-q^{2n+1}}\frac{(q^2/a,q^2/b;q^2)_n}{(aq^2,bq^2;q^2)_n}(-ab)^nq^{n^2+n}.
\end{align}
\end{theorem}
\begin{proof}
Taking $(q,c,d,u,v) \rightarrow (q^2,q^3,q^3,q^2,q^2)$ in Theorem \ref{thm-main}, we deduce that
\begin{align}
&\frac{(q^2,ab;q^2)_\infty}{(aq^2,bq^2;q^2)_\infty}
{}_3\phi_2\bigg(\genfrac{}{}{0pt}{}{q^2/a,q^2/b,q^2}{q^3,q^3};q^2,ab\bigg) \nonumber\\
=&\sum_{n=0}^{\infty}(1-q^{4n+2})\frac{(q^2/a,q^2/b;q^2)_n}{(aq^2,bq^2;q^2)_n}(-ab)^nq^{n^2-n}
{}_3\phi_2\bigg(\genfrac{}{}{0pt}{}{q^{-2n},q^{2n+2},q^2}{q^3,q^3};q^2,q^2\bigg). \label{ab-2-1-2nd-proof}
\end{align}
Taking $(q,a,b,c)\rightarrow (q^2, q^2,q,q)$ in \eqref{Pfaff}, we deduce that
\begin{align}\label{A-proof-3}
{}_3\phi_2\bigg(\genfrac{}{}{0pt}{}{q^{-2n}, q^{2n+2},q^2}{q^3,q^3};q^2,q^2\bigg)=\frac{(q,q;q^2)_{n}}{(q^3,q^3;q^2)_{n}}q^{2n}=\frac{(1-q)^2}{(1-q^{2n+1})^2}q^{2n}.
\end{align}
The theorem follows after substituting \eqref{A-proof-3} into \eqref{ab-2-1-2nd-proof}.
\end{proof}
\begin{corollary}
Identities \eqref{mock-2-1} and \eqref{mock-2-2} hold. In addition, we have
\begin{align}
&\sum_{n=0}^{\infty}\frac{q^{2n(n+1)}}{(q;q^2)_{n+1}^2}\nonumber \\
=&\frac{1}{(q^2;q^2)_\infty}\sum_{n=0}^{\infty}(-1)^n\frac{1+q^{2n+1}}{1-q^{2n+1}}q^{3n^2+3n} \label{2-1-cor-3} \\
=&-q^{-1}m(q,q^6,q^2)-q^{-1}m(q,q^6,q^4)\label{2-1-cor-3-m-exp-pre} \\
=&-2q^{-1}m(q,q^6,q^2)+\frac{J_6^3}{J_2J_{3,6}}, \label{2-1-cor-3-m-exp} \\
&\sum_{n=0}^{\infty}\frac{(-1)^nq^{n(n+1)}(q^2;q^2)_n}{(q;q^2)_{n+1}^2}\nonumber \\
=&\sum_{n=0}^{\infty}\frac{1+q^{2n+1}}{1-q^{2n+1}}q^{2n^2+2n} \label{2-1-cor-4} \\
=&\overline{m}(-q^2,q^4,-q^4)+q\overline{m}(-1,q^4,-q^6).  \label{2-1-cor-4-barm-exp}
\end{align}
\end{corollary}
\begin{proof}
We use the second expression of $A^{(2)}(q)$ to write it as
\begin{align}
A^{(2)}(q)=\frac{q}{(1-q)^2}\sum_{n=0}^{\infty}\frac{(-q;q^2)_{n}q^{n^2+2n}}{(q^3;q^2)_{n}^2}. \label{A-proof-add-1}
\end{align}
Taking $(a,b)\rightarrow (-q, 0)$ in Theorem \ref{thm-ab-2-1-1st}, we obtain \eqref{mock-2-1}.

We use the second expression of $B^{(2)}(q)$ in \eqref{mock-2-2-defn} to write it as
\begin{align}
B^{(2)}(q)=\frac{1}{(1-q)^2}\sum_{n=0}^{\infty}\frac{(-q^2;q^2)_{n}}{(q^3;q^2)_{n}^2}q^{n^2+n}.
\end{align}
Taking $(a,b)=(0,-1)$ in Theorem \ref{thm-ab-2-1-1st}, we obtain \eqref{mock-2-2}.

Taking $(a,b)\rightarrow (0,0)$ and $(0,1)$ in Theorem \ref{thm-ab-2-1-1st}, we obtain \eqref{2-1-cor-3} and \eqref{2-1-cor-4}, respectively.

Deducing \eqref{2-1-cor-3-m-exp-pre} and \eqref{2-1-cor-4-barm-exp} are routine exercises. In addition, to get \eqref{2-1-cor-3-m-exp} from \eqref{2-1-cor-3-m-exp-pre}, we apply Lemma \ref{lem-m-minus} with $(x,q,z_1,z_0)\rightarrow (q,q^6,q^2,q^4)$.
\end{proof}
\begin{rem}
The left side of \eqref{2-1-cor-3} is the third order mock theta function $\omega^{(3)}(q)$. See \eqref{mock-3-5-defn}. The expressions \eqref{2-1-cor-3-m-exp-pre} and \eqref{2-1-cor-3-m-exp} can be found in \cite[Eq.\ (5.8)]{Hickerson-Mortenson}.
\end{rem}

\begin{theorem}\label{thm-ab-2-1-2nd}
For $\max\{|abq|, |aq^2|, |bq^2|\}<1$, we have
\begin{align}
&\frac{(q^2,abq;q)_\infty}{(aq^2,bq^2;q)_\infty}{}_3\phi_2\bigg(\genfrac{}{}{0pt}{}{q/a, q/b,q}{q^{3/2},-q^{3/2}};q,abq\bigg) \nonumber \\
=& \sum_{n=0}^{\infty}(1-q^{2n+2})\frac{(q/a,q/b;q)_n}{(aq^2,bq^2;q)_n}(-ab)^nq^{n^2+2n}\sum_{j=0}^{n}q^{-j(j+1)/2}.
\end{align}
\end{theorem}
\begin{proof}
Taking $(c,d,u,v)\rightarrow (q^{3/2}, -q^{3/2}, q^2, q)$ in Theorem \ref{thm-main}, we deduce that
\begin{align}
&\frac{(q^2,abq;q)_\infty}{(aq^2,bq^2;q)_\infty}{}_3\phi_2\bigg(\genfrac{}{}{0pt}{}{q/a, q/b,q}{q^{3/2},-q^{3/2}};q,abq\bigg) \nonumber \\
=&\sum_{n=0}^{\infty}(1-q^{2n+2})\frac{1-q^{n+1}}{1-q}\frac{(q/a,q/b;q)_n}{(aq^2,bq^2;q)_n}(-abq^2)^nq^{(n^2-3n)/2}
{}_3\phi_2\bigg(\genfrac{}{}{0pt}{}{q^{-n}, q^{n+2},q}{q^{3/2},-q^{3/2}};q,q\bigg). \label{ab-2-1-proof}
\end{align}
Setting $(\alpha, \beta, c, d)=(q^2, q, q^{3/2}, -q^{3/2})$ in \eqref{sears:32}, we obtain
\begin{align}
{}_3\phi_2\bigg(\genfrac{}{}{0pt}{}{q^{-n}, q^{n+2},q}{q^{3/2},-q^{3/2}};q,q\bigg)=(-1)^nq^{(n^2+2n)/2}{}_3\phi_2\bigg(\genfrac{}{}{0pt}{}{q^{-n}, q^{n+2},-q^{1/2}}{q^{3/2},-q^{3/2}};q,q^{1/2}\bigg). \label{ab-2-1-proof-2}
\end{align}
Setting $(\alpha, c, d)=(q, -q^{1/2}, q^{1/2})$ in Lemma \ref{meq:1}, we deduce that
\begin{align}
{}_3\phi_2\bigg(\genfrac{}{}{0pt}{}{q^{-n}, q^{n+2},-q^{1/2}}{q^{3/2},-q^{3/2}};q,q^{1/2}\bigg)=(-1)^nq^{n/2}\frac{1-q}{1-q^{n+1}}\sum_{j=0}^{n}q^{-j(j+1)/2}. \label{ab-2-1-proof-3}
\end{align}
Substituting \eqref{ab-2-1-proof-3} into \eqref{ab-2-1-proof-2}, we obtain
\begin{align}
{}_3\phi_2\bigg(\genfrac{}{}{0pt}{}{q^{-n}, q^{n+2},q}{q^{3/2},-q^{3/2}};q,q\bigg)=q^{(n^2+3n)/2}\frac{1-q}{1-q^{n+1}}\sum_{j=0}^{n}q^{-j(j+1)/2}. \label{ab-2-1-proof-4}
\end{align}
The theorem follows after substituting \eqref{ab-2-1-proof-4} into \eqref{ab-2-1-proof}.
\end{proof}
\begin{corollary}\label{cor-2-1}
We have
\begin{align}
A^{(2)}(q)=&q\frac{(-q^2;q^2)_\infty}{(q;q)_{\infty}^2}\sum_{n=0}^{\infty}\frac{1-q^{2n+2}}{1+q^{2n+2}}(-1)^nq^{n^2+2n}\sum_{j=0}^{n}q^{-j(j+1)/2}, \label{cor-A}
\end{align}
\begin{align}
&\sum_{n=0}^{\infty}\frac{q^{n^2+2n}}{(q;q^2)_{n+1}}\nonumber \\
=&\frac{1}{(q;q)_\infty}\sum_{n=0}^{\infty}(-1)^n(1-q^{2n+2})q^{2n^2+3n}\sum_{j=0}^{n}q^{-j(j+1)/2} \label{2-1-cor-1} \\
=&\frac{1}{2J_1}\left(f_{3,5,3}(q^4,q^5,q)-q^5f_{3,5,3}(q^8,q^9,q)+j(q^4;q^3)  \right) \label{2-1-cor-1-H-pre} \\
=&\frac{1}{J_1}f_{3,5,3}(q^4,q^5,q) \label{2-1-cor-1-H} \\
=&-q^{-1}m(-q^{22},q^{48},-1)-q^{-2}m(-q^{14},q^{48},-1)-q^{-3}m(-q^{10},q^{48},-1)\nonumber \\
&-q^{-6}m(-q^2,q^{48},-1)+\frac{1}{J_1}\Phi_{3,2}(q^4,q^5,q) \label{2-1-cor-1-A} \\
=&-q^{-1}m(q,-q^3,q)+\frac{J_{12}^3}{J_{3,12}J_{4,12}}, \label{2-1-cor-1-final}
\end{align}
\begin{align}
&\sum_{n=0}^{\infty}\frac{(-1)^nq^{n(n+3)/2}(q;q)_n}{(q;q^2)_{n+1}}\nonumber \\
=&\sum_{n=0}^{\infty}(1+q^{n+1})q^{(3n^2+5n)/2}\sum_{j=0}^nq^{-j(j+1)/2} \label{2-1-cor-2} \\
=&\frac{1}{2}\left(\overline{f}_{1,2,1}(-q^3,-q^4,q^2)+q^4\overline{f}_{1,2,1}(-q^6,-q^7,q^2)-q^{-1} \right), \label{2-1-cor-2-barf-exp}
\end{align}
\begin{align}
&\sum_{n=0}^\infty \frac{q^{(n^2+3n)/2}(-q;q)_n}{(q;q^2)_{n+1}}\nonumber \\
=&\frac{(-q;q)_\infty}{(q;q)_\infty}\sum_{n=0}^\infty (-1)^n(1-q^{n+1})q^{(3n^2+5n)/2}\sum_{j=0}^n q^{-j(j+1)/2} \label{2-1-cor-add-sigma} \\
=&\frac{J_2}{2J_1^2}\left(f_{1,2,1}(q^3,q^4,q^2)-q^4f_{1,2,1}(q^6,q^7,q^2)+j(q^3;q^2)  \right) \label{2-1-cor-add-sigma-H-pre} \\
=&\frac{J_2}{J_1^2} f_{1,2,1}(q^3,q^4,q^2) \label{2-1-cor-add-sigma-H} \\
=& -q^{-1}m(q^2,q^6,-1)+\frac{1}{2}q^{-1}\frac{J_{2}^4J_{6}^5}{J_{1}^2J_{3}^2J_{4}^2J_{12}^2} \label{2-1-cor-add-sigma-A}\\
=&-q^{-1}m(q^2,q^6,q). \label{2-1-cor-add-sigma-A-simplify}
\end{align}
\end{corollary}
\begin{proof}
We use the first expression of $A^{(2)}(q)$ to write it as
\begin{align}
A^{(2)}(q)=\frac{q}{1-q}\sum_{n=0}^{\infty}\frac{(-q^2;q^2)_{n}q^n}{(q^3;q^2)_{n}}. \label{A-proof-1}
\end{align}
Taking $(a,b)\rightarrow (i,-i)$ in Theorem \ref{thm-ab-2-1-2nd}, we obtain \eqref{cor-A}.

Taking $(a,b)\rightarrow (0,0)$, $(0,1)$ and $(0,-1)$ in Theorem \ref{thm-ab-2-1-2nd}, we obtain \eqref{2-1-cor-1} and \eqref{2-1-cor-2}, respectively.

Setting $(a,b,c,x,y,q)\rightarrow (3,5,3,q^4,q^5,q)$ in \eqref{f-id-3}, we get \eqref{2-1-cor-1-H} from \eqref{2-1-cor-1-H-pre}. Similarly, setting $(a,b,c,x,y,q)\rightarrow (1,2,1,q^3,q^4,q^2)$ in \eqref{f-id-3}, we get \eqref{2-1-cor-add-sigma-H} from \eqref{2-1-cor-add-sigma-H-pre}.

The expression \eqref{2-1-cor-1-A} follows from \eqref{2-1-cor-1-H} and Theorem \ref{thm-fg-1} with $(n,p)=(3,2)$. The expression \eqref{2-1-cor-add-sigma-A} follows from \eqref{2-1-cor-add-sigma-H} and Theorem \ref{thm-fg-1} with $(n,p)=(1,1)$. The equivalence of \eqref{2-1-cor-add-sigma-A} and \eqref{2-1-cor-add-sigma-A-simplify} can be seen from \eqref{m-minus} with $(x,q,z_1,z_0)\rightarrow (q^2,q^6,-1,q)$.

Taking $(x,q,z,z',n)\rightarrow (q,-q^3,q,-1,4)$ in Lemma \ref{lem-m-decompose} and using the method in \cite{Garvan-Liang} to prove theta function identities, we obtain \eqref{2-1-cor-1-final} from \eqref{2-1-cor-1-A}.
\end{proof}
\begin{rem}\label{rem-mock-2-A}
(1) After multiplying by a factor $q$, the left side of \eqref{2-1-cor-1} becomes the third order mock theta function $\psi^{(3)}(q)$. See \eqref{mock-3-3-defn}. The expression \eqref{2-1-cor-1-final} can be found in \cite[Eq.\ (5.6)]{Hickerson-Mortenson} with a typo that $J_4$ should be $J_{4,12}$. \\
(2) After multiplying by a factor $q$, the left side of \eqref{2-1-cor-add-sigma} becomes the sixth order mock theta function $\sigma^{(6)}(q)$. See \eqref{mock-6-4-defn}. The expression \eqref{2-1-cor-add-sigma-A-simplify} appears in \cite[Eq.\ (5.26)]{Hickerson-Mortenson}. \\
(3) It is unclear whether \eqref{cor-A} can be converted to a Hecke-type series in terms of $f_{a,b,c}(x,y,q)$ or $\overline{f}_{a,b,c}(x,y,q)$. Moreover, from \cite[Eq.\ (5.1)]{Hickerson-Mortenson} we find that
\begin{align}
A^{(2)}(q)=-m(q,q^4,q^2). \label{cor-A-m-exp}
\end{align}
We are not sure whether this can be deduced from \eqref{cor-A} or not. A similar problem exists for \eqref{cor-B}.
\end{rem}

\subsection{Representations for $B^{(2)}(q)$}
Note that \eqref{mock-2-2} has already been proved in Corollary \ref{cor-2-1}. Now we establish a new representation for $B^{(2)}(q)$.
\begin{theorem}\label{thm-ab-2-2}
For $\max\{|ab|, |aq|, |bq|\}<1$, we have
\begin{align}
&\frac{(q,ab;q)_\infty}{(aq,bq;q)_\infty}
{}_3\phi_2\bigg(\genfrac{}{}{0pt}{}{q/a,q/b,q}{q^{3/2},-q^{3/2}};q,ab\bigg) \nonumber\\
=&~ (1-q)\sum_{n=0}^{\infty}\sum_{j=-n}^{n}\frac{(q/a,q/b;q)_n}{(aq,bq;q)_n}(-ab)^nq^{n^2+n-\binom{j+1}{2}}. \label{revise-thm-ab-2-2-eq}
\end{align}
\end{theorem}
\begin{proof}
Taking $(c,d,u,v) \rightarrow (q^{3/2}, -q^{3/2}, q,q)$ in Theorem \ref{thm-main}, we deduce that
\begin{align}
&\frac{(q,ab;q)_\infty}{(aq,bq;q)_\infty}
{}_3\phi_2\bigg(\genfrac{}{}{0pt}{}{q/a,q/b,q}{q^{3/2},-q^{3/2}};q,ab\bigg) \nonumber\\
=&\sum_{n=0}^{\infty}(1-q^{2n+1})\frac{(q/a,q/b;q)_n}{(aq,bq;q)_n}(-ab)^nq^{\frac{n^2-n}{2}}
{}_3\phi_2\bigg(\genfrac{}{}{0pt}{}{q^{-n},q^{n+1},q}{q^{3/2},-q^{3/2}};q,q\bigg). \label{ab-6-3-proof}
\end{align}
Setting $(\alpha, \beta, c, d)=(q,q,q^{3/2},-q^{3/2})$ in \eqref{sears:32}, we deduce that
\begin{align}
{}_{3}\phi_{2}\bigg(\genfrac{}{}{0pt}{}{q^{-n}, q^{n+1},q}{q^{3/2},-q^{3/2}};q,q\bigg)=(-1)^nq^{(n^2+2n)/2}\frac{1-q^{\frac{1}{2}}}{1-q^{n+\frac{1}{2}}}{}_{3}\phi_{2}\bigg(\genfrac{}{}{0pt}{}{q^{-n}, q^{n+1},-q^{1/2}}{q^{1/2},-q^{3/2}};q,q^{1/2}\bigg). \label{6-rho-proof-2}
\end{align}
Setting $(\alpha, c, d)=(1, -q^{1/2}, q^{1/2})$ in Lemma \ref{meq:1}, we obtain
\begin{align}
{}_{3}\phi_{2}\bigg(\genfrac{}{}{0pt}{}{q^{-n}, q^{n+1},-q^{1/2}}{q^{1/2},-q^{3/2}};q,q^{1/2}\bigg)=(-1)^nq^{n/2}\frac{1+q^{\frac{1}{2}}}{1+q^{n+\frac{1}{2}}} \sum_{j=-n}^{n}q^{-j(j+1)/2}. \label{6-rho-proof-3}
\end{align}
Substituting \eqref{6-rho-proof-3} into \eqref{6-rho-proof-2}, we obtain
\begin{align}
{}_{3}\phi_{2}\bigg(\genfrac{}{}{0pt}{}{q^{-n}, q^{n+1},q}{q^{3/2},-q^{3/2}};q,q\bigg)=q^{(n^2+3n)/2}\frac{1-q}{1-q^{2n+1}} \sum_{j=-n}^{n}q^{-j(j+1)/2}. \label{revise-6-rho-proof-4}
\end{align}
Substituting \eqref{revise-6-rho-proof-4} into \eqref{ab-6-3-proof}, we arrive at \eqref{revise-thm-ab-2-2-eq}.
\end{proof}
\begin{corollary}
We have
\begin{align}
B^{(2)}(q)=\frac{(-q;q^2)_\infty}{(q;q)_\infty^2}\sum_{n=0}^{\infty}\frac{(-1)^nq^{n^2+2n}}{1+q^{2n+1}}\sum_{j=-n}^{n}q^{-j(j+1)/2}. \label{cor-B}
\end{align}
\end{corollary}
\begin{proof}
We use the first expression of $B^{(2)}(q)$ in \eqref{mock-2-2-defn} to write it as
\begin{align}
B^{(2)}(q)=\frac{1}{1-q}\sum_{n=0}^{\infty}\frac{(iq^{1/2},-iq^{1/2};q)_{n}}{(q^{3/2},-q^{3/2};q)_{n}}q^n. \label{B-proof-1}
\end{align}
Taking $(a,b)\rightarrow (iq^{1/2}, -iq^{1/2})$ in Theorem \ref{thm-ab-2-2}, we obtain \eqref{cor-B}.
\end{proof}

\begin{corollary}
We have
\begin{align}
&\sum_{n=0}^{\infty}\frac{q^{n(n+1)}}{(q;q^2)_{n+1}}\nonumber \\
=&\frac{1}{(q;q)_\infty}\sum_{n=0}^{\infty}\sum_{j=-n}^{n}(-1)^nq^{2n^2+2n-j(j+1)/2} \label{2-2-cor-1} \\
=&\frac{1}{2J_1}\left(f_{3,5,3}(q^3,q^4,q)-q^4f_{3,5,3}(q^7,q^8,q)  \right) \label{2-2-cor-1-H-pre} \\
=&\frac{1}{J_1}f_{3,5,3}(q^3,q^4,q) \label{2-2-cor-1-H} \\
=&-q^{-1}m(-q^{16},q^{48},q^3)-q^{-1}m(-q^{16},q^{48},q^{-3})-q^{-3}m(-q^8,q^{48},q^{-3}) \nonumber \\
& -q^{-3}m(-q^8,q^{48},q^3) -\frac{1}{J_1}\Theta_{3,2}(q^3,q^4,q) \label{2-2-cor-1-A} \\
=&-2q^{-1}m(-q^{16},q^{48},q^3)-2q^{-3}m(-q^8,q^{48},q^3)+q^{-4}\frac{J_{48}^3J_{6,48}\overline{J}_{16,48}}{J_{-3,48}J_{3,48}\overline{J}_{13,48}\overline{J}_{19,48}} \nonumber \\
&+q^{-6}\frac{J_{48}^3J_{6,48}\overline{J}_{8,48}}{J_{-3,48}J_{3,48}\overline{J}_{5,48}\overline{J}_{11,48}}-\frac{1}{J_1}\Theta_{3,2}(q^3,q^4,q).\label{2-2-cor-1-A-simplify} \\
&\sum_{n=0}^{\infty}\frac{(-1)^nq^{n(n+1)/2}(q;q)_n}{(q;q^2)_{n+1}} \nonumber \\
=&\sum_{n=0}^\infty\sum_{j=-n}^nq^{(3n^2+3n)/2-j(j+1)/2} \label{2-2-cor-2} \\
=&\frac{1}{2}\left(\overline{f}_{1,2,1}(-q^2,-q^3,q^2)+q^3\overline{f}_{1,2,1}(-q^5,-q^6,q^2) \right), \label{2-2-cor-2-simplify} \\
&\sum_{n=0}^{\infty}\frac{q^{n(n+1)/2}(-q;q)_n} {(q;q^2)_{n+1}} \nonumber \\
=&\frac{(-q;q)_\infty}{(q;q)_\infty}\sum_{n=0}^{\infty}(-1)^nq^{(3n^2+3n)/2}\sum_{j=-n}^nq^{-j(j+1)/2} \label{2-2-cor-3-rho} \\
=&\frac{J_{2}}{2J_{1}^2}\left(f_{1,2,1}(q^2,q^3,q^2)-q^3f_{1,2,1}(q^5,q^6,q^2)  \right)  \label{2-2-cor-3-rho-H-pre} \\
=&\frac{J_{2}}{J_{1}^2}f_{1,2,1}(q^2,q^3,q^2) \label{2-2-cor-3-rho-H}\\
=&-q^{-1}m(1,q^6,-1)+\frac{1}{4}q^{-1}\frac{J_{2}^6J_{3}^4}{J_{1}^4J_{4}^2J_{6}J_{12}^2}. \label{2-2-cor-3-rho-A}
\end{align}
\end{corollary}
\begin{proof}
Taking $(a,b)\rightarrow (0,0)$, $(0,1)$  and $(0,-1)$ in Theorem \ref{thm-ab-2-2}, we obtain \eqref{2-2-cor-1}, \eqref{2-2-cor-2} and \eqref{2-2-cor-3-rho}, respectively.

Setting $(a,b,c,x,y,q)\rightarrow (3,5,3,q^3,q^4,q)$ in \eqref{f-id-2}, we get \eqref{2-2-cor-1-H} from \eqref{2-2-cor-1-H-pre}. Similarly, setting $(a,b,c,x,y,q)\rightarrow (1,2,1,q^2,q^3,q^2)$ in \eqref{f-id-2}, we get \eqref{2-2-cor-3-rho-H} from \eqref{2-2-cor-3-rho-H-pre}.

The expression \eqref{2-2-cor-1-A} follows from \eqref{2-2-cor-1-H} and Theorem \ref{thm-fg-2} with $n=3$. Similarly, \eqref{2-2-cor-3-rho-A} follows from \eqref{2-2-cor-3-rho-H} and Theorem \ref{thm-fg-1} with $(n,p)=(1,1)$.

Taking $(x,q,z_1,z_0)\rightarrow (-q^{16},q^{48},q^3,q^{-3})$ in Lemma \ref{lem-m-minus}, we obtain
\begin{align}
m(-q^{16},q^{48},q^3)-m(-q^{16},q^{48},q^{-3})=q^{-3}\frac{J_{48}^3J_{6,48}\overline{J}_{16,48}}{J_{-3,48}J_{3,48}\overline{J}_{13,48}\overline{J}_{19,48}}. \label{revise-Cor4.8-1}
\end{align}
Taking $(x,q,z_1,z_0)\rightarrow (-q^8,q^{48},q^3,q^{-3})$ in Lemma \ref{lem-m-minus}, we obtain
\begin{align}
m(-q^8,q^{48},q^3)-m(-q^8,q^{48},q^{-3})=q^{-3}\frac{J_{48}^3J_{6,48}\overline{J}_{8,48}}{J_{-3,48}J_{3,48}\overline{J}_{5,48}\overline{J}_{11,48}}.\label{revise-Cor4.8-2}
\end{align}
Substituting \eqref{revise-Cor4.8-1} and \eqref{revise-Cor4.8-2} into \eqref{2-2-cor-1-A}, we obtain \eqref{2-2-cor-1-A-simplify}.
\end{proof}
\begin{rem}
(1) After replacing $q$ by $-q$, the left side of \eqref{2-2-cor-1} becomes the third order mock theta function $\nu^{(3)}(q)$. See \eqref{mock-3-6-defn}. Identity \eqref{2-2-cor-1} was also found by Mortenson \cite[Eq.\ (2.6)]{Mortenson-2013}. See \eqref{Mortenson-nu}.
From \cite[Eq.\ 5.9]{Hickerson-Mortenson} we find
\begin{align}\label{HM-nu3-m}
\nu^{(3)}(q)=2q^{-1}m(q^2,q^{12},-q^3)+\frac{J_1J_{3,12}}{J_2}.
\end{align}
Taking $(x,q,z,z',n)\rightarrow (q^2,q^{12},-q^3,-q^3,2)$ in Lemma \ref{lem-m-decompose} and using the method in \cite{Garvan-Liang} to prove theta function identities, we see that \eqref{HM-nu3-m} is equivalent to \eqref{2-2-cor-1-A-simplify}.\\
(2) The left side of \eqref{2-2-cor-3-rho} is the sixth order mock theta function $\rho^{(6)}(q)$. See \eqref{mock-6-3-defn}. From \cite[Eq.\ (5.25)]{Hickerson-Mortenson} we find
\begin{align}\label{HM-rho6-m}
\rho^{(6)}(q)=-q^{-1}m(1,q^6,q).
\end{align}
Taking $(x,q,z_1,z_0)\rightarrow (1,q^6,-1,q)$ in Lemma \ref{lem-m-minus}, we see that \eqref{HM-rho6-m} is equivalent to \eqref{2-2-cor-3-rho-A}.
\end{rem}

\subsection{Representations for $\mu^{(2)}(q)$}
\begin{theorem}\label{thm-ab-2-3}
For $\max\{|ab/q|, |a|, |b|\}<1$, we have
\begin{align}
&\frac{(q,ab/q;q)_\infty}{(a,b;q)_\infty}
{}_3\phi_2\bigg(\genfrac{}{}{0pt}{}{q/a,q/b,q}{-q,-q};q,ab/q\bigg) \nonumber\\
=&~ 1+4\sum_{n=1}^{\infty}\frac{1}{1+q^{n}}\frac{(q/a,q/b;q)_n}{(a,b;q)_n}(-ab)^nq^{(n^2-n)/2}.
\end{align}
\end{theorem}
\begin{proof}
Taking $(c,d,u,v) \rightarrow (-q,-q,1,q)$ in Theorem \ref{thm-main}, we deduce that
\begin{align}
&\frac{(q,ab/q;q)_\infty}{(a,b;q)_\infty}
{}_3\phi_2\bigg(\genfrac{}{}{0pt}{}{q/a,q/b,q}{-q,-q};q,ab/q\bigg)\nonumber\\
=&1+\sum_{n=1}^{\infty}(1+q^{n})\frac{(q/a,q/b;q)_n}{(a,b;q)_n}(-ab)^nq^{(n^2-3n)/2}
{}_3\phi_2\bigg(\genfrac{}{}{0pt}{}{q^{-n},q^{n},q}{-q,-q};q,q\bigg). \label{ab-2-3-proof}
\end{align}
Taking $(a,b,c) \rightarrow (1,-1, -1)$  in \eqref{Pfaff}, we obtain
\begin{align}
{}_3\phi_2\bigg(\genfrac{}{}{0pt}{}{q^{-n}, q^{n},q}{-q,-q};q,q\bigg)=\frac{4q^{n}}{(1+q^{n})^2}. \label{mu-proof-2}
\end{align}
The theorem then follows after substituting \eqref{mu-proof-2} into \eqref{ab-2-3-proof}.
\end{proof}
\begin{corollary}
Identity \eqref{mock-2-3} holds. In addition, we have
\begin{align}
&\sum_{n=0}^\infty\frac{q^{n^2}}{(-q;q)_n^2}\nonumber \\
=&\frac{2}{(q;q)_\infty}\sum_{n=-\infty}^\infty \frac{(-1)^nq^{(3n^2+n)/2}}{1+q^{n}} \label{2-3-cor-1} \\
=&4m(-q,q^3,q^2)-\frac{J_3^4}{J_1J_6^2}, \label{2-3-cor-1-simplify} \\
&\sum_{n=0}^\infty\frac{q^{(n^2+n)/2}}{(1+q^{n})(-q;q)_n}=2\frac{(-q;q)_\infty}{(q;q)_\infty}\sum_{n=-\infty}^{\infty}\frac{(-1)^nq^{n^2+n}}{(1+q^{n})^2}, \label{2-3-cor-2} \\
&\sum_{n=0}^{\infty}\frac{q^{(n^2-n)/2}}{(-q;q)_n}=2. \label{2-3-cor-3-unusual}
\end{align}
\end{corollary}
\begin{proof}
Taking $(q, a,b)\rightarrow (q^2, 0,q)$, $(a,b)\rightarrow (0,0)$ and $(-q, 0)$ in Theorem \ref{thm-ab-2-3}, we obtain \eqref{mock-2-3}, \eqref{2-3-cor-1} and \eqref{2-3-cor-2}, respectively.

Taking $(a,b)\rightarrow (0,-1)$ in Theorem \ref{thm-ab-2-3}, we obtain
\begin{align}
&\sum_{n=0}^{\infty}\frac{q^{(n^2-n)/2}}{(-q;q)_n}=2\frac{(-q;q)_\infty}{(q;q)_\infty}\sum_{n=-\infty}^{\infty}(-1)^nq^{n^2}. \label{add-cor5.3-eq}
\end{align}
Note that
\begin{align*}
\sum_{n=-\infty}^{\infty}(-1)^nq^{n^2}=\frac{(q;q)_\infty^2}{(q^2;q^2)_\infty}.
\end{align*}
Identity \eqref{2-3-cor-3-unusual} follows immediately from \eqref{add-cor5.3-eq} and the above identity.
\end{proof}
\begin{rem}
(1) The left side of \eqref{2-3-cor-1} is the third order mock theta function $f^{(3)}(q)$. See \eqref{mock-3-1-defn}. From \cite[Eq.\ (5.4)]{Hickerson-Mortenson} we find
\begin{align}\label{HM-f3-m}
f^{(3)}(q)=4m(-q,q^3,q)+\frac{J_{3,6}^2}{J_1}.
\end{align}
Taking $(x,q,z_1,z_0)\rightarrow (-q,q^3,q^2,q)$ in Lemma \ref{lem-m-minus}, we see that \eqref{HM-f3-m} is equivalent to \eqref{2-3-cor-1-simplify}. \\
(2) The curious identity \eqref{2-3-cor-3-unusual} appeared in the work of Andrews \cite[Eq.\ (4.4)]{Andrews-GJM}, where it was proved combinatorially. Our proof of this identity is new. \\
(3) It seems that \eqref{2-3-cor-2} cannot be expressed as a nice form in terms of $m(x,q,z)$.
\end{rem}

\section{Mock theta functions of order 3}\label{sec-order-3}
In his last letter to Hardy and lost notebook \cite{lostnotebook}, Ramanujan gave seven mock theta functions of order 3. They are defined as
\begin{align}
&f^{(3)}(q):=\sum_{n=0}^{\infty}\frac{q^{n^2}}{(-q;q)_{n}^2}, \label{mock-3-1-defn} \\
&\phi^{(3)}(q):=\sum_{n=0}^{\infty}\frac{q^{n^2}}{(-q^2;q^2)_{n}},  \label{mock-3-2-defn} \\
&\psi^{(3)}(q):=\sum_{n=1}^{\infty}\frac{q^{n^2}}{(q;q^2)_{n}}, \label{mock-3-3-defn} \\
&\chi^{(3)}(q):=\sum_{n=0}^{\infty}\frac{q^{n^2}(-q;q)_{n}}{(-q^3;q^3)_{n}}, \label{mock-3-4-defn} \\
&\omega^{(3)}(q):=\sum_{n=0}^{\infty}\frac{q^{2n(n+1)}}{(q;q^2)_{n+1}^2}, \label{mock-3-5-defn} \\
&\nu^{(3)}(q):=\sum_{n=0}^{\infty}\frac{q^{n(n+1)}}{(-q;q^2)_{n+1}}, \label{mock-3-6-defn}\\
&\rho^{(3)}(q):=\sum_{n=0}^{\infty}\frac{q^{2n(n+1)}(q;q^2)_{n+1}}{(q^3;q^6)_{n+1}}\label{mock-3-7-defn}.
\end{align}
Watson \cite{Watson} provided Appell-Lerch series representations for these functions as follows.
\begin{theorem}\label{thm-ord-3}
We have
\begin{align}
&f^{(3)}(q)=\frac{2}{(q;q)_{\infty}}\sum_{n=-\infty}^{\infty}\frac{(-1)^nq^{n(3n+1)/2}}{1+q^n}, \label{mock-3-1} \\
&\phi^{(3)}(q)=\frac{1}{(q;q)_{\infty}}\sum_{n=-\infty}^{\infty}\frac{(-1)^n(1+q^n)q^{n(3n+1)/2}}{1+q^{2n}}, \label{mock-3-2} \\
&\psi^{(3)}(q)=\frac{q}{(q^4;q^4)_{\infty}}\sum_{n=-\infty}^{\infty}\frac{(-1)^nq^{6n(n+1)}}{1-q^{4n+1}}, \label{mock-3-3} \\
&\chi^{(3)}(q)=\frac{1}{2(q;q)_{\infty}}\sum_{n=-\infty}^{\infty}\frac{(-1)^n(1+q^n)q^{n(3n+1)/2}}{1-q^n+q^{2n}}, \label{mock-3-4} \\
&\omega^{(3)}(q)=\frac{1}{(q^2;q^2)_{\infty}}\sum_{n=0}^{\infty}(-1)^nq^{3n(n+1)}\frac{1+q^{2n+1}}{1-q^{2n+1}}, \label{mock-3-5} \\
&\nu^{(3)}(q)=\frac{1}{(q;q)_{\infty}}\sum_{n=0}^{\infty}(-1)^nq^{3n(n+1)/2}\frac{1-q^{2n+1}}{1+q^{2n+1}}, \label{mock-3-6}\\
&\rho^{(3)}(q)=\frac{1}{(q^2;q^2)_{\infty}}\sum_{n=0}^{\infty}(-1)^n(1-q^{4n+2})\frac{q^{3n^2+3n}}{1+q^{2n+1}+q^{4n+2}}. \label{mock-3-7}
\end{align}
\end{theorem}
We will provide new proofs for these formulas by establishing more general $(a,b)$-parameterized identities. As byproducts, we find a new Hecke-type series representation for $\psi^{(3)}(q)$ (see \eqref{3-3-cor-psi}), which is different from the following two given by Andrews \cite[Eq.\ (1.10)]{A12} and Mortenson \cite[Eq.\ (2.5)]{Mortenson-2013}, respectively:
\begin{align}
\psi^{(3)}(q)&=\frac{1}{(q;q)_\infty}\sum_{n=0}^\infty (-1)^nq^{2n^2+n}(1-q^{6n+6})\sum_{j=0}^nq^{-(j+1)j/2}-1, \label{Andrews-psi} \\
\psi^{(3)}(q)&=\frac{1}{2(q;q)_\infty}\sum_{n=0}^\infty (-1)^nq^{2n^2+n}(1+q^{2n+1})\sum_{j=-n}^{n}q^{-\binom{j+1}{2}}-\frac{1}{2}. \label{Mortenson-psi}
\end{align}
Applying routine arguments, we can show that these two representations are equivalent. In fact, both of them can be expressed directly as
\begin{align}
\psi^{(3}(q)=-\frac{1}{2}+\frac{1}{2J_1}\left(f_{3,5,3}(q^2,q^3,q)-q^3f_{3,5,3}(q^6,q^7,q) \right). \label{A-M-psi-3-equivalent}
\end{align}
In contrast, our new representation for $\psi^{(3)}(q)$ has different form in $f_{a,b,c}(x,y,q)$ (see \eqref{3-3-cor-psi-H-pre}). However, after converting \eqref{A-M-psi-3-equivalent} and our expression to forms in $m(x,q,z)$ and using properties of $m(x,q,z)$ we can show they can be deduced from each other.

Mortenson \cite[Eq.\ (2.6)]{Mortenson-2013} also found that
\begin{align}
\nu^{(3)}(-q)=\frac{1}{(q;q)_\infty}\sum_{n=0}^\infty (-1)^nq^{2n^2+2n}\sum_{j=-n}^nq^{-\binom{j+1}{2}}, \label{Mortenson-nu}
\end{align}
which coincides with \eqref{2-2-cor-1}. We will provide a different Hecke-type series representation for $\nu^{(3)}(q)$ in \eqref{6-2-cor-2-nu}, and we will also show that it is equivalent to \eqref{Mortenson-nu} using their forms in $m(x,q,z)$.

It is worthy mention that Garvan \cite{Garvan-2015} established two-variable Hecke-Rogers identities for three universal mock theta functions. As a direct consequence of one of these identities, he showed that each of Ramanujan's third-order mock theta function has a Hecke-type series representation. See \cite[Theorem 1.1, Corollary 1.2]{Garvan-2015} for details.
\subsection{Representations for $f^{(3)}(q)$}
Identity \eqref{mock-3-1} has  been established in \eqref{2-3-cor-1}.

\subsection{Representations for $\phi^{(3)}(q)$}
\begin{theorem}\label{thm-ab-3-2}
For $\max\{|ab/q|, |a|, |b|\}<1$, we have
\begin{align}
&\frac{(q,ab/q;q)_\infty}{(a,b;q)_\infty}
{}_3\phi_2\bigg(\genfrac{}{}{0pt}{}{q/a,q/b,q}{iq,-iq};q,ab/q\bigg)\\
\nonumber
=& 1+2\sum_{n=1}^{\infty}\frac{1+q^n}{1+q^{2n}}\frac{(q/a,q/b;q)_n}{(a,b;q)_n}(-ab)^nq^{(n^2-n)/2}.
\end{align}
\end{theorem}
\begin{proof}
Taking $(c,d,u,v)\rightarrow (iq,-iq,1,q)$ in Theorem \ref{thm-main}, we deduce that
\begin{align}
&\frac{(q,ab/q;q)_\infty}{(a,b;q)_\infty}
{}_3\phi_2\bigg(\genfrac{}{}{0pt}{}{q/a,q/b,q}{iq,-iq};q,ab/q\bigg) \nonumber \\
=& 1+\sum_{n=1}^{\infty}(1+q^{n})\frac{(q/a,q/b;q)_n}{(a,b;q)_n}(-ab)^nq^{(n^2-3n)/2}
{}_3\phi_2\bigg(\genfrac{}{}{0pt}{}{q^{-n},q^{n},q}{iq,-iq};q,q\bigg). \label{ab-3-2-proof-1}
\end{align}
Setting $(a,b,c)=(1,i,-i)$ in \eqref{Pfaff}, we obtain
\begin{align}
{}_3\phi_2\bigg(\genfrac{}{}{0pt}{}{q^{-n}, q^{n},q}{iq,-iq};q,q\bigg)=\frac{2q^n}{1+q^{2n}}. \label{phi-proof-2}
\end{align}
Substituting \eqref{phi-proof-2} into \eqref{ab-3-2-proof-1}, we obtain the theorem.
\end{proof}

\begin{corollary}
Identity \eqref{mock-3-2} holds. In addition, we have
\begin{align}
&\sum_{n=0}^{\infty}\frac{(-q;q)_nq^{(n^2-n)/2}}{(-q^2;q^2)_n}\nonumber \\
=&\frac{(-q;q)_\infty}{(q;q)_\infty}\sum_{n=-\infty}^{\infty}\frac{(1+q^n)^2}{1+q^{2n}}(-1)^nq^{n^2} \label{3-2-cor-1} \\
=&1+\frac{J_2^5}{J_1^2J_4^2}, \label{3-2-cor-1-simplify} \\
&\sum_{n=0}^{\infty}\frac{q^{(n^2+n)/2}(-1;q)_n}{(-q^2;q^2)_n}\nonumber \\
=&2\frac{(-q;q)_\infty}{(q;q)_\infty}\sum_{n=-\infty}^{\infty}\frac{(-1)^nq^{n^2+n}}{1+q^{2n}} \label{3-2-cor-2} \\
=&\frac{J_2^5}{J_1^2J_4^2}, \label{3-2-cor-2-simplify} \\
&\sum_{n=0}^\infty \frac{q^{n^2}(-q;q^2)_n}{(-q^4;q^4)_n} \nonumber \\ =&\frac{(-q;q^2)_{\infty}}{(q^2;q^2)_{\infty}}\sum_{n=-\infty}^{\infty}\frac{1+q^{2n}}{1+q^{4n}}(-1)^{n}q^{2n^2+n} \label{3-2-cor-3} \\
=&m(-q,q^4,q^{-1})-q^{-1}m(-q^{-1},q^4,q)=2m(-q,q^4,q^3). \label{3-2-cor-3-simplify}
\end{align}
\end{corollary}
\begin{proof}
Taking $(a,b)\rightarrow (0,0)$, $(0, -1)$ and $(0, -q)$ in Theorem \ref{thm-ab-3-2}, we obtain \eqref{mock-3-2}, \eqref{3-2-cor-1} and \eqref{3-2-cor-2}, respectively.

Taking $(q,a,b)\rightarrow (q^2, 0, -q)$ in Theorem \ref{thm-ab-3-2}, we obtain \eqref{3-2-cor-3}.

Taking $(q,z)\rightarrow (q^2,-1)$ in \eqref{reciprocal-Jacobi}, we obtain
\begin{align}
\sum_{n=-\infty}^\infty \frac{(-1)^nq^{n^2+n}}{1+q^{2n}}=\frac{J_2^4}{2J_4^2}. \label{reciprocal-apply-1}
\end{align}
Substituting this identity into \eqref{3-2-cor-1} (resp.\ \eqref{3-2-cor-2}), we get \eqref{3-2-cor-1-simplify} (resp.\ \eqref{3-2-cor-2-simplify}).
\end{proof}
\begin{rem}
The left side of \eqref{3-2-cor-3} is the eighth order mock theta function $U_0^{(8)}(q)$. See \eqref{mock-8-5-defn}. From \cite[Eq.\ (5.39)]{Hickerson-Mortenson} we find
\begin{align}\label{HM-U08-m}
U_0^{(8)}(q)=2m(-q,q^4,-1).
\end{align}
Taking $(x,q,z_1,z_0)\rightarrow (-q,q^4,q^3,-1)$ in Lemma \ref{lem-m-minus}, we see that \eqref{HM-U08-m} is equivalent to \eqref{3-2-cor-3-simplify}.
\end{rem}

\subsection{Representations for $\psi^{(3)}(q)$}
\begin{theorem}\label{thm-ab-3-3-add}
For $\max \{|ab|, |qa|, |qb|\}<1$ and $x$ being neither zero nor an integral powers of $q$, we have
\begin{align}
&\frac{(q,ab;q)_\infty}{(aq,bq;q)_\infty}{}_3\phi_2\bigg(\genfrac{}{}{0pt}{}{q/a,q/b,q}{xq,q^2/x};q,ab\bigg)\nonumber \\
=&\sum_{n=0}^\infty(1-q^{2n+1})\frac{(1-x)(1-q/x)}{(1-xq^n)(1-q^{n+1}/x)}\frac{(q/a,q/b;q)_n}{(aq,bq;q)_n}(-ab)^nq^{(n^2+n)/2}.
\end{align}
\end{theorem}
\begin{proof}
Taking $(c,d, u,v)\rightarrow (xq, q^2/x, q,q)$ in Theorem \ref{thm-main}, we deduce that
\begin{align}
&\frac{(q,ab;q)_\infty}{(aq,bq;q)_\infty}{}_3\phi_2\bigg(\genfrac{}{}{0pt}{}{q/a,q/b,q}{xq,q^2/x};q,ab\bigg)\nonumber \\
=& \sum_{n=0}^\infty(1-q^{2n+1})\frac{(q/a,q/b;q)_n}{(aq,bq;q)_n}(-ab)^nq^{(n^2-n)/2}{}_3\phi_2\bigg(\genfrac{}{}{0pt}{}{q^{-n},q^{n+1},q}{xq,q^2/x};q,q\bigg). \label{ab-3-3-add-proof}
\end{align}
Setting $(a,b,c)=(q,x,q/x)$ in \eqref{Pfaff}, we deduce that
\begin{align}
{}_{3}\phi_{2}\bigg(\genfrac{}{}{0pt}{}{q^{-n}, q^{n+1},q}{xq,q^2/x};q,q\bigg)=\frac{(1-x)(1-q/x)q^n}{(1-xq^n)(1-q^{n+1}/x)}. \label{g-2}
\end{align}
The theorem follows after substituting \eqref{g-2} into \eqref{ab-3-3-add-proof}.
\end{proof}

Theorem \ref{thm-ab-3-3-add} allows us to give alternative representations for the universal mock theta function
\begin{align}\label{g-defn}
g(x,q):=x^{-1}\left(-1+\sum_{n=0}^{\infty}\frac{q^{n^2}}{(x;q)_{n+1}(q/x;q)_{n}}\right)=\sum_{n=0}^{\infty}\frac{q^{n(n+1)}}{(x;q)_{n+1}(q/x;q)_{n+1}}, \end{align}
where $x$ is neither 0 nor an integral power of $q$. This function was defined by Hickerson \cite[Definition 2.0]{Hickerson-proof}.

We can write
\begin{align}
g(x,q)=\frac{1}{(1-x)(1-q/x)}\sum_{n=0}^{\infty}\frac{q^{n(n+1)}}{(xq,q^2/x;q)_n}. \label{g-exp}
\end{align}
Taking $(a,b)\rightarrow (0,0)$ in Theorem \ref{thm-ab-3-3-add}, we get
\begin{align}
g(x,q)&=\frac{1}{(q;q)_{\infty}}\sum_{n=0}^{\infty}\frac{(-1)^nq^{(3n^2+3n)/2}(1-q^{2n+1})}{(1-xq^n)(1-q^{n+1}/x)}. \label{g-last}
\end{align}
We may further write
\begin{align}
g(x,q)&=\frac{1}{(q;q)_{\infty}}\left(\sum_{n=0}^{\infty}\frac{(-1)^nq^{(3n^2+3n)/2}}{(1-xq^n)(1-q^{n+1}/x)}
+\sum_{n=0}^{\infty}\frac{(-1)^{n+1}q^{(3n^2+3n)/2+2n+1}}{(1-xq^n)(1-q^{n+1}/x)}\right) \nonumber \\
&=\frac{1}{(q;q)_{\infty}}\left(\sum_{n=0}^{\infty}\frac{(-1)^nq^{(3n^2+3n)/2}}{(1-xq^n)(1-q^{n+1}/x)}
+\sum_{n=-\infty}^{-1}\frac{(-1)^{n}q^{(3n^2+3n)/2-2n-1}}{(1-xq^{-n-1})(1-q^{-n}/x)}\right) \nonumber \\
&=\frac{1}{(q;q)_{\infty}}\sum_{n=-\infty}^{\infty}\frac{(-1)^nq^{3n(n+1)/2}}{(1-xq^n)(1-q^{n+1}/x)}. \label{g-final}
\end{align}
Moreover,  we can also write
\begin{align}
g(x,q)&=\frac{1}{(q;q)_{\infty}}\sum_{n=0}^{\infty}\frac{(-1)^nq^{(3n^2+3n)/2}\left((1-q^{n+1}/x)+(1-xq^n)q^{n+1}/x\right)}{(1-xq^n)(1-q^{n+1}/x)} \nonumber \\
&=\frac{1}{(q;q)_{\infty}}\left(\sum_{n=0}^{\infty}\frac{(-1)^nq^{(3n^2+3n)/2}}{1-xq^n}+\sum_{n=0}^{\infty}\frac{(-1)^nq^{(3n^2+3n)/2}\cdot q^{n+1}/x}{1-q^{n+1}/x} \right)\nonumber \\
&=\frac{1}{(q;q)_{\infty}}\left(\sum_{n=0}^{\infty}\frac{(-1)^nq^{(3n^2+3n)/2}}{1-xq^n}+\sum_{n=-\infty}^{-1}\frac{(-1)^{-n-1}q^{(3n^2+3n)/2}\cdot q^{-n}/x}{1-q^{-n}/x} \right)\nonumber \\
&=\frac{1}{(q;q)_\infty}\sum_{n=-\infty}^{\infty}\frac{(-1)^nq^{3n(n+1)/2}}{1-xq^n}. \label{g-simple}
\end{align}
Identity \eqref{g-simple} appears as Lemma 7.9 in the work of Garvan \cite{Garvan}.
\begin{corollary}
Identity \eqref{mock-3-3} holds.
\end{corollary}
\begin{proof}
We have
\begin{align}
\psi^{(3)}(q)&=\sum_{n=0}^{\infty}\frac{q^{(2n+1)^2}}{(q;q^2)_{2n+1}}+\sum_{n=1}^\infty\frac{q^{(2n)^2}}{(q;q^2)_{2n}} \nonumber \\
&=\frac{q}{1-q}+\sum_{n=1}^\infty\frac{q^{4n^2+4n+1}+q^{4n^2}(1-q^{4n+1})}{(q;q^2)_{2n+1}} \nonumber \\
&=\frac{q}{1-q}+\frac{1}{1-q}\sum_{n=1}^\infty\frac{q^{4n^2}}{(q^3;q^2)_{2n}} \nonumber \\
&=-1+\frac{1}{1-q}\sum_{n=0}^{\infty}\frac{q^{4n^2}}{(q^3,q^5;q^4)_n}.
\end{align}
Therefore, by \eqref{g-defn} we have $\psi^{(3)}(q)=qg(q,q^4)$. Identity \eqref{mock-3-3} then follows from \eqref{g-simple}.
\end{proof}

We can give another Hecke-type series representation for $\psi^{(3)}(q)$. For this we establish the following parameterized identity.
\begin{theorem}\label{thm-ab-3-3}
For $\max\{|ab/q|, |a|, |b|\}<1$, we have
\begin{align}
&\frac{(q,ab/q;q)_\infty}{(a,b;q)_\infty}
{}_3\phi_2\bigg(\genfrac{}{}{0pt}{}{q/a,q/b,q}{q^{1/2},-q^{1/2}};q,ab/q\bigg) \nonumber\\
= &  1+\sum_{n=1}^{\infty}(1+q^{n})q^{n^2-2n}\frac{(q/a,q/b;q)_n}{(a,b;q)_n}(-ab)^n \left(q^{-n(n+1)/2}-(1-q^n)\sum_{j=0}^{n}q^{-j(j+1)/2} \right). \label{ab-3-3}
\end{align}
\end{theorem}
\begin{proof}
Taking $(q,c,d,u,v)\rightarrow (q,q^{1/2},-q^{1/2},1,q)$ in Theorem \ref{thm-main}, we deduce that
\begin{align}
&\frac{(q,ab/q;q)_\infty}{(a,b;q)_\infty}
{}_3\phi_2\bigg(\genfrac{}{}{0pt}{}{q/a,q/b,q}{q^{1/2},-q^{1/2}};q,ab/q\bigg) \nonumber\\
=&1+\sum_{n=1}^{\infty}(1+q^{n})\frac{(q/a,q/b;q)_n}{(a,b;q)_n}(-ab)^nq^{(n^2-3n)/2}
{}_3\phi_2\bigg(\genfrac{}{}{0pt}{}{q^{-n},q^{n},q}{q^{1/2},-q^{1/2}};q,q\bigg). \label{ab-3-3-proof}
\end{align}
Setting $(\alpha, \beta, c, d)=(1, q, q^{1/2}, -q^{1/2})$ in \eqref{sears:32}, we obtain
\begin{align}
{}_3\phi_2\bigg(\genfrac{}{}{0pt}{}{q^{-n},q^{n},q}{q^{1/2},-q^{1/2}};q,q\bigg)
=(-1)^nq^{n^2/2}{}_3\phi_2\bigg(\genfrac{}{}{0pt}{}{q^{-n},q^{n},-q^{-1/2}}{q^{1/2},-q^{1/2}};q,q^{3/2}\bigg). \label{ab-3-3-proof-1}
\end{align}
Taking $(c, d)\rightarrow (-q^{3/2}, q^{3/2})$ in Lemma \ref{lem-limit}, we deduce that
\begin{align}
&{}_3\phi_2\bigg(\genfrac{}{}{0pt}{}{q^{-n},q^{n},-q^{-1/2}}{q^{1/2},-q^{1/2}};q,q^{3/2}\bigg) \nonumber \\
=&~(-1)^nq^{-n/2}(1-q^n)\left(\frac{q}{1-q}+\sum_{j=2}^{n}\frac{(1-q^{-1})q^{2j-j(j+1)/2}}{(1-q^{j-1})(1-q^j)}\right). \label{ab-3-3-proof-2}
\end{align}
Note that
\begin{align*}
(1-q^{-1})q^{j}=(1-q^{j-1})-(1-q^j).
\end{align*}
We deduce that
\begin{align}
&\sum_{j=2}^{n}\frac{(1-q^{-1})q^{2j-j(j+1)/2}}{(1-q^{j-1})(1-q^j)}\nonumber \\
=& \sum_{j=2}^{n}q^{(j-j^2)/2}\left(\frac{1}{1-q^j}-\frac{1}{1-q^{j-1}} \right) \nonumber \\
=& \sum_{j=2}^{n}\frac{q^{(j-j^2)/2}}{1-q^j}-\sum_{j=1}^{n-1}\frac{q^{-j(j+1)/2}}{1-q^{j}}  \nonumber \\
=&-\sum_{j=1}^{n}\frac{q^{-j(j+1)/2}(1-q^j)}{1-q^j}-\frac{1}{1-q}+\frac{q^{-n(n+1)/2}}{1-q^n} \nonumber \\
=&~-\sum_{j=1}^{n}q^{-j(j+1)/2}-\frac{1}{1-q}+\frac{q^{-n(n+1)/2}}{1-q^n}. \label{ab-3-3-proof-3}
\end{align}
Substituting \eqref{ab-3-3-proof-3} into \eqref{ab-3-3-proof-2}, we conclude that
\begin{align}
&{}_3\phi_2\bigg(\genfrac{}{}{0pt}{}{q^{-n},q^{n},-q^{-1/2}}{q^{1/2},-q^{1/2}};q,q^{3/2}\bigg) \nonumber \\
=&~(-1)^nq^{-n/2}\left(q^{-n(n+1)/2}-(1-q^n)\sum_{j=0}^{n}q^{-j(j+1)/2} \right). \label{ab-3-3-proof-4}
\end{align}
Substituting \eqref{ab-3-3-proof-4} into \eqref{ab-3-3-proof-1}, we obtain
\begin{align}
{}_3\phi_2\bigg(\genfrac{}{}{0pt}{}{q^{-n},q^{n},q}{q^{1/2},-q^{1/2}};q,q\bigg)=q^{(n^2-n)/2}\left(q^{-n(n+1)/2}-(1-q^n)\sum_{j=0}^{n}q^{-j(j+1)/2} \right). \label{ab-3-3-proof-final}
\end{align}
The theorem follows after substituting \eqref{ab-3-3-proof-final} into \eqref{ab-3-3-proof}.
\end{proof}
\begin{corollary}
We have
\begin{align}
\psi^{(3)}(q)=&~ -\frac{1}{(q;q)_\infty}\sum_{n=1}^\infty(-1)^n(1-q^{2n})q^{2n^2-n}\sum_{j=0}^{n-1}q^{-j(j+1)/2} \label{3-3-cor-psi}\\
=&\frac{1}{2J_1}\left(qf_{3,5,3}(q^4,q^5,q)-f_{3,5,3}(1,q,q) \right) \label{3-3-cor-psi-H-pre} \\
=&-\frac{1}{J_1}f_{3,5,3}(1,q,q) \label{3-3-cor-psi-H} \\
=&-m(-q^{22},q^{48},q^3)-q^{-1}m(-q^{14},q^{48},q^{-3})-q^{-2}m(-q^{10},q^{48},q^{-3})\nonumber \\
&-q^{-5}m(-q^2,q^{48},q^3)+\frac{1}{J_1}\Theta_{3,2}(1,q,q) \label{3-3-cor-psi-A}\\
=&-m(q,-q^3,-q)+q\frac{J_{12}^3}{J_{4,12}J_{3,12}}, \label{3-3-cor-psi-final}
\end{align}
\begin{align}
&\sum_{n=0}^{\infty}\frac{q^{(n^2+n)/2}(-1;q)_n}{(q;q^2)_n}\nonumber \\
=&~\frac{(-q;q)_\infty}{(q;q)_\infty} \Big(1+2\sum_{n=1}^{\infty}(-1)^nq^{(3n^2-n)/2}\Big(q^{-n(n+1)/2}-(1-q^n)\sum_{j=0}^nq^{-j(j+1)/2}\Big)\Big) \label{3-3-cor-1} \\
=& 1-2\frac{J_2}{J_1^2}f_{1,2,1}(q,1,q^2) \label{3-3-cor-1-H} \\
=&1-2m(q^2,q^6,q). \label{3-3-cor-1-A}
\end{align}
Furthermore, \eqref{3-3-cor-psi} can be deduced from \eqref{Andrews-psi} or \eqref{Mortenson-psi} and vice versa.
\end{corollary}
\begin{proof}
Taking $(a,b)\rightarrow (0,0)$ and $(0,-q)$ in Theorem \ref{thm-ab-3-3}, we obtain \eqref{3-3-cor-psi} and \eqref{3-3-cor-1}, respectively.

Now we write the Hecke-type series in terms of $f_{a,b,c}(x,y,q)$. We only give a proof for \eqref{3-3-cor-1-H}. The proof of \eqref{3-3-cor-psi-H} is similar and easier.

Note that
\begin{align}
\sum_{j=0}^nq^{-j(j+1)/2}=\frac{1}{2}\sum_{j=-n}^nq^{-j(j+1)/2}+\frac{1}{2}q^{-n(n+1)/2}.
\end{align}
After rearrangements and simplifications, the right side of \eqref{3-3-cor-1} can be written as
\begin{align}
&\frac{J_2}{J_1^2}\left(1+2\sum_{n=1}^\infty (-1)^nq^{n^2} -\sum_{n=0}^\infty \sum_{j=-n}^n (-1)^nq^{(3n^2-n)/2-j(j+1)/2} \right. \nonumber \\
&\quad \quad \left. +\sum_{n=-\infty}^{-1} \sum_{j=-n}^n (-1)^nq^{(3n^2-n)/2-j(j+1)/2} \right) \nonumber \\
=&1-\frac{J_2}{J_1^2}\left(\sum_{\substack{n+j\geq 0\\ n-j\geq 0}}-\sum_{\substack{n+j<0\\ n-j<0}} \right)(-1)^nq^{(3n^2-n)/2-j(j+1)/2}  \nonumber \\
=& 1-\frac{J_2}{J_1^2} \left(f_{1,2,1}(q,1,q^2)-qf_{1,2,1}(q^3,q^4,q^2) \right).  \label{3-3-cor-1-revise-1}
\end{align}
Here the last equality follows by setting $n+j=r$ and $n-j=s$ and  routine arguments. Now by \eqref{f-id-4} with $(a,b,c,x,y,q)\rightarrow (1,2,1,q,1,q^2)$ we deduce that
\begin{align}
f_{1,2,1}(q,1,q^2)=-qf_{1,2,1}(q^3,q^4,q^2).
\end{align}
Substituting this identity into \eqref{3-3-cor-1-revise-1}, we arrive at \eqref{3-3-cor-1-H}.

The expression \eqref{3-3-cor-psi-A} (resp.\ \eqref{3-3-cor-1-A}) follows from \eqref{3-3-cor-psi-H} (resp.\ \eqref{3-3-cor-1-H}) and Theorem \ref{thm-fg-2} (resp.\ Theorem \ref{thm-fg-1}) with $n=3$ (resp.\ $(n,p)=(1,1)$).

Taking $(x,q,z,z',n)\rightarrow (q,-q^3,-q,q^3,4)$ in Lemma \ref{lem-m-decompose}, applying Lemma \ref{lem-m-minus} and using the method in \cite{Garvan-Liang} to prove theta function identities, we see that \eqref{3-3-cor-psi-final} follows from \eqref{3-3-cor-psi-A}.

To show that the representation \eqref{3-3-cor-psi} is equivalent to the representations of Andrews and Mortenson given in \eqref{Andrews-psi} and \eqref{Mortenson-psi}, we first convert the expression in \eqref{A-M-psi-3-equivalent} to a form in terms of $m(x,q,z)$.

Applying Theorem \ref{thm-fg-2} with $n=3$, we obtain
\begin{align}
f_{3,5,3}(q^2,q^3,q)=&J_1\left(m(-q^{26},q^{48},q^{-3})-q^{-2}m(-q^{10},q^{48},q^{-3})\right. \nonumber \\
&\left.-q^{-5}m(-q^2,q^{48},q^3)-q^{-1}m(-q^{14},q^{48},q^{-3}) \right)-\Theta_{3,2}(q^2,q^3,q) \label{revise-psi-equiv-1}
\end{align}
and
\begin{align}\label{revise-psi-equiv-2}
f_{3,5,3}(q^6,q^7,q)=&J_1\left(q^{-8}m(-q^2,q^{48},q^{-3})+q^{-4}m(-q^{14},q^{48},q^3)\right. \nonumber \\
&\left. +q^{-5}m(-q^{10},q^{48},q^3)+q^{-3}m(-q^{22},q^{48},q^{-3})\right)-\Theta_{3,2}(q^6,q^7,q).
\end{align}
Substituting \eqref{revise-psi-equiv-1} and \eqref{revise-relation-2} into \eqref{A-M-psi-3-equivalent}, we see that \eqref{A-M-psi-3-equivalent} can be written as
\begin{align}\label{revise-psi-equiv-3}
\psi^{(3)}(q)=&-\frac{1}{2}+\frac{1}{2}\left(m(-q^{26},q^{48},q^{-3})-q^{-2}m(-q^{10},q^{48},q^{-3})\right. \nonumber \\
&-q^{-5}m(-q^2,q^{48},q^3)-q^{-1}m(-q^{14},q^{48},q^{-3})-q^{-5}m(-q^2,q^{48},q^{-3}) \nonumber \\
&\left.-q^{-1}m(-q^{14},q^{48},q^3)-q^{-2}m(-q^{10},q^{48},q^3)-m(-q^{22},q^{48},q^{-3}) \right)\nonumber \\
&-\frac{1}{2J_1}\left(\Theta_{3,2}(q^2,q^3,q)-q^3\Theta_{3,2}(q^6,q^7,q) \right).
\end{align}
Now we simplify the above identity. Taking $(x,q,z_1,z_0)\rightarrow (-q^2,q^{48},q^3,q^{-3})$ in Lemma \ref{lem-m-minus}, we deduce that
\begin{align}\label{revise-psi-equiv-4}
m(-q^2,q^{48},q^3)-m(-q^2,q^{48},q^{-3})
=q^{-3}\frac{J_{48}^3J_{6,48}\overline{J}_{2,48}}{J_{-3,48}J_{3,48}\overline{J}_{-1,48}\overline{J}_{5,48}}.
\end{align}
Taking $(x,q,z_1,z_0)\rightarrow (-q^{14},q^{48},q^3,q^{-3})$ in Lemma \ref{lem-m-minus}, we deduce that
\begin{align}\label{revise-psi-equiv-5}
m(-q^{14},q^{48},q^3)-m(-q^{14},q^{48},q^{-3})
=q^{-3}\frac{J_{48}^3J_{6,48}\overline{J}_{14,48}}{J_{-3,48}J_{3,48}\overline{J}_{11,48}\overline{J}_{17,48}}.
\end{align}
Taking $(x,q,z_1,z_0)\rightarrow (-q^{10},q^{48},q^3,q^{-3})$ in Lemma \ref{lem-m-minus}, we deduce that
\begin{align}\label{revise-psi-equiv-6}
m(-q^{10},q^{48},q^3)-m(-q^{10},q^{48},q^{-3})
=q^{-3}\frac{J_{48}^3J_{6,48}\overline{J}_{10,48}}{J_{-3,48}J_{3,48}\overline{J}_{7,48}\overline{J}_{13,48}}.
\end{align}
Taking $(x,q,z_1,z_0)\rightarrow (-q^{22},q^{48},q^3,q^{-3})$ in Lemma \ref{lem-m-minus}, we deduce that
\begin{align}\label{revise-psi-equiv-7}
m(-q^{22},q^{48},q^{3})-m(-q^{22},q^{48},q^{-3})
=q^{-3}\frac{J_{48}^3J_{6,48}\overline{J}_{22,48}}{J_{-3,48}J_{3,48}\overline{J}_{19,48}\overline{J}_{25,48}}.
\end{align}
Finally, setting $(x,q,z)\rightarrow (-q^{26},q^{48},q^{-3})$ in \eqref{m-id-2} and then using \eqref{m-id-3}, we deduce that
\begin{align}\label{revise-psi-equiv-8}
m(-q^{26},q^{48},q^{-3})=-q^{-26}m(-q^{-26},q^{48},q^3)=1-m(-q^{22},q^{48},q^3).
\end{align}
Substituting \eqref{revise-psi-equiv-4}--\eqref{revise-psi-equiv-8} into \eqref{revise-psi-equiv-3}, we see that \eqref{revise-psi-equiv-3} differs from \eqref{3-3-cor-psi-A} only by some infinite products. But showing that these infinite products sum to zero is a routine exercise, which may also be verified using the method in \cite{Garvan-Liang}.
\end{proof}

\begin{rem}
(1) The expression \eqref{3-3-cor-psi-final} can be found in \cite[Eq.\ (5.6)]{Hickerson-Mortenson}. See Remark \ref{rem-mock-2-A} (1).\\
(2) From \cite[Eq.\ (5.26)]{Hickerson-Mortenson} we know that
\begin{align}
\sigma^{(6)}(q)=-m(q^2,q^6,q).
\end{align}
Comparing this with \eqref{3-3-cor-1-A}, we deduce that
\begin{align}
\sum_{n=0}^{\infty}\frac{q^{(n^2+n)/2}(-1;q)_n}{(q;q^2)_n}=1+2\sigma^{(6)}(q). \label{revise-relation-3}
\end{align}
This relation can also be deduced directly from the definition of $\sigma^{(6)}(q)$ given in \eqref{mock-6-4-defn}.
\end{rem}

\begin{theorem}\label{thm-ab-3-4}
For $\max\{|ab/q|, |a|, |b|\}<1$, we have
\begin{align}
&\frac{(q,ab/q;q)_\infty}{(a,b;q)_\infty}
{}_3\phi_2\bigg(\genfrac{}{}{0pt}{}{q/a,q/b,q}{-\zeta_3q,-\zeta_3^2q};q,ab/q\bigg)\\
\nonumber
=&1+\sum_{n=1}^{\infty}\frac{1+q^{n}}{1-q^{n}+q^{2n}}\frac{(q/a,q/b;q)_n}{(a,b;q)_n}(-ab)^nq^{(n^2-n)/2}.
\end{align}
\end{theorem}
\begin{proof}
Taking $(c,d,u,v)\rightarrow (-\zeta_3q,-\zeta_3^2q,1,q)$ in  Theorem \ref{thm-main}, we deduce that
\begin{align}
&\frac{(q,ab/q;q)_\infty}{(a,b;q)_\infty}
{}_3\phi_2\bigg(\genfrac{}{}{0pt}{}{q/a,q/b,q}{-\zeta_3q,-\zeta_3^2q};q,ab/q\bigg) \nonumber\\
=&1+\sum_{n=1}^{\infty}(1+q^{n})\frac{(q/a,q/b;q)_n}{(a,b;q)_n}(-ab)^nq^{(n^2-3n)/2}
{}_3\phi_2\bigg(\genfrac{}{}{0pt}{}{q^{-n},q^{n},q}{-\zeta_3q,-\zeta_3^2q};q,q\bigg).\label{ab-3-4-proof}
\end{align}
Setting $(a,b,c)=(1, -\zeta_3, -\zeta_{3}^2)$ in \eqref{Pfaff}, we obtain
\begin{align}
{}_3\phi_2\bigg(\genfrac{}{}{0pt}{}{q^{-n}, q^{n},q}{-\zeta_{3}q,-\zeta_{3}^2q};q,q\bigg)=\frac{q^n}{1-q^n+q^{2n}}. \label{chi-proof-2}
\end{align}
The theorem follows after substituting \eqref{chi-proof-2} into \eqref{ab-3-4-proof}.
\end{proof}
\begin{corollary}
Identity \eqref{mock-3-4} holds. In addition, we have
\begin{align}
&\sum_{n=0}^\infty\frac{q^{(n^2+n)/2}(-1,-q;q)_n}{(-q^3;q^3)_n}\nonumber \\
=& \frac{(-q;q)_\infty}{(q;q)_\infty}\sum_{n=-\infty}^\infty\frac{(-1)^nq^{n^2+n}}{1-q^n+q^{2n}} \label{3-4-cor-1} \\
= & \frac{1-\zeta_{6}}{1+\zeta_6}\left(m(\zeta_6^2q,q^2,-\zeta_6^2)-m(-\zeta_6q,q^2,\zeta_6) \right), \label{3-4-cor-1-simplify} \\
&\sum_{n=0}^\infty\frac{q^{(n^2-n)/2}(-q;q)_n^2}{(-q^3;q^3)_n}\nonumber \\
=& \frac{1}{2}\frac{(-q;q)_\infty}{(q;q)_\infty}\sum_{n=-\infty}^\infty\frac{(-1)^nq^{n^2}(1+q^n)^2}{1-q^n+q^{2n}} \label{3-4-cor-2} \\
=& \frac{1}{2}+\frac{3(1-\zeta_6)}{2(1+\zeta_6)}\left(m(\zeta_6^2q,q^2,-\zeta_6^2)-m(-\zeta_6q,q^2,\zeta_6) \right). \label{3-4-cor-2-simplify}
\end{align}
\end{corollary}
\begin{proof}
Taking $(a,b)\rightarrow (0,0)$, $(0,-q)$ and $(0,-1)$ in Theorem \ref{thm-ab-3-4}, we obtain \eqref{mock-3-4}, \eqref{3-4-cor-1} and \eqref{3-4-cor-2}, respectively.

To write the Appell-Lerch series in terms of $m(x,q,z)$, we use the fact that
\begin{align}
\frac{1}{1-q^n+q^{2n}}=\frac{1}{1+\zeta_6}\left(\frac{1}{1-\zeta_6q^n}+\frac{\zeta_6}{1+\zeta_6^2q^n} \right).
\end{align}
Then we can rewrite the Appell-Lerch series on the right side of \eqref{3-4-cor-1} as
\begin{align}
\sum_{n=-\infty}^\infty\frac{(-1)^nq^{n^2+n}}{1-q^n+q^{2n}} =& \frac{1}{1+\zeta_6}\left(\sum_{n=-\infty}^\infty \frac{(-1)^nq^{n^2+n}}{1-\zeta_6q^n}+\zeta_6\sum_{n=-\infty}^\infty \frac{(-1)^nq^{n^2+n}}{1+\zeta_6^2q^n} \right)\nonumber \\
=& \frac{1}{1+\zeta_6}j(q;q^2)\left(h(\zeta_6,q)+\zeta_6h(-\zeta_6^2,q)\right).
\end{align}
Here we used \eqref{h-defn}. Now applying \eqref{h-m} we arrive at \eqref{3-4-cor-1-simplify}.

Observe that
\begin{align}
\frac{(-1)^nq^{n^2}(1+q^n)^2}{1-q^n+q^{2n}}=(-1)^nq^{n^2}+3\frac{(-1)^nq^{n^2+n}}{1-q^n+q^{2n}}.
\end{align}
The identity \eqref{3-4-cor-2-simplify} then follows from \eqref{3-4-cor-1-simplify}.
\end{proof}
\begin{rem}
Comparing \eqref{3-4-cor-1-simplify} with \eqref{3-4-cor-2-simplify} we deduce the following identity:
\begin{align}
\sum_{n=0}^\infty\frac{q^{(n^2-n)/2}(-q;q)_n^2}{(-q^3;q^3)_n}=\frac{3}{2}\sum_{n=0}^\infty\frac{q^{(n^2+n)/2}(-1,-q;q)_n}{(-q^3;q^3)_n}+\frac{1}{2}. \label{revise-relation-2}
\end{align}
\end{rem}
\subsection{Representations for $\omega^{(3)}(q)$}
As for $\omega^{(3)}(q)$, a two parameter generalization of \eqref{mock-3-5} has already been given in Theorem \ref{thm-ab-2-1-1st}. Identity \eqref{mock-3-5} already appears as \eqref{2-1-cor-3}.

\subsection{Representations for $\nu^{(3)}(q)$}
\begin{theorem}\label{thm-ab-3-6}
For $\max\{|ab|, |aq|, |bq|\}<1$, we have
\begin{align}
&\frac{(q,ab;q)_\infty}{(aq,bq;q)_\infty}
{}_3\phi_2\bigg(\genfrac{}{}{0pt}{}{q/a,q/b,q}{iq^{3/2},-iq^{3/2}};q,ab\bigg)\\
\nonumber
=&(1+q)\sum_{n=0}^{\infty}\frac{1-q^{2n+1}}{1+q^{2n+1}}\frac{(q/a,q/b;q)_n}{(aq,bq;q)_n}(-ab)^nq^{(n^2+n)/2}.
\end{align}
\end{theorem}
\begin{proof}
Taking $(c,d,u,v)\rightarrow (iq^{3/2},-iq^{3/2},q,q)$ in Theorem \ref{thm-main}, we deduce that
\begin{align}
&\frac{(q,ab;q)_\infty}{(aq,bq;q^2)_\infty}
{}_3\phi_2\bigg(\genfrac{}{}{0pt}{}{q/a,q/b,q}{iq^{3/2},-iq^{3/2}};q,ab\bigg) \nonumber \\
=&\sum_{n=0}^{\infty}(1-q^{2n+1})\frac{(q/a,q/b;q)_n}{(aq,bq;q)_n}(-ab)^nq^{(n^2-n)/2}
{}_3\phi_2\bigg(\genfrac{}{}{0pt}{}{q^{-n},q^{n+1},q}{iq^{3/2},-iq^{3/2}};q,q\bigg). \label{ab-3-6-proof}
\end{align}
Setting $(a,b,c)=(q,iq^{1/2}, -iq^{1/2})$ in \eqref{Pfaff}, we deduce that
\begin{align}
{}_3\phi_2\bigg(\genfrac{}{}{0pt}{}{q^{-n}, q^{n+1},q}{iq^{3/2},-iq^{3/2}};q,q\bigg)=\frac{(1+q)q^n}{1+q^{2n+1}}. \label{nu-proof-2}
\end{align}
The theorem follows after substituting \eqref{nu-proof-2} into \eqref{ab-3-6-proof}.
\end{proof}
\begin{corollary}
Identity \eqref{mock-3-6} holds. In addition, we have
\begin{align}
&\sum_{n=0}^\infty\frac{(-1)^nq^{n(n+1)/2}(q;q)_n}{(-q;q^2)_{n+1}}\nonumber \\
=& \sum_{n=0}^\infty\frac{1-q^{2n+1}}{1+q^{2n+1}}q^{n^2+n} \label{3-6-cor-1} \\
=&\overline{m}(q,q^2,-q^2), \label{3-6-cor-1-A} \\
&\sum_{n=0}^\infty\frac{q^{n(n+1)/2}(-q;q)_n}{(-q;q^2)_{n+1}}\nonumber \\
=&\frac{(-q;q)_\infty}{(q;q)_\infty}\sum_{n=0}^\infty(-1)^n\frac{1-q^{2n+1}}{1+q^{2n+1}}q^{n^2+n} \label{3-6-cor-2} \\
=&\frac{J_{4}^2}{J_{2}}, \label{3-6-cor-2-simplify} \\
&\sum_{n=0}^\infty \frac{q^{(n+1)^2}(-q;q^2)_n}{(-q^2;q^4)_{n+1}}\nonumber \\
=& q\frac{(-q;q^2)_{\infty}}{(q^2;q^2)_{\infty}}\sum_{n=-\infty}^{\infty}(-1)^n\frac{q^{2n^2+3n}}{1+q^{4n+2}}  \label{3-6-cor-3} \\
=&-m(-q,q^4,q), \label{3-6-cor-3-simplify} \\
&\sum_{n=0}^\infty \frac{q^{n+1}(-q;q)_{2n}}{(-q^2;q^4)_{n+1}} \nonumber \\ =&q\frac{(-q;q)_\infty}{(q;q)_\infty}\sum_{n=0}^\infty (-1)^n\frac{1-q^{2n+1}}{1+q^{4n+2}}q^{n^2+2n} \label{3-6-cor-4} \\
=&-\frac{1}{2}(1+i)m(i,q^2,q). \label{3-6-cor-4-simplify}
\end{align}
\end{corollary}
\begin{proof}
Taking $(a,b)\rightarrow (0,0)$, $(0,1)$ and $(0,-1)$ in Theorem \ref{thm-ab-3-6}, we obtain \eqref{mock-3-6}, \eqref{3-6-cor-1} and \eqref{3-6-cor-2}, respectively.

Taking $(q,a,b)\rightarrow (q^2, 0, -q)$ and $(q^2, -q, -1)$ in Theorem \ref{thm-ab-3-6}, we obtain \eqref{3-6-cor-3} and \eqref{3-6-cor-4}, respectively.

To get \eqref{3-6-cor-2-simplify}, we first rewrite the sum in \eqref{3-6-cor-2} as
\begin{align}
\sum_{n=0}^\infty(-1)^n\frac{1-q^{2n+1}}{1+q^{2n+1}}q^{n^2+n}=\sum_{n=-\infty}^\infty(-1)^n\frac{q^{n^2+n}}{1+q^{2n+1}}. \label{revise-rewrite-1}
\end{align}
Setting $(q,z)\rightarrow (q^2,-q)$ in \eqref{reciprocal-Jacobi}, we obtain
\begin{align}
\sum_{n=-\infty}^\infty(-1)^n\frac{q^{n^2+n}}{1+q^{2n+1}}=\frac{J_1^2J_4^2}{J_2^2}. \label{reciprocal-apply-3}
\end{align}
Substituting this identity into \eqref{revise-rewrite-1}, we get \eqref{3-6-cor-2-simplify} immediately.
\end{proof}
\begin{rem}
(1) In \eqref{6-2-cor-2-nu} we will provide the following Hecke-type series representation for $\nu^{(3)}(q)$:
\begin{align}
\nu^{(3)}(q)=\frac{(-q^2;q^2)_\infty}{(q^2;q^2)_\infty}\sum_{n=0}^\infty (1-q^{2n+1})q^{3n^2+2n}\sum_{j=-n}^n(-1)^jq^{-j^2}. \label{nu-Hecke-pre}
\end{align}
(2) The left sides of \eqref{3-6-cor-3} and \eqref{3-6-cor-4} are the eighth order mock theta functions $U_1^{(8)}(q)$ and $V_1^{(8)}(q)$, respectively. See \eqref{mock-8-6-defn} and \eqref{mock-8-8-defn}. From \cite[Eq.\ (5.40)]{Hickerson-Mortenson} we find
\begin{align}\label{HM-U1-m}
U_1^{(8)}(q)=-m(-q,q^4,-q^2).
\end{align}
Taking $(x,q,z_1,z_0)\rightarrow (-q,q^4,q,-q^2)$ in Lemma \ref{lem-m-minus}, we see that \eqref{HM-U1-m} is equivalent to \eqref{3-6-cor-3-simplify}.
 From \cite[Eq.\ (5.42)]{Hickerson-Mortenson} we find that
\begin{align}\label{HM-V1-m}
V_1^{(8)}(q)=-m(q^2,q^8,q).
\end{align}
This can be shown to be equivalent to \eqref{3-6-cor-4-simplify} by using Lemmas \ref{lem-m-minus} and \ref{lem-m-decompose}. We omit the details.\\
(3) If we set $x=i$ in \cite[Eq.\ (2.12)]{Mortenson-2014AIM}, we get \eqref{3-6-cor-4-simplify}.
\end{rem}

\subsection{Representations for $\rho^{(3)}(q)$}
\begin{theorem}\label{thm-ab-3-7}
For $\max \{|ab|, |aq^2|, |bq^2|\}<1$, we have
\begin{align}
&\frac{(q^2,ab;q^2)_\infty}{(aq^2,bq^2;q^2)_\infty}
{}_3\phi_2\bigg(\genfrac{}{}{0pt}{}{q^2/a,q^2/b,q^2}{\zeta_3q^3,\zeta_3^2q^3};q^2,ab\bigg)\\
\nonumber
=&(1+q+q^{2})\sum_{n=0}^{\infty}\frac{1-q^{4n+2}}{1+q^{2n+1}+q^{4n+2}}\frac{(q^2/a,q^2/b;q^2)_n}{(aq^2,bq^2;q^2)_n}(-ab)^nq^{n^2+n}.
\end{align}
\end{theorem}
\begin{proof}
Taking $(q,c,d,u,v)\rightarrow (q^2,\zeta_3q^3,\zeta_3^2q^3,q^2,q^2)$ in Theorem \ref{thm-main}, we deduce that
\begin{align}
&\frac{(q^2,ab;q^2)_\infty}{(aq^2,bq^2;q^2)_\infty}
{}_3\phi_2\bigg(\genfrac{}{}{0pt}{}{q^2/a,q^2/b,q^2}{\zeta_3q^3,\zeta_3^2q^3};q^2,ab\bigg) \nonumber\\
=&\sum_{n=0}^{\infty}(1-q^{4n+2})\frac{(q^2/a,q^2/b;q^2)_n}{(aq^2,bq^2;q^2)_n}(-ab)^nq^{n^2-n}
{}_3\phi_2\bigg(\genfrac{}{}{0pt}{}{q^{-2n},q^{2n+2},q^2}{\zeta_3q^3,\zeta_3^2q^3};q^2,q^2\bigg). \label{ab-3-7-proof}
\end{align}
Taking $(q,a,b,c) \rightarrow (q^2, q^2,\zeta_{3}q, \zeta_{3}^2q)$ in \eqref{Pfaff}, we obtain
\begin{align}
{}_{3}\phi_{2}\bigg(\genfrac{}{}{0pt}{}{q^{-2n}, q^{2n+2},q^2}{\zeta_{3}q^{3},\zeta_{3}^2q^{3}};q^2,q^2\bigg)=\frac{(1+q+q^2)q^{2n}}{1+q^{2n+1}+q^{4n+2}}. \label{rho-proof-2}
\end{align}
The theorem follows after substituting \eqref{rho-proof-2} into \eqref{ab-3-7-proof}.
\end{proof}
\begin{corollary}
Identity \eqref{mock-3-7} holds. In addition, we have
\begin{align}
&\sum_{n=0}^\infty\frac{(-1)^nq^{n(n+1)}(q;q^2)_{n+1}(q^2;q^2)_n}{(q^3;q^6)_{n+1}} \nonumber \\
=&\sum_{n=0}^\infty\frac{1-q^{4n+2}}{1+q^{2n+1}+q^{4n+2}}q^{2n^2+2n} \label{3-7-cor-1} \\
=&\frac{1}{1-\zeta_3}\left(\overline{m}(-\zeta_3^2q^2,q^4,-q^4)+\zeta_3q\overline{m}(-\zeta_3^2,q^4,-q^6)\right. \nonumber \\
&\quad \quad \left.-\zeta_3\overline{m}(-\zeta_3q^2,q^4,-q^4)-q\overline{m}(-\zeta_3,q^4,-q^6) \right), \label{3-7-cor-1-simplify} \\
&\sum_{n=0}^\infty\frac{(-1)^nq^{n^2+2n}(q;q^2)_n(q;q^2)_{n+1}}{(q^3;q^6)_{n+1}}\nonumber \\
=&\frac{(q;q^2)_\infty}{(q^2;q^2)_\infty}\sum_{n=0}^\infty\frac{1+q^{2n+1}}{1+q^{2n+1}+q^{4n+2}}q^{2n^2+3n} \label{3-7-cor-2} \\
=&\frac{1}{1-\zeta_3}q^{-1}\left(m(-\zeta_3q,q^4,-q)-\zeta_3m(-\zeta_{3}^2q,q^4,-q) \right). \label{3-7-cor-2-simplify}
\end{align}
\end{corollary}
\begin{proof}
Taking $(a,b)\rightarrow (0,0)$, $(0,1)$ and $(0,q)$ in Theorem \ref{thm-ab-3-7}, we obtain \eqref{mock-3-7}, \eqref{3-7-cor-1} and \eqref{3-7-cor-2}, respectively.
\end{proof}

\section{Mock theta functions of order 5}\label{sec-order-5}
In his last letter to Hardy, Ramanujan gave ten mock theta functions of order 5. They are defined as
\begin{align}
&f_0^{(5)}(q):=\sum_{n= 0}^{\infty} \frac{q^{n^2}}{(-q;q)_{n}}, \label{mock-5-1-defn} \\
&\phi_{0}^{(5)}(q):=\sum_{n=0}^{\infty}q^{n^2}(-q;q^2)_{n}, \label{mock-5-2-defn} \\
&\psi_0^{(5)}(q):=\sum_{n=0}^\infty q^{(n+2)(n+1)/2}(-q;q)_n, \label{mock-5-3-defn} \\
&F_0^{(5)}(q):=\sum_{n= 0}^{\infty} \frac{q^{2n^2}}{(q;q^2)_{n}}, \label{mock-5-4-defn} \\
&f_1^{(5)}(q):=\sum_{n=0}^{\infty} \frac{q^{n(n+1)}}{(-q;q)_{n}}, \label{mock-5-6-defn} \\
&\phi_1^{(5)}(q):=\sum_{n=0}^{\infty} q^{(n+1)^2}(-q;q^2)_{n}, \label{mock-5-7-defn} \\
&\psi_1^{(5)}(q):=\sum_{n=0}^{\infty}q^{\binom{n+1}{2}}(-q;q)_{n}, \label{mock-5-8-defn} \\
&F_1^{(5)}(q):=\sum_{n=0}^{\infty} \frac{q^{2n(n+1)}}{(q;q^2)_{n+1}}, \label{mock-5-9-defn} \\
&\chi_0^{(5)}(q):=\sum_{n=0}^{\infty} \frac{q^n}{(q^{n+1};q)_{n}}, \label{mock-5-5-defn} \\
&\chi_1^{(5)}(q)=\sum_{n=0}^{\infty} \frac{q^n}{(q^{n+1};q)_{n+1}}.  \label{mock-5-10}
\end{align}
Thanks to the work of Andrews \cite{Andrews-TAMS}, except for $\chi_0^{(5)}(q)$ and $\chi_1^{(5)}(q)$, we know that there are single Hecke-type series representations for mock theta functions of order 5.
\begin{theorem}\label{thm-ord-5}
We have
\begin{align}
&f_0^{(5)}(q)=\frac{1}{(q;q)_{\infty}}\sum_{n=0}^{\infty}\sum_{j=-n}^{n}(-1)^j(1-q^{4n+2})q^{n(5n+1)/2-j^2}, \label{mock-5-1} \\
&\phi_{0}^{(5)}(q)=\frac{(-q;q^2)_{\infty}}{(q^2;q^2)_{\infty}}\sum_{n=0}^{\infty}\sum_{|j|\leq n}(-1)^jq^{5n^2+2n-3j^2-j}(1-q^{6n+3}), \label{mock-5-2} \\
&\psi_0^{(5)}(q)=-\frac{(-q;q)_\infty}{(q;q)_\infty}\sum_{n=1}^\infty \sum_{j=-n}^{n-1}(1-q^n)(-1)^jq^{n(5n-1)/2-j(3j+1)/2}, \label{mock-5-3} \\
&F_0^{(5)}(q)=\frac{1}{(q^2;q^2)_{\infty}}\sum_{n=0}^{\infty}\sum_{j=0}^{2n}(-1)^nq^{5n^2+2n-\binom{j+1}{2}} (1+q^{6n+3}), \label{mock-5-4} \\
&f_1^{(5)}(q)=\frac{1}{(q;q)_{\infty}}\sum_{n=0}^{\infty}\sum_{|j|\leq n}(-1)^jq^{n(5n+3)/2-j^2}(1-q^{2n+1}), \label{mock-5-6} \\
&\phi_1^{(5)}(q)=q\frac{(-q;q^2)_{\infty}}{(q^2;q^2)_{\infty}}\sum_{n\geq 0}\sum_{|j|\leq n}(-1)^jq^{5n^2+4n-3j^2-j}(1-q^{2n+1}), \label{mock-5-7} \\
&\psi_1^{(5)}(q)=\frac{(-q;q)_{\infty}}{(q;q)_{\infty}}\sum_{n\geq 0}\sum_{|j|\leq n}(-1)^jq^{n(5n+3)/2-j(3j+1)/2}(1-q^{2n+1}), \label{mock-5-8} \\
&F_1^{(5)}(q)=\frac{1}{(q^2;q^2)_{\infty}}\sum_{n\geq 0}\sum_{j=0}^{2n}(-1)^nq^{5n^2+4n-\binom{j+1}{2}}(1+q^{2n+1}). \label{mock-5-9}
\end{align}
\end{theorem}
As for the remaining two functions, from \cite{Watson-2} we know that
\begin{align}
\chi_{0}^{(5)}(q)=2F_0^{(5)}(q)-\phi_0^{(5)}(-q), \quad \chi_1^{(5)}(q)=2F_1^{(5)}(q)+q^{-1}\phi_1^{(5)}(-q).
\end{align}
This means that $\chi_0^{(5)}(q)$ and $\chi_1^{(5)}(q)$ can be expressed as combinations of two Hecke-type series. In addition, Zagier \cite{Zagier} stated indefinite theta series identities for these two functions. Using Bailey pairs,  Zwegers \cite{Zwegers-Rama} found triple sum identities for $\chi_0^{(5)}(q)$ and $\chi_1^{(5)}(q)$. In a recent work, Garvan \cite{Garvan-arXiv} gave some Hecke-type series representations for them.

We will provide new proofs for the Hecke-type series representations in Theorem \ref{thm-ord-5}.
\subsection{Representations for $f_0^{(5)}(q)$}
\begin{theorem}\label{thm-ab-5-1}
For $\max \{|ab/q|, |a|, |b|\}<1$, we have
\begin{align}
&\frac{(q,ab/q;q)_\infty}{(a,b;q)_\infty}
{}_3\phi_2\bigg(\genfrac{}{}{0pt}{}{q/a,q/b,q}{0,-q};q,ab/q\bigg) \nonumber\\
=&1+\sum_{n=1}^{\infty}\frac{(q/a,q/b;q)_n}{(a,b;q)_n}(-ab)^nq^{\frac{n^2-3n}{2}}
\left(2q^n+(-1)^{n-1}(1-q^n)\sum_{|j|<n}(-1)^jq^{n^2-j^2}\right).
\end{align}
\end{theorem}
\begin{proof}
Taking $(c,d,u,v)\rightarrow (0,-q,1,q)$ in Theorem \ref{thm-main}, we deduce that
\begin{align}
&\frac{(q,ab/q;q)_\infty}{(a,b;q)_\infty}
{}_3\phi_2\bigg(\genfrac{}{}{0pt}{}{q/a,q/b,q}{0,-q};q,ab/q\bigg) \nonumber\\
=&1+\sum_{n=1}^{\infty}(1+q^{n})\frac{(q/a,q/b;q)_n}{(a,b;q)_n}(-ab)^nq^{(n^2-3n)/2}
{}_3\phi_2\bigg(\genfrac{}{}{0pt}{}{q^{-n},q^{n},q}{0,-q};q,q\bigg). \label{ab-5-1-proof}
\end{align}
Setting $c=-q$ in Lemma \ref{lem-limit-2}, we deduce that
\begin{align}
&{}_3\phi_2\bigg(\genfrac{}{}{0pt}{}{q^{-n},q^{n},q}{0,-q};q,q\bigg) \nonumber\\
=&~ 2(-1)^nq^{n^2}\frac{1-q^n}{1+q^n}\left(-\frac{3-q}{2(1-q)}+\sum_{j=2}^n\frac{1-q^{2j-1}}{(1-q^{j-1})(1-q^j)}(-1)^jq^{j-j^2} \right). \label{revise-5-1-proof-1}
\end{align}
We have
\begin{align}
&\sum_{j=2}^n\frac{1-q^{2j-1}}{(1-q^{j-1})(1-q^j)}(-1)^jq^{j-j^2} \nonumber \\
=&~ \sum_{j=2}^n\frac{(1-q^{j-1})+q^{j-1}(1-q^j)}{(1-q^{j-1})(1-q^j)}(-1)^jq^{j-j^2} \nonumber \\
=&~ \sum_{j=2}^n \frac{(-1)^jq^{j-j^2}}{1-q^j}+\sum_{j=2}^n \frac{(-1)^jq^{2j-1-j^2}}{1-q^{j-1}}  \quad \textrm{(replace $j$ by $j+1$ in the second sum)}\nonumber \\
=&~ \sum_{j=2}^n \frac{(-1)^jq^{j-j^2}}{1-q^j}+\sum_{j=1}^{n-1} \frac{(-1)^{j+1}q^{-j^2}}{1-q^{j}} \nonumber \\
=&~ \sum_{j=1}^n \frac{(-1)^jq^{-j^2}(q^j-1)}{1-q^j} +\frac{1}{1-q}+\frac{(-1)^nq^{-n^2}}{1-q^n} \nonumber \\
=&~ -\sum_{j=1}^n(-1)^jq^{-j^2} + \frac{1}{1-q}+\frac{(-1)^nq^{-n^2}}{1-q^n}. \label{revise-5-1-proof-2}
\end{align}
Substituting \eqref{revise-5-1-proof-2} into \eqref{revise-5-1-proof-1}, we deduce that
\begin{align}
&{}_3\phi_2\bigg(\genfrac{}{}{0pt}{}{q^{-n},q^{n},q}{0,-q};q,q\bigg) \nonumber\\
=&~ \frac{2}{1+q^n}-(-1)^nq^{n^2}\frac{1-q^n}{1+q^n}\sum_{j=-n}^n (-1)^jq^{-j^2} \nonumber \\
=&~ \frac{2q^n}{1+q^n}+(-1)^{n-1}\frac{1-q^n}{1+q^n} \sum_{|j|<n}(-1)^jq^{n^2-j^2}. \label{ab-5-1-proof-2}
\end{align}
The theorem follows after substituting this identity into \eqref{ab-5-1-proof}.
\end{proof}
\begin{rem}
Identity \eqref{ab-5-1-proof-2} appears as \cite[Eq.\ (3.10)]{Liu2013IJNT}. Our proof for this identity is new.
\end{rem}

\begin{corollary}
Identity \eqref{mock-5-1} holds. In addition, we have
\begin{align}
&\sum_{n=0}^\infty\frac{q^{n^2}(-q;q^2)_n}{(-q^2;q^2)_n} \nonumber \\
=& \frac{(-q;q^2)_{\infty}}{(q^2;q^2)_{\infty}}\sum_{n=0}^{\infty}q^{4n^2+n}(1-q^{6n+3})\sum_{j=-n}^{n}(-1)^jq^{-2j^2} \label{5-1-cor-1} \\
=&\frac{J_2}{J_1J_4}\left(f_{1,3,1}(q^3,q^3,q^4)+q^5f_{1,3,1}(q^{11},q^{11},q^4) \right)\label{5-1-cor-1-H-pre} \\
=&\frac{J_2}{J_1J_4} f_{1,3,1}(q,-q,-q) \label{5-1-cor-1-H} \\
=&2m(-q^3,q^8,-1)+q\frac{\overline{J}_{1,8}J_{2,8}^2}{J_{3,8}^2}. \label{5-1-cor-1-A}
\end{align}
\end{corollary}
\begin{proof}
Taking $(a,b)\rightarrow (0,0)$  in Theorem \ref{thm-ab-5-1}, we deduce that
\begin{align}
f_0^{(5)}(q)&=\frac{1}{(q;q)_\infty}\left(1+2\sum_{n=1}^\infty(-1)^nq^{(3n^2+n)/2}-\sum_{n=1}^\infty(1-q^n)q^{(5n^2-n)/2}\sum_{|j|<n}(-1)^jq^{-j^2} \right) \nonumber \\
&=\frac{1}{(q;q)_\infty}\Big(1+2\sum_{n=1}^\infty(-1)^nq^{(3n^2+n)/2}+\sum_{n=1}^\infty q^{(5n^2+n)/2}\sum_{j=-n}^{n}(-1)^jq^{-j^2}\nonumber \\
&-2\sum_{n=1}^\infty q^{(5n^2+n)/2}\cdot (-1)^nq^{-n^2} -\sum_{n=0}^\infty q^{\big(5(n+1)^2-(n+1)\big)/2}\sum_{j=-n}^{n}(-1)^jq^{-j^2}\Big).
\end{align}
After simplifications, we obtain \eqref{mock-5-1}.

Similarly, taking $(q, a, b)\rightarrow (q^2, 0,-q)$ in Theorem \ref{thm-ab-5-1}, after rearrangements and simplifications, we obtain \eqref{5-1-cor-1}.

To get \eqref{5-1-cor-1-H} from \eqref{5-1-cor-1-H-pre}, we apply \eqref{f-id-1} with $(a,b,c,x,y,q)\rightarrow (1,3,1,q,-q,-q)$.

The expression \eqref{5-1-cor-1-A}  follows from \eqref{5-1-cor-1-H} and Theorem \ref{thm-fg-2} with $n=1$.
\end{proof}
\begin{rem}
(1) Theorem \ref{thm-ab-5-1} appears as Proposition 6.14 in \cite{Liu2013Rama} in slightly different form. Identity \eqref{mock-5-1} was also obtained by Liu as Proposition 6.15 in \cite{Liu2013Rama}.\\
(2) The left side of \eqref{5-1-cor-1} is the eighth order mock theta function $S_0^{(8)}(q)$. See \eqref{mock-8-1-defn}. The expression \eqref{5-1-cor-1-A} appears in \cite[Eq.\ (5.35)]{Hickerson-Mortenson}.\\
(3) Identity \eqref{5-1-cor-1}  was also obtained by Cui, Gu and Hao \cite[Eq.\ (1.9)]{CGH} using Bailey pairs.
\end{rem}
\subsection{Representations for $\phi_0^{(5)}(q)$ and  $\psi_0^{(5)}(q)$}
\begin{theorem}\label{thm-ab-5-2}
For $\max \{|ab/q|, |a|, |b|\}<1$, we have
\begin{align}
&\frac{(q,ab/q;q)_{\infty}}{(a,b;q)_{\infty}}\sum_{n=0}^{\infty}(q/a,q/b;q)_n(ab/q)^n \nonumber \\
=&~ \sum_{n=0}^{\infty}\frac{(q/a,q/b;q)_n(ab)^nq^{2n^2}}{(a,b;q)_n}\left( 1-\frac{(a-q^{n+1})(b-q^{n+1})q^{2n}}{(1-aq^n)(1-bq^n)} \right)\sum_{|j|\leq n}(-1)^jq^{-j(3j+1)/2} \label{Liu-Prop6.10-original} \\
=&~ 1+\sum_{n=1}^\infty(1+q^n)\frac{(q/a,q/b;q)_n}{(a,b;q)_n}(-ab)^nq^{(n^2-3n)/2}\nonumber \\
&\quad \times \left(q^n+(-1)^{n-1}(1-q^n)q^{(3n^2-n)/2}\sum_{|j|<n}(-1)^jq^{-j(3j+1)/2} \right). \label{ab-5-2-new}
\end{align}
\end{theorem}
Identity \eqref{Liu-Prop6.10-original} appears as \cite[Proposition 6.10]{Liu2013Rama}. However, its equivalent form \eqref{ab-5-2-new} was not presented there. Here we follow Liu's arguments to prove \eqref{ab-5-2-new}.
\begin{proof}
Taking $(c,d,u,v)\rightarrow (0,0, 1, q)$ in Theorem \ref{thm-main}, we deduce that
\begin{align}
&\frac{(q,ab/q;q)_{\infty}}{(a,b;q)_{\infty}}\sum_{n=0}^{\infty}(q/a,q/b;q)_n(ab/q)^n \nonumber \\
=&~ 1+\sum_{n=1}^\infty (1+q^n)\frac{(q/a,q/b;q)_n}{(a,b;q)_n}(-ab)^nq^{(n^2-3n)/2}{}_3\phi_2\bigg(\genfrac{}{}{0pt}{}{q^{-n},q^{n},q}{0,0};q,q\bigg). \label{5-2-add-proof}
\end{align}
Taking $c\rightarrow \infty$ in Lemma \ref{lem-limit-2}, we deduce that
\begin{align}
&{}_3\phi_2\bigg(\genfrac{}{}{0pt}{}{q^{-n},q^{n},q}{0,0};q,q\bigg) \nonumber \\
=&~(-1)^n(1-q^n)q^{(3n^2-n)/2}\left(\frac{q-2}{1-q}+\sum_{j=2}^n\frac{1-q^{2j-1}}{(1-q^{j-1})(1-q^{j})}(-1)^jq^{-(3j^2-3j)/2} \right). \label{revise-5-2-proof-1}
\end{align}
We have
\begin{align}
&\sum_{j=2}^n\frac{1-q^{2j-1}}{(1-q^{j-1})(1-q^{j})}(-1)^jq^{-(3j^2-3j)/2} \nonumber \\
=&~\sum_{j=2}^n\frac{(1-q^{j-1})+q^{j-1}(1-q^{j})}{(1-q^{j-1})(1-q^{j})}(-1)^jq^{-(3j^2-3j)/2}  \nonumber \\
=&~\sum_{j=2}^n\frac{(-1)^jq^{-(3j^2-3j)/2}}{1-q^j}+\sum_{j=1}^{n-1}\frac{(-1)^{j+1}q^{-(3j^2+j)/2}}{1-q^j} \nonumber \\
=&~ -\sum_{j=1}^n\frac{(-1)^j(1-q^{2j})}{1-q^j}q^{-(3j^2+j)/2}+\frac{1}{1-q}+\frac{(-1)^nq^{-(3n^2+n)/2}}{1-q^n} \nonumber \\
=&~ -\sum_{j=1}^n (-1)^j(1+q^j)q^{-(3j^2+j)/2}+\frac{1}{1-q}+\frac{(-1)^nq^{-(3n^2+n)/2}}{1-q^n}. \label{revise-5-2-proof-2}
\end{align}
Substituting \eqref{revise-5-2-proof-2} into \eqref{revise-5-2-proof-1}, we obtain
\begin{align}
&{}_3\phi_2\bigg(\genfrac{}{}{0pt}{}{q^{-n},q^{n},q}{0,0};q,q\bigg) \nonumber \\
=&~q^{-n}-(-1)^nq^{(3n^2-n)/2}(1-q^n)\sum_{j=-n}^n(-1)^jq^{-(3j^2+j)/2} \nonumber \\
=&~q^n+(-1)^{n-1}(1-q^n)q^{(3n^2-n)/2}\sum_{|j|<n}(-1)^jq^{-j(3j+1)/2}. \label{revise-5-2-proof-3}
\end{align}

Substituting \eqref{revise-5-2-proof-3} into \eqref{5-2-add-proof}, we arrive at \eqref{ab-5-2-new}.
\end{proof}
\begin{rem}
We may also use the following arguments to deduce \eqref{revise-5-2-proof-3}.

From \cite[Lemma 4.1]{Liu2013Rama} we find
\begin{align}
\sum_{k=0}^{n}\frac{(q^{-n},aq^n;q)_kq^k}{(cq;q)_k}=a^{n}q^{n^2}\frac{(q;q)_n}{(cq;q)_n}\sum_{j=0}^n\frac{(c;q)_ja^{-j}q^{j(1-n)}}{(q;q)_j}. \label{Liu-Lemma4.1}
\end{align}
Setting $(a,c)=(1,0)$ in \eqref{Liu-Lemma4.1}, we obtain
\begin{align}
\sum_{k=0}^{n}(q^{-n},q^{n};q)_kq^k=q^{n^2}(q;q)_n\sum_{j=0}^{n}\frac{q^{j(1-n)}}{(q;q)_j}. \label{5-psi-2}
\end{align}
From \cite[(6.2)]{Liu2013Rama} we find
\begin{align}
\sum_{j=0}^{n}\frac{q^{-jn}}{(q;q)_j}=(-1)^n\frac{q^{n(n+1)/2}}{(q;q)_n}\sum_{j=-n}^{n}(-1)^jq^{-j(3j+1)/2}. \label{5-psi-3}
\end{align}
Substituting \eqref{5-psi-3} with $n$ replaced by $n-1$ into \eqref{5-psi-2}, we obtain
\begin{align*}
\sum_{k=0}^{n}(q^{-n},q^{n};q)_kq^k=q^{n}+(-1)^{n-1}(1-q^n)q^{(3n^2-n)/2}\sum_{|j|<n}(-1)^jq^{-j(3j+1)/2},
\end{align*}
and this is \eqref{revise-5-2-proof-3}.
\end{rem}

\begin{corollary}
Identities \eqref{mock-5-2} and \eqref{mock-5-3} hold.
\end{corollary}
\begin{proof}
Taking $(q,a,b)\rightarrow (q^2, 0, -q)$ in \eqref{Liu-Prop6.10-original}, we obtain \eqref{mock-5-2}.

From \eqref{mock-5-3-defn} we have
\begin{align}
1+2\psi_{0}^{(5)}(q)=\sum_{n=0}^{\infty} (-1;q)_nq^{\binom{n+1}{2}}. \label{mock-5-3-new-defn}
\end{align}
Taking $(a,b)\rightarrow (0, -q)$ in \eqref{ab-5-2-new}, we obtain
\begin{align}
1+2\psi_{0}^{(5)}(q)=&~ -\frac{(-q;q)_{\infty}}{(q;q)_{\infty}}\Big(1+2\sum_{n=1}^{\infty}(-1)^nq^{n^2+n} \nonumber \\
& -2\sum_{n=1}^{\infty}\sum_{|j|<n}(1-q^n)(-1)^jq^{n(5n-1)/2-j(3j+1)/2}\Big).
\end{align}
After rearrangements and simplifications, we arrive at  \eqref{mock-5-3}.
\end{proof}
For more identities deduced from Theorem \ref{thm-ab-5-2}, see \cite[Propositions 6.11, 6.12, 6.13]{Liu2013Rama}.

\subsection{Representations for $F_0^{(5)}(q)$}
\begin{theorem}\label{thm-ab-5-4}
For $\max\{|ab/q^2|, |a|, |b|\}<1$, we have
\begin{align}
&\frac{(q^2,ab/q^2;q^2)_\infty}{(a,b;q^2)_\infty}
{}_3\phi_2\bigg(\genfrac{}{}{0pt}{}{q^2/a,q^2/b,q^2}{0,q};q^2,ab/q^2\bigg) \nonumber\\
=&~ 1+\sum_{n=1}^{\infty}\frac{(q^2/a,q^2/b;q^2)_n}{(a,b;q^2)_n}(-ab)^nq^{3n^2-4n}\left(q^{4n}\sum_{j=0}^{2n}q^{-\binom{j+1}{2}}-\sum_{j=0}^{2n-2}q^{-\binom{j+1}{2}}\right).
\end{align}
\end{theorem}
\begin{proof}
Taking $(q, c, d, u, v)\rightarrow (q^2, 0, q, 1, q^2)$ in Theorem \ref{thm-main}, we deduce that
\begin{align}
&\frac{(q^2,ab/q^2;q^2)_\infty}{(a,b;q^2)_\infty}
{}_3\phi_2\bigg(\genfrac{}{}{0pt}{}{q^2/a,q^2/b,q^2}{0,q};q^2,ab/q^2\bigg) \nonumber\\
=&1+\sum_{n=1}^{\infty}(1+q^{2n})\frac{(q^2/a,q^2/b;q^2)_n}{(a,b;q^2)_n}(-ab)^nq^{n^2-3n}
{}_3\phi_2\bigg(\genfrac{}{}{0pt}{}{q^{-2n},q^{2n},q^2}{0,q};q^2,q^2\bigg). \label{add-5-4-proof}
\end{align}
Taking $(q,c)\rightarrow (q^2, q^3)$ in Lemma \ref{lem-limit-2}, we deduce that
\begin{align}
&{}_3\phi_2\bigg(\genfrac{}{}{0pt}{}{q^{-2n},q^{2n},q^2}{0,q};q^2,q^2\bigg)\nonumber \\
=&~ (1-q^{2n})q^{2n^2-n}\left(-\frac{1-q-q^2}{1-q^2}+\sum_{j=2}^{n}\frac{(1+q^{2j-1})(1-q^{-1})}{(1-q^{2j})(1-q^{2j-2})}q^{3j-2j^2} \right). \label{revise-5-4-proof-1}
\end{align}
Note that
\begin{align}
(1-q^{-1})(1+q^{2j-1})=(1-q^{2j-2})-q^{-1}(1-q^{2j}).
\end{align}
We have
\begin{align}
&\sum_{j=2}^{n}\frac{(1+q^{2j-1})(1-q^{-1})}{(1-q^{2j})(1-q^{2j-2})}q^{3j-2j^2}\nonumber \\
=&~\sum_{j=2}^n\left(\frac{1}{1-q^{2j}}-\frac{q^{-1}}{1-q^{2j-2}} \right)q^{3j-2j^2} \nonumber \\
=&~\sum_{j=2}^n\frac{q^{3j-2j^2}}{1-q^{2j}}-\sum_{j=1}^{n-1}\frac{q^{-2j^2-j}}{1-q^{2j}} \nonumber \\
=&~\sum_{j=1}^n\frac{q^{-2j^2-j}(q^{4j}-1)}{1-q^{2j}}+\frac{q^{-2n^2-n}}{1-q^{2n}}-\frac{q}{1-q^2} \nonumber \\
=&~\frac{q^{-2n^2-n}}{1-q^{2n}}-\frac{q}{1-q^2}-\sum_{j=1}^nq^{-2j^2-j}(1+q^{2j}). \label{revise-5-4-proof-2}
\end{align}
Substituting \eqref{revise-5-4-proof-2} into \eqref{revise-5-4-proof-1}, we deduce that
\begin{align}
{}_3\phi_2\bigg(\genfrac{}{}{0pt}{}{q^{-2n},q^{2n},q^2}{0,q};q^2,q^2\bigg)&=q^{-2n}-(1-q^{2n})q^{2n^2-n}\left(1+\sum_{j=1}^nq^{-2j^2-j}(1+q^{2j})\right) \nonumber \\
&=q^{-2n}-(1-q^{2n})q^{2n^2-n}\sum_{j=-n}^nq^{-j(2j+1)}. \label{revise-5-4-proof-3}
\end{align}
Now we substitute \eqref{revise-5-4-proof-3} back into \eqref{add-5-4-proof}. Note that
\begin{align}\label{revise-j-convert}
\sum_{j=-n}^nq^{-j(2j+1)}=\sum_{j=0}^nq^{-2j(2j+1)/2}+\sum_{j=1}^nq^{-((2j-1)\cdot 2j)/2}=\sum_{j=0}^{2n}q^{-j(j+1)/2}
\end{align}
and
\begin{align}
&(1+q^{2n})q^{n^2-3n}\left(q^{-2n}-(1-q^{2n})q^{2n^2-n}\sum_{j=-n}^nq^{-j(2j+1)}\right) \nonumber \\ =&~q^{3n^2-4n}\left(q^{4n}\sum_{j=0}^{2n}q^{-\binom{j+1}{2}}-\sum_{j=0}^{2n-2}q^{-\binom{j+1}{2}}\right).
\end{align}
We eventually arrive at the desired identity.
\end{proof}
\begin{rem}
We may also use different arguments based on a result of Andrews.
Replacing $q$ by $q^2$ and setting $(a,c)=(1, q^{-1})$ in \eqref{Liu-Lemma4.1}, we deduce that
\begin{align}
\sum_{k=0}^{n}\frac{(q^{-2n};q^2)_k(q^{2n};q^2)_kq^{2k}}{(q;q^2)_k}=q^{2n^2}\frac{(q^2;q^2)_n}{(q;q^2)_n}\sum_{j=0}^{n}\frac{(q^{-1};q^2)_jq^{2j(1-n)}}{(q^2;q^2)_j}.\label{mock-5-F0-2}
\end{align}
The sum on the right side has been evaluated by Andrews \cite[Eq.\ (5.8)]{Andrews-TAMS}. Indeed, from \cite[Eq.\ (5.8)]{Andrews-TAMS} we deduce that
\begin{align}
q^{n}(1+q^{2n})\frac{(q^2;q^2)_n}{(q;q^2)_n}\sum_{j=0}^{n}\frac{(q^{-1};q^2)_jq^{2j(1-n)}}{(q^2;q^2)_j}
=q^{4n}\sum_{j=0}^{2n}q^{-\binom{j+1}{2}}-\sum_{j=0}^{2n-2}q^{-\binom{j+1}{2}}.\label{mock-5-F0-3}
\end{align}
The theorem follows after substituting \eqref{mock-5-F0-3} into \eqref{mock-5-F0-2} and then applying to \eqref{add-5-4-proof}.
\end{rem}

\begin{corollary}
Identity \eqref{mock-5-4} holds. In addition, we have
\begin{align}
&\sum_{n=0}^\infty\frac{q^{n^2}(-q;q^2)_n}{(q;q^2)_n}\nonumber \\
=&\frac{(-q;q^2)_\infty}{(q^2;q^2)_\infty}\sum_{n=0}^\infty(-1)^nq^{4n^2+2n}(1+q^{4n+2})\sum_{j=0}^{2n}q^{-j(j+1)/2} \label{5-4-cor-1} \\
=&\frac{J_2}{J_1J_4}\left(f_{1,3,1}(q^3,q^5,q^4)-q^6f_{1,3,1}(q^{11},q^{13},q^4) \right) \label{5-4-cor-1-H} \\
=&m(q^4,q^8,q^{-1/2})-q^{-1}m(1,q^8,q^{1/2})-q^{-1/2}\frac{J_{\frac{3}{2},8}^2J_{\frac{5}{2},8}J_{4,8}J_{8}^3}{J_{\frac{1}{2},8}^2J_{1,8}J_{2,8}J_{3,8}J_{\frac{7}{2},8}}, \label{5-4-cor-1-A} \\
&\sum_{n=0}^\infty \frac{q^{n^2+n}(-1;q^2)_n}{(q;q^2)_n}=\frac{(-q^2;q^2)_\infty}{(q^2;q^2)_\infty} \nonumber \\
&\quad \quad \times \left(1+2\sum_{n=1}^\infty (-1)^nq^{2n^2-2n}-2\sum_{n=1}^\infty (-1)^nq^{4n^2-n}(1-q^{2n})\sum_{j=0}^{2n}q^{-j(j+1)/2} \right) \label{5-4-cor-2} \\
=&1-2\frac{J_4}{J_2^2}\left(f_{1,3,1}(1,q^2,q^4)-q^3f_{1,3,1}(q^8,q^{10},q^4) \right) \label{5-4-cor-2-H} \\
=&1-2m(q^3,q^8,q^{1/2})-2q^{1/2}\frac{J_{\frac{3}{2},4}J_{4}J_{8}^2}{J_{\frac{1}{2},4}J_{3,8}J_2}. \label{5-4-cor-2-A}
\end{align}
\end{corollary}
\begin{proof}
Taking $(a,b)\rightarrow (0,0)$ in Theorem \ref{thm-ab-5-4},  we deduce that
\begin{align}
&F_0^{(5)}(q)\nonumber \\
=&~\frac{1}{(q^2;q^2)_\infty}\left(\sum_{n=0}^{\infty}(-1)^nq^{5n^2+2n}\sum_{j=0}^{2n}q^{-\binom{j+1}{2}}-\sum_{n=1}^{\infty}(-1)^nq^{5n^2-2n}\sum_{j=0}^{2n-2}q^{-\binom{j+1}{2}}  \right) \nonumber \\
=&~\frac{1}{(q^2;q^2)_\infty}\left(\sum_{n=0}^{\infty}(-1)^nq^{5n^2+2n}\sum_{j=0}^{2n}q^{-\binom{j+1}{2}}-\sum_{n=0}^{\infty}(-1)^{n+1}q^{5n^2+8n+3}\sum_{j=0}^{2n}q^{-\binom{j+1}{2}}  \right) \nonumber \\
=&~\frac{1}{(q^2;q^2)_\infty}\sum_{n=0}^{\infty}(-1)^nq^{5n^2+2n}\sum_{j=0}^{2n}q^{-\binom{j+1}{2}}(1+q^{6n+3}).
\end{align}
This proves \eqref{mock-5-4}.

Similarly, taking $(a,b)\rightarrow (0,-q)$ and $(0,-q^2)$ in Theorem \ref{thm-ab-5-4}, after rearrangements and simplifications, we obtain \eqref{5-4-cor-1} and \eqref{5-4-cor-2}, respectively.

Substituting \eqref{revise-j-convert} into \eqref{5-4-cor-1} and \eqref{5-4-cor-2} and applying routine arguments, we get \eqref{5-4-cor-1-H} and \eqref{5-4-cor-2-H}, respectively.

The expression \eqref{5-4-cor-1-A} (resp. \eqref{5-4-cor-2-A}) follows from \eqref{5-4-cor-1-H} (resp.\ \eqref{5-4-cor-2-H}) and Theorem \ref{thm-fg-2} with $n=1$. Here we only give the details for \eqref{5-4-cor-1-A}.

By Theorem \ref{thm-fg-2} we deduce that
\begin{align}
&f_{1,3,1}(q^3,q^5,q^4)-q^6f_{1,3,1}(q^{11},q^{13},q^4) \nonumber \\\
=& J_{1,4}\left(m(-q^{16},q^{32},q^{-2})-q^{-1}m(-q^8,q^{32},q^2)-q^{-4}m(-1,q^{32},q^{-2})\right.\nonumber \\
&\left. +q^{-9}m(-q^{-8},q^{32},q^2)  \right)-\Theta_{1,2}(q^3,q^5,q^4)+q^6\Theta_{1,2}(q^{11},q^{13},q^4). \label{revise-f131-proof-1}
\end{align}
Setting $(x,q,z)\rightarrow (q^4,q^8,q^{-1/2})$ in Lemma \ref{lem-m-add}, we get
\begin{align}
m(-q^{16},q^{32},q^{-2})-q^{-4}m(-1,q^{32},q^{-2})=m(q^4,q^8,q^{-1/2})+ q^{-4}\frac{J_{16}J_{32}\overline{J}_{3,8}\overline{J}_{\frac{5}{2},8}}{J_{\frac{7}{2},8}J_{-2,32}\overline{J}_{14,16}}. \label{revise-f131-proof-2}
\end{align}
Setting $(x,q,z)\rightarrow (1,q^8,q^{1/2})$ in Lemma \ref{lem-m-add}, we get
\begin{align}
m(-q^8,q^{32},q^2)-q^{-8}m(-q^{-8},q^{32},q^2)=m(1,q^8,q^{1/2})+\frac{J_{16}J_{32}\overline{J}_{1,8}\overline{J}_{\frac{3}{2},8}}{J_{\frac{1}{2},8}J_{2,32}\overline{J}_{10,16}}  . \label{revise-f131-proof-3}
\end{align}
Now substituting \eqref{revise-f131-proof-2} and \eqref{revise-f131-proof-3} into \eqref{revise-f131-proof-1}, we get the desired $m(x,q,z)$'s. Showing that the remaining infinite products add up to the desired one is a routine exercise, which may also be verified directly using the computer package and method in \cite{Garvan-Liang}.
\end{proof}
\begin{rem}
(1) The left side of \eqref{5-4-cor-1} is essentially the eighth order mock theta function $V_0^{(8)}(q)$. See \eqref{mock-8-7-defn}. From \cite[Eq.\ (5.41)]{Hickerson-Mortenson} we find
\begin{align}\label{HM-V08-m}
V_0^{(8)}(q)=-2q^{-1}m(1,q^8,q)-\frac{\overline{J}_{1,4}^2}{J_{2,8}}.
\end{align}
Using Lemma \ref{lem-m-minus} and the following fact \cite[Eq.\ (3.3)]{Hickerson-Mortenson}:
\begin{align}\label{eq-fact}
m(q,q^2,-1)=\frac{1}{2},
\end{align}
we can show that \eqref{HM-V08-m} is equivalent to \eqref{5-4-cor-1-A}. We omit the details.
 \\
(2) It is not difficult to see that the left side of \eqref{5-4-cor-2} is equal to $1+2T_0^{(8)}(-q)$, where $T_0^{(8)}(q)$ is an eighth order mock theta function. See \eqref{mock-8-3-defn}. From \cite[Eq.\ (5.37)]{Hickerson-Mortenson} we know that
\begin{align}\label{HM-T0-m}
T_0^{(8)}(q)=-m(-q^3,q^8,q^2).
\end{align}
Applying Lemma \ref{lem-m-minus} with $(x,q,z_1,z_0)\rightarrow (q^3,q^8,q^{1/2},q^2)$, we see that \eqref{HM-T0-m} can be deduced from \eqref{5-4-cor-2-A} and vice versa.
\end{rem}

\subsection{Representations for $f_1^{(5)}(q)$}
\begin{theorem}\label{meq:2.1}
(Cf.\ \cite[Proposition 6.6]{Liu2013Rama}.) For  $\max \{|ab|, |aq|,|bq|\}<1$,  we have
\begin{align}
&\frac{(q,ab;q)_\infty}{(aq,bq;q)_\infty}{}_3\phi_2\bigg(\genfrac{}{}{0pt}{}{q/a,q/b,q}{0,-q};q,ab\bigg) \nonumber\\
=&~ \sum_{n=0}^\infty\sum_{j=-n}^n(-1)^{j}(1-q^{2n+1})\frac{(q/a,q/b;q)_n}{(aq,bq;q)_n}(ab)^nq^{n(3n+1)/2-j^2}.
\end{align}
\end{theorem}
For the sake of completeness, we provide its proof here.
\begin{proof}
Taking $(c,d,u,v)\rightarrow (0, -q, q, q)$ in Theorem \ref{thm-main}, we have
\begin{align}\label{3chen12}
&\frac{(q,ab;q)_\infty}{(aq,bq;q)_\infty}{}_3\phi_2\bigg(\genfrac{}{}{0pt}{}{q/a,q/b,q}{0,-q};q,ab\bigg)\\
\nonumber
=&\sum_{n=0}^\infty(1-q^{2n+1})\frac{(q/a,q/b;q)_n}{(aq,bq;q)_n}(-ab)^nq^{n(n-1)/2}
{}_3\phi_2\bigg(\genfrac{}{}{0pt}{}{q^{-n},q^{n+1},q}{0,-q};q,q\bigg).
\end{align}
Taking $(\alpha,c) \rightarrow (1,-q)$ in Lemma \ref{meq:2} and simplifying, we find
\begin{eqnarray}
{}_3\phi_2\bigg(\genfrac{}{}{0pt}{}{q^{-n}, q^{n+1},q}{0,-q};q,q\bigg)
=(-1)^nq^{n^2+n}\sum_{j=-n}^n(-1)^jq^{-j^2}.
\end{eqnarray}
Combining the above two equations, we complete the proof of the Theorem \ref{meq:2.1}.
\end{proof}
\begin{corollary}
Identity \eqref{mock-5-6} holds. In addition, we have
\begin{align}
&\sum_{n=0}^\infty(-1)^n\frac{(q;q)_n}{(-q;q)_{n}}q^{n(n+1)/2} \nonumber \\
=&\sum_{n=0}^\infty\sum_{j=-n}^n(-1)^{n+j}(1-q^{2n+1})q^{2n^2+n-j^2} \label{T6} \\
=&\overline{f}_{1,3,1}(-q^2,-q^2,q^2)-q^3\overline{f}_{1,3,1}(-q^6,-q^6,q^2), \label{T6-simplify} \\
&\sum_{n=0}^\infty \frac{(-1)^nq^n(q;q^2)_n}{(-q;q)_n} \nonumber \\ =&\frac{(q;q^2)_{\infty}}{(q^2;q^2)_{\infty}}\sum_{n=0}^{\infty}(-1)^nq^{3n(n+1)/2}\sum_{|j|\leq n}(-1)^jq^{-j^2} \label{5-5-cor-2-lambda}\\
=& \frac{J_{1}}{2J_{2}^2}\left(f_{1,5,1}(-q^2,-q^2,q)-q^3f_{1,5,1}(-q^5,-q^5,q)  \right) \label{5-5-cor-2-lambda-H-pre} \\
=& \frac{J_{1}}{J_{2}^2}f_{1,5,1}(-q^2,-q^2,q) \label{5-5-cor-2-lambda-H} \\
=&4q^{-1}m(-q^6,q^{24},-1)+\frac{J_1}{J_2^2}\Phi_{1,4}(-q^2,-q^2,q), \label{5-5-cor-2-lambda-A} \\
&\sum_{n=0}^\infty \frac{(-1)^nq^{n^2+2n}(q;q^2)_n}{(-q^2;q^2)_n} \nonumber \\
=&\frac{(q;q^2)_\infty}{(q^2;q^2)_\infty}\sum_{n=0}^\infty (-1)^n(1+q^{2n+1})q^{4n^2+3n}\sum_{j=-n}^n(-1)^jq^{-2j^2} \label{5-5-cor-3} \\
=& \frac{J_1}{J_{2}^2}\left(f_{1,3,1}(-q^5,-q^5,q^4)-q^7f_{1,3,1}(-q^{13},-q^{13},q^4)  \right) \label{5-5-cor-3-H-pre} \\
=&\frac{J_1}{J_{2}^2}f_{1,3,1}(q^2,-q^2,q) \label{5-5-cor-3-H} \\
=&2q^{-1}m(q,q^8,-1)-q^{-1}\frac{J_1^2J_8J_{6,16}^2}{J_2^2J_{3,16}J_{5,16}}. \label{5-5-cor-3-A}
\end{align}
\end{corollary}
\begin{proof}
Taking $(a,b)\rightarrow (0,0)$, $(0,1)$ and $(q^{1/2}, -q^{1/2})$ in Theorem \ref{meq:2.1}, we obtain \eqref{mock-5-6}, \eqref{T6} and \eqref{5-5-cor-2-lambda}, respectively. Similarly, taking $(q,a,b)\rightarrow (q^2, 0, q)$ in Theorem \ref{meq:2.1}, we obtain \eqref{5-5-cor-3}.

Setting $(a,b,c,x,y,q)\rightarrow (1,5,1,-q^2,-q^2,q)$ in \eqref{f-id-2}, we obtain \eqref{5-5-cor-2-lambda-H} from \eqref{5-5-cor-2-lambda-H-pre}. Similarly, setting $(a,b,c,x,y,q)\rightarrow (1,3,1,q^2,-q^2,q)$ in \eqref{f-id-1}, we obtain \eqref{5-5-cor-3-H} from \eqref{5-5-cor-3-H-pre}.

The expression \eqref{5-5-cor-2-lambda-A} (resp.\ \eqref{5-5-cor-3-A}) follows from \eqref{5-5-cor-2-lambda-H} (resp.\ \eqref{5-5-cor-3-H}) and Theorem \ref{thm-fg-1} (resp.\ Theorem \ref{thm-fg-2}) with $(n,p)=(1,4)$ (resp.\ $n=1$).
\end{proof}
\begin{rem}
(1) Identity \eqref{mock-5-6} was proved by Liu using Theorem \ref{meq:2.1}. See \cite[Proposition 6.9]{Liu2013Rama}. \\
(2) The left side of \eqref{5-5-cor-2-lambda} is the sixth order mock theta function $\lambda^{(6)}(q)$. See \eqref{mock-6-5-defn}. From \cite[Eq.\ (5.27)]{Hickerson-Mortenson} we find
\begin{align}\label{HM-lambda6-m}
\lambda^{(6)}(q)=2q^{-1}m(1,q^6,-q^2)+\frac{J_{1,2}\overline{J}_{3,12}}{\overline{J}_{1,4}}.
\end{align}
Taking $(x,q,z,z',n)\rightarrow (1,q^6,-q^2,-1,2)$ in Lemma \ref{lem-m-decompose}, and using the method in \cite{Garvan-Liang} to prove theta function identities, we can show that \eqref{HM-lambda6-m} is equivalent to \eqref{5-5-cor-2-lambda-A}.
 \\
(3) After replacing $q$ by $-q$, the left side of \eqref{5-5-cor-3} becomes the eighth order mock theta function $S_1^{(8)}(q)$. See \eqref{mock-8-2-defn}. From \cite[Eq.\ (5.36)]{Hickerson-Mortenson} we find
\begin{align}\label{HM-S18-m}
S_1^{(8)}(q)=-2q^{-1}m(-q,q^8,-1)+q^{-1}\frac{\overline{J}_{3,8}J_{2,8}^2}{J_{1,8}^2}.
\end{align}
This is the same as \eqref{5-5-cor-3-A} upon replacing $q$ by $-q$.
\end{rem}

\subsection{Representations for $\phi_1^{(5)}(q)$ and $\psi_1^{(5)}(q)$}
\begin{theorem}\label{thm-ab-5-6}
(Cf.\ \cite[Proposition 6.1]{Liu2013Rama}.) For $\max\{|ab|, |aq|, |bq|\}<1$, we have
\begin{align}
&\frac{(q,ab;q)_{\infty}}{(aq,bq;q)_{\infty}}\sum_{n=0}^{\infty}(q/a,q/b;q)_{n}(ab)^n \nonumber \\
=&\sum_{n=0}^{\infty}\sum_{j=-n}^{n}(-1)^j\frac{(1-q^{2n+1})(q/a,q/b;q)_n(ab)^nq^{2n^2+n-j(3j+1)/2}}{(aq,bq;q)_n}.
\end{align}
\end{theorem}

\begin{corollary}
Identities \eqref{mock-5-7} and \eqref{mock-5-8} hold.
\end{corollary}
\begin{proof}
Taking $(q,a,b)\rightarrow (q^2, 0, -q)$ in Theorem \ref{thm-ab-5-6}, we obtain \eqref{mock-5-7}. Taking $(a,b)\rightarrow (0,-1)$ in Theorem \ref{thm-ab-5-6}, we obtain \eqref{mock-5-8}.
\end{proof}
As before, identity \eqref{mock-5-8} appears as \cite[Proposition 6.4]{Liu2013Rama}. But it was not pointed out there that \eqref{mock-5-7} also follows from Theorem \ref{thm-ab-5-6}.

\subsection{Representations for $F_1^{(5)}(q)$}
\begin{theorem}\label{thm-ab-5-8}
For $\max \{|ab|, |aq^2|, |bq^2|\}<1$, we have
\begin{align}
&\frac{(q^2,ab;q^2)_\infty}{(aq^2,bq^2;q^2)_\infty}
{}_3\phi_2\bigg(\genfrac{}{}{0pt}{}{q^2/a,q^2/b,q^2}{0,q^3};q^2,ab\bigg) \nonumber\\
=&(1-q)\sum_{n=0}^{\infty}(1+q^{2n+1})\frac{(q^2/a,q^2/b;q^2)_n}{(aq^2,bq^2;q^2)_n}(-ab)^nq^{3n^2+2n}\sum_{j=-n}^{n}q^{-2j^2-j}.
\end{align}
\end{theorem}
\begin{proof}
Taking $(q,c,d,u,v) \rightarrow (q^2, 0, q^3,q^2, q^2)$ in Theorem \ref{thm-main}, we deduce that
\begin{align}
&\frac{(q^2,ab;q^2)_\infty}{(aq^2,bq^2;q^2)_\infty}
{}_3\phi_2\bigg(\genfrac{}{}{0pt}{}{q^2/a,q^2/b,q^2}{0,q^3};q^2,ab\bigg) \nonumber\\
=&\sum_{n=0}^{\infty}(1-q^{4n+2})\frac{(q^2/a,q^2/b;q^2)_n}{(aq^2,bq^2;q^2)_n}(-ab)^nq^{n^2-n}
{}_3\phi_2\bigg(\genfrac{}{}{0pt}{}{q^{-2n},q^{2n+2},q^2}{0,q^3};q^2,q^2\bigg). \label{ab-5-8-proof}
\end{align}
Taking $(q,\alpha,c)\rightarrow (q^2,1,q)$ in Lemma \ref{meq:2} and simplifying, we find
\begin{align}
{}_3\phi_2\bigg(\genfrac{}{}{0pt}{}{q^{-2n}, q^{2n+2},q^2}{0,q^3};q^2,q^2\bigg)
=\frac{1-q}{1-q^{2n+1}}q^{2n^2+3n}\sum_{j=-n}^nq^{-2j^2-j}. \label{ab-5-8-proof-1}
\end{align}
The theorem follows after substituting \eqref{ab-5-8-proof-1} into \eqref{ab-5-8-proof}.
\end{proof}
\begin{corollary}
Identity \eqref{mock-5-9} holds. In addition, we have
\begin{align}
&\sum_{n=0}^\infty(-1)^n\frac{(q^2;q^2)_n}{(q;q^2)_{n+1}}q^{n(n+1)} \nonumber \\
=&\sum_{n=0}^\infty\sum_{j=-n}^n(1+q^{2n+1})q^{4n^2+3n-2j^2-j} \label{5-8-cor-1} \\
=&\overline{f}_{1,3,1}(-q^4,-q^6,q^4)+q^7\overline{f}_{1,3,1}(-q^{12},-q^{14},q^4), \label{5-8-cor-1-simplify}\\
&\sum_{n=0}^\infty \frac{q^{n(n+1)}(-q^2;q^2)_n}{(q;q^2)_{n+1}} \nonumber \\
=&\frac{(-q^2;q^2)_\infty}{(q^2;q^2)_\infty}\sum_{n=0}^\infty (1+q^{2n+1})(-1)^nq^{4n^2+3n}\sum_{j=-n}^nq^{-2j^2-j} \label{5-8-cor-2} \\
=&\frac{J_4}{J_{2}^2}\left(f_{1,3,1}(q^4,q^6,q^4)-q^7f_{1,3,1}(q^{12},q^{14},q^4)\right) \label{5-8-cor-2-H} \\
=&-q^{-1}m(q,q^8,q^{-1/2})+q^{-1}\frac{J_{\frac{3}{2},8}^2J_{8}^3}{J_{\frac{1}{2},8}^2J_{1,8}J_{2,8}}, \label{5-8-cor-2-A} \\
&\sum_{n=1}^{\infty}\frac{q^n(-q;q)_{2n-2}}{(q;q^2)_n} \nonumber \\
=&q\frac{(-q;q)_\infty}{(q;q)_\infty}\sum_{n=0}^\infty (-1)^nq^{3n^2+3n}\sum_{j=-n}^nq^{-2j^2-j} \label{5-8-cor-3-psi} \\
=&\frac{1}{2}q\frac{J_2}{J_{1}^2}\left(f_{1,5,1}(q^3,q^5,q^2)-q^6f_{1,5,1}(q^9,q^{11},q^2)  \right) \label{5-8-cor-3-psi-H-pre} \\
=&q\frac{J_2}{J_{1}^2}f_{1,5,1}(q^3,q^5,q^2) \label{5-8-cor-3-psi-H} \\
=&-m(-q^{18},q^{48},-1)+q^{-3}m(-q^6,q^{48},-1)+q\frac{J_2}{J_1^2}\Phi_{1,4}(q^3,q^5,q^2) \label{5-8-cor-3-psi-A} \\
=&-m(-q^3,q^{12},q)+q^2\frac{J_2^2J_{12}J_{24}J_{4,24}^2}{J_1^2J_{1,24}J_{10,24}J_{11,24}}, \label{5-8-cor-3-psi-final} \\
&\sum_{n=0}^\infty \frac{q^{(n+1)^2}(-q;q^2)_n}{(q;q^2)_{n+1}} \nonumber \\ =&q\frac{(-q;q^2)_{\infty}}{(q^2;q^2)_{\infty}}\sum_{n=0}^{\infty}(-1)^nq^{4n^2+4n}\sum_{j=-n}^{n}q^{-2j^2-j} \label{5-8-cor-4} \\
=&\frac{1}{2}q\frac{J_2}{J_1J_4}\left(f_{1,3,1}(q^5,q^7,q^4)-q^8f_{1,3,1}(q^{13},q^{15},q^4)  \right) \label{5-8-cor-4-H-pre} \\
=&q\frac{J_2}{J_1J_4}f_{1,3,1}(q^5,q^7,q^4) \label{5-8-cor-4-H} \\
=&-m(q^{2},q^8,q^{-1/2})+\frac{J_{\frac{5}{2},8}J_8^3}{J_{\frac{1}{2},8}J_{1,8}J_{3,8}}. \label{5-8-cor-4-A}
\end{align}
\end{corollary}
\begin{proof}
Taking $(a,b)\rightarrow (0,0)$  in Theorem \ref{thm-ab-5-8}, we obtain \eqref{mock-5-9} upon using \eqref{revise-j-convert}.

Taking $(a,b)\rightarrow (0,1)$ and $(0,-1)$ in Theorem \ref{thm-ab-5-8}, we obtain \eqref{5-8-cor-1} and \eqref{5-8-cor-2}, respectively.

We write
\begin{align}
\sum_{n=1}^{\infty}\frac{q^n(-q;q)_{2n-2}}{(q;q^2)_n}=\frac{q}{1-q}\sum_{n=0}^\infty \frac{(-q,-q^2;q^2)_n q^n}{(q^3;q^2)_n}.
\end{align}
Taking $(a,b)\rightarrow (-q,-1)$ in Theorem \ref{thm-ab-5-8}, we obtain \eqref{5-8-cor-3-psi}.

Similarly, taking $(a,b)\rightarrow (0,-q)$ in Theorem \ref{thm-ab-5-8}, we obtain \eqref{5-8-cor-4}.

Setting $(a,b,c,x,y,q)\rightarrow (1,5,1,q^3,q^5,q^2)$ in \eqref{f-id-2}, we obtain \eqref{5-8-cor-3-psi-H} from \eqref{5-8-cor-3-psi-H-pre}. Similarly, setting $(a,b,c,x,y,q)\rightarrow (1,3,1,q^5,q^7,q^4)$ in \eqref{f-id-2}, we obtain
\eqref{5-8-cor-4-H} from \eqref{5-8-cor-4-H-pre}.

The expression \eqref{5-8-cor-2-A} (resp.\ \eqref{5-8-cor-4-A}) follows from \eqref{5-8-cor-2-H} (resp.\ \eqref{5-8-cor-4-H}) and Theorem \ref{thm-fg-2} with $n=1$. The expression  \eqref{5-8-cor-3-psi-A} follows from \eqref{5-8-cor-3-psi-H} and Theorem \ref{thm-fg-1} with $(n,p)=(1,4)$.

Taking $(x,q,z,z',n)\rightarrow (-q^3,q^{12},q,-1,2)$ in Lemma \ref{lem-m-decompose} and using the method in \cite{Garvan-Liang} to prove theta function identities, we get \eqref{5-8-cor-3-psi-final} from \eqref{5-8-cor-3-psi-A}.
\end{proof}
\begin{rem}
(1) After replacing $q$ by $-q$, the left side of \eqref{5-8-cor-2} becomes the eighth order mock theta function $T_1^{(8)}(q)$. See \eqref{mock-8-4-defn}. From \cite[Eq.\ (5.38)]{Hickerson-Mortenson} we know that
\begin{align}\label{HM-T1-m}
T_1^{(8)}(q)=q^{-1}m(-q,q^8,q^6).
\end{align}
Taking $(x,q,z_1,z_0)\rightarrow (q,q^8,q^{-1/2},q^6)$ in Lemma \ref{lem-m-minus}, we see that \eqref{HM-T1-m} is equivalent to \eqref{5-8-cor-2-A}. \\
(2) The left side of \eqref{5-8-cor-3-psi} is the sixth order mock theta function $\psi_{-}^{(6)}(q)$. See \eqref{mock-6-9-defn}. From \cite[Eq.\ (5.31)]{Hickerson-Mortenson} we find the following representation:
\begin{align}\label{HM-psi-minus-m}
\psi_{-}^{(6)}(q)=-\frac{1}{2}m(1,q^3,q)+\frac{1}{2}q\frac{J_6^3}{J_1J_2}.
\end{align}
Setting $(x,q,z,z',n)\rightarrow (1,q^3,q,-1,4)$ in Lemma \ref{lem-m-decompose} and using the method in \cite{Garvan-Liang} to prove theta function identities, we see that \eqref{HM-psi-minus-m} is equivalent to \eqref{5-8-cor-3-psi-A}.\\
(3) The left side of \eqref{5-8-cor-4} is the eighth order mock theta function $V_1^{(8)}(q)$. See \eqref{mock-8-8-defn}. Taking $(x,q,z_1,z_0)\rightarrow (q^2,q^8,q^{-1/2},q)$ in Lemma \ref{lem-m-minus}, we see that \eqref{HM-V1-m} is equivalent to \eqref{5-8-cor-4-A}.
\end{rem}

\section{Mock theta functions of order 6}\label{sec-order-6}
In his lost notebook \cite{lostnotebook}, Ramanujan recorded seven mock theta functions of order 6. They are defined as
\begin{align}
&\phi^{(6)}(q):=\sum_{n=0}^{\infty}\frac{(-1)^n(q;q^2)_{n}q^{n^2}}{(-q;q)_{2n}}, \label{mock-6-1-defn} \\
&\psi^{(6)}(q):=\sum_{n=0}^{\infty}\frac{(-1)^nq^{(n+1)^2}(q;q^2)_{n}}{(-q;q)_{2n+1}}, \label{mock-6-2-defn} \\
&\rho^{(6)}(q):=\sum_{n=0}^{\infty}\frac{q^{\binom{n+1}{2}}(-q;q)_{n}}{(q;q^2)_{n+1}}, \label{mock-6-3-defn} \\
&\sigma^{(6)}(q):=\sum_{n=0}^{\infty}\frac{q^{\binom{n+2}{2}}(-q;q)_{n}}{(q;q^2)_{n+1}}, \label{mock-6-4-defn} \\
&\lambda^{(6)}(q):=\sum_{n=0}^{\infty}\frac{(-1)^nq^n(q;q^2)_{n}}{(-q;q)_{n}}, \label{mock-6-5-defn} \\
&\mu^{(6)}(q):=\frac{1}{2}+\frac{1}{2}\sum_{n=0}^{\infty}\frac{(-1)^nq^{n+1}(1+q^n)(q;q^2)_{n}}{(-q;q)_{n+1}}, \label{mock-6-6-defn} \\
&\gamma^{(6)}(q):=\sum_{n=0}^{\infty}\frac{q^{n^2}(q;q)_{n}}{(q^3;q^3)_{n}}.  \label{mock-6-7-defn}
\end{align}
In Ramanujan's original definition of $\mu^{(6)}(q)$, the series does not converge, and \eqref{mock-6-6-defn} is the correct understanding of his definition.

In 2007, Berndt and Chan \cite{Berndt-Chan} defined two new mock theta functions as:
\begin{align}
&\phi_{-}^{(6)}(q):=\sum_{n=1}^{\infty}\frac{q^n(-q;q)_{2n-1}}{(q;q^2)_{n}}, \label{mock-6-8-defn} \\
&\psi_{-}^{(6)}(q):=\sum_{n=1}^{\infty}\frac{q^{n}(-q;q)_{2n-2}}{(q;q^2)_{n}}. \label{mock-6-9-defn}
\end{align}
These two functions were discovered after an examination of the summands of $\phi^{(6)}(q)$ and $\psi^{(6)}(q)$ over the negative and non-positive integers.

The following representations for sixth order mock theta functions were proved by Andrews and Hickerson \cite{Andrews-Hickerson}, and Berndt and Chan \cite{Berndt-Chan}.
\begin{theorem}\label{thm-ord-6}
We have
\begin{align}
&\phi^{(6)}(q)=\frac{(q;q^2)_{\infty}}{(q^2;q^2)_{\infty}}\left(1+2\sum_{n=1}^{\infty}q^{2n^2-n}+\sum_{n=1}^{\infty}(-1)^nq^{3n^2-n}(1-q^{2n})\sum_{j=-n}^{n}(-1)^{j+1}q^{-j^2} \right), \label{mock-6-1} \\
&\psi^{(6)}(q)=q\frac{(q;q^2)_{\infty}}{(q^2;q^2)_{\infty}}
\sum_{n=0}^{\infty}\sum_{j=-n}^{n}(-1)^{n+j}q^{3n^2+3n-j^2}, \label{mock-6-2} \\
&\rho^{(6)}(q)=\frac{(-q;q)_{\infty}}{(q;q)_{\infty}}\sum_{n=0}^{\infty}\sum_{j=-n}^{n}(-1)^nq^{3n(n+1)/2-j(j+1)/2}
, \label{mock-6-3} \\
&\sigma^{(6)}(q)=q\frac{(-q;q)_\infty}{(q;q)_\infty}\sum_{n=0}^\infty (-1)^n(1-q^{n+1})q^{(3n^2+5n)/2}\sum_{j=0}^n q^{-j(j+1)/2},  \label{mock-6-4} \\
&\lambda^{(6)}(q)=\frac{(q;q^2)_{\infty}}{(q^2;q^2)_{\infty}}\sum_{n=0}^{\infty}(-1)^nq^{3n(n+1)/2}\sum_{|j|\leq n}(-1)^jq^{-j^2}, \label{mock-6-5} \\
&\mu^{(6)}(q)=\frac{(q;q^2)_\infty}{2(q^2;q^2)_\infty}\sum_{n=0}^\infty (-1)^nq^{(3n^2+n)/2}(1+q^{2n+1})\sum_{j=-n}^n(-1)^jq^{-j^2}, \label{mock-6-6} \\
&\gamma^{(6)}(q)=\frac{1}{(q;q)_\infty}\bigg(1+3\sum_{n=1}^{\infty}(-1)^n\frac{1+q^n}{1+q^n+q^{2n}}q^{\frac{3n^2+n}{2}}\bigg),  \label{mock-6-7} \\
&\phi_{-}^{(6)}(q)=\frac{(-q;q)_\infty}{(q;q)_\infty}\sum_{n=0}^{\infty}(1-q^{2n+2})(-1)^nq^{3n^2+5n+1}\sum_{j=0}^n(1+q^{2j+1})q^{-2j^2-3j}, \label{mock-6-8} \\
&\psi_{-}^{(6)}(q)=q\frac{(-q;q)_{\infty}}{(q;q)_{\infty}}\sum_{n=0}^{\infty}(-1)^nq^{3n^3+3n}\sum_{j=-n}^{n}q^{-2j^2-j}. \label{mock-6-9}
\end{align}
\end{theorem}
We will provide parameterized identities which recover these representations. As a byproduct, a new Hecke-type series representation for $\nu^{(3)}(q)$ will be provided in  \eqref{6-2-cor-2-nu}.
\subsection{Representations for $\phi^{(6)}(q)$}
\begin{theorem}\label{thm-ab-6-1}
For $\max\{|ab|/q^2|, |a|, |b|\}<1$, we have
\begin{align}
&\frac{(q^2,ab/q^2;q^2)_\infty}{(a,b;q^2)_\infty}
{}_3\phi_2\bigg(\genfrac{}{}{0pt}{}{q^2/a,q^2/b,q^2}{-q,-q^2};q^2,ab/q^2\bigg) \nonumber\\
=&1+\sum_{n=1}^{\infty}\frac{(q^2/a,q^2/b;q^2)_n}{(a,b;q^2)_n}(ab)^nq^{2n^2-3n}\left(2(-1)^nq^{-n^2}-(1-q^{2n})\sum_{j=-n}^{n}(-1)^{j}q^{-j^2}  \right).
\end{align}
\end{theorem}
\begin{proof}
Taking $(q,c,d,u,v)\rightarrow (q^2,-q,-q^2,1,q^2)$ in Theorem \ref{thm-main}, we deduce that
\begin{align}
&\frac{(q^2,ab/q^2;q^2)_\infty}{(a,b;q^2)_\infty}
{}_3\phi_2\bigg(\genfrac{}{}{0pt}{}{q^2/a,q^2/b,q^2}{-q,-q^2};q^2,ab/q^2\bigg) \nonumber\\
=&~ 1+\sum_{n=1}^{\infty}(1+q^{2n})\frac{(q^2/a,q^2/b;q^2)_n}{(a,b;q^2)_n}(-ab)^nq^{n^2-3n}
{}_3\phi_2\bigg(\genfrac{}{}{0pt}{}{q^{-2n},q^{2n},q^2}{-q,-q^2};q^2,q^2\bigg). \label{ab-6-1-proof}
\end{align}
To evaluate the ${}_{3}\phi_{2}$ series on the right side, taking $(q, \alpha, \beta, c, d) \rightarrow (q^2, 1, q^2, -q, -q^2)$ in \eqref{sears:32}, we obtain
\begin{align}
{}_{3}\phi_{2}\bigg(\genfrac{}{}{0pt}{}{q^{-2n}, q^{2n},q^2}{-q,-q^{2}};q^2,q^2\bigg)=q^{n^2}{}_{3}\phi_{2}\bigg(\genfrac{}{}{0pt}{}{q^{-2n}, q^{2n},-1}{-q,-q^{2}};q^2,-q^3\bigg). \label{6-phi-proof-2}
\end{align}
Now taking $(q, c,d) \rightarrow (q^2, -q^2,-q^3)$ in Lemma \ref{lem-limit}, we deduce that
\begin{align}
&{}_{3}\phi_{2}\bigg(\genfrac{}{}{0pt}{}{q^{-2n}, q^{2n},-1}{-q,-q^{2}};q^2,-q^3\bigg)\nonumber \\
=&(-1)^n\frac{2(1-q^{2n})}{1+q^{2n}}\left(\frac{-1-2q+q^2}{2(1-q)(1+q)}+\sum_{j=2}^{n}(-1)^jq^{2j-j^2}\frac{(1+q^{-1})(1-q^{2j-1})}{(1-q^{2j-2})(1-q^{2j})} \right). \label{ab-6-1-proof-2}
\end{align}
Note that
\begin{align}
(1+q^{-1})(1-q^{2j-1})=(1-q^{2j-2})+q^{-1}(1-q^{2j}).
\end{align}
We have
\begin{align}
&\sum_{j=2}^{n}(-1)^jq^{2j-j^2}\frac{(1+q^{-1})(1-q^{2j-1})}{(1-q^{2j-2})(1-q^{2j})} \nonumber \\
=&~\sum_{j=2}^{n}\frac{(-1)^jq^{2j-j^2}}{1-q^{2j}}+\sum_{j=2}^{n}\frac{(-1)^jq^{2j-j^2-1}}{1-q^{2j-2}} \nonumber \\
=&~\sum_{j=2}^{n}\frac{(-1)^jq^{2j-j^2}}{1-q^{2j}}+\sum_{j=1}^{n-1}\frac{(-1)^{j+1}q^{-j^2}}{1-q^{2j}} \nonumber \\
=&~\frac{q^{-1}}{1-q^2}+\frac{(-1)^nq^{2n-n^2}}{1-q^{2n}}+\sum_{j=2}^{n-1}(-1)^{j+1}q^{-j^2}. \label{ab-6-1-proof-3}
\end{align}
Substituting \eqref{ab-6-1-proof-3} into \eqref{ab-6-1-proof-2}, after simplifications, we obtain
\begin{align}
&{}_{3}\phi_{2}\bigg(\genfrac{}{}{0pt}{}{q^{-2n}, q^{2n},-1}{-q,-q^{2}};q^2,-q^3\bigg) \nonumber \\
=&~ (-1)^n\cdot \frac{2(1-q^{2n})}{1+q^{2n}}\left(\frac{(-1)^nq^{-n^2}}{1-q^{2n}}-\frac{1}{2}+\sum_{j=1}^{n}(-1)^{j+1}q^{-j^2} \right). \label{6-phi-proof-3}
\end{align}
The theorem follows after substituting \eqref{6-phi-proof-3} into \eqref{6-phi-proof-2} and then combining with \eqref{ab-6-1-proof}.
\end{proof}
\begin{corollary}
Identity \eqref{mock-6-1} holds. In addition, we have
\begin{align}
&\sum_{n=0}^\infty \frac{q^{n^2-n}}{(-q;q^2)_n}\nonumber \\
=&2\frac{(-q^2;q^2)_\infty}{(q^2;q^2)_\infty} \left(1+\sum_{n=1}^\infty (-1)^n (1+q^{2n})q^{2n^2-2n}\right. \nonumber \\
&\quad \quad \left.-\frac{1}{2}\sum_{n=1}^\infty (1-q^{4n})q^{3n^2-2n}\sum_{j=-n}^{n}(-1)^{j}q^{-j^2} \right) \label{6-1-cor-1} \\
=&1-\frac{J_4}{J_{2}^2}\left(f_{1,2,1}(1,1,q^4)+qf_{1,2,1}(q^6,q^6,q^4) \right) \label{6-1-cor-1-H} \\
=&1+2q^{-1}m(q^2,q^{12},-1)-q^{-1}\frac{J_1J_3J_{8}^2J_{12}^2}{J_2J_4J_6J_{24}^2}, \label{6-1-cor-1-A} \\
&\sum_{n=0}^\infty \frac{q^{n^2}}{(-q^2;q^2)_n} \nonumber \\
=&\frac{(-q;q^2)_\infty}{(q^2;q^2)_\infty} \left(1+2\sum_{n=1}^\infty (-1)^nq^{2n^2-n}-\sum_{n=1}^\infty (1-q^{2n})q^{3n^2-n}\sum_{j=-n}^{n}(-1)^{j}q^{-j^2}  \Big)  \right) \label{6-1-cor-2} \\
=&2-\frac{J_2}{J_1J_4}\left(f_{1,2,1}(q,q,q^4)+q^2f_{1,2,1}(q^7,q^7,q^4) \right) \label{6-1-cor-2-H}\\
=&2-2m(q^7,q^{12},-1)+2q^{-1}m(q,q^{12},-1)-q^{-1}\frac{J_1J_6^2J_8^2}{J_2J_4J_{24}^2} \label{6-1-cor-2-A} \\
=&2m(q,-q^3,-1)+2q\frac{J_{12}^3}{J_{3,12}J_{4,12}}. \label{6-1-cor-2-final}
\end{align}
\end{corollary}
\begin{proof}
Taking $(a,b)\rightarrow (0,q), (0, -1)$ and $(0,-q)$ in Theorem \ref{thm-ab-6-1}, we obtain \eqref{mock-6-1}, \eqref{6-1-cor-1} and \eqref{6-1-cor-2}, respectively.

The Appell-Lerch series expressions follow from their corresponding Hecke-type series and Theorem \ref{thm-fg-1} with $(n,p)=(1,1)$.

By Lemma \ref{lem-m-prop} we have
\begin{align}\label{revise-Cor7.3-neweq-1}
2-2m(q^7,q^{12},-1)=2q^{-5}m(q^{-5},q^{12},-1)=2m(q^5,q^{12},-1).
\end{align}
Taking $(x,q,z,z',n)\rightarrow (q,-q^3,-1,-1,2)$ in Lemma \ref{lem-m-decompose}, then substituting the result and \eqref{revise-Cor7.3-neweq-1} into \eqref{6-1-cor-2-A},  and using the method in \cite{Garvan-Liang} to prove theta function identities, we obtain \eqref{6-1-cor-2-final}.
\end{proof}
\begin{rem}
(1) The left side of \eqref{6-1-cor-1} is  equal to $\nu^{(3}(q)+1$. See \eqref{mock-3-6-defn} and \eqref{6-2-cor-2-nu}. Taking $(x,q,z_1,z_0)\rightarrow (q^2,q^{12},-1,-q^3)$ in Lemma \ref{lem-m-minus} and using the method in \cite{Garvan-Liang} for proving theta function identities, we see that \eqref{HM-nu3-m} is equivalent to \eqref{6-1-cor-1-A}. \\
(2) The left side of \eqref{6-1-cor-2} is the third order mock theta function $\phi^{(3)}(q)$. See \eqref{mock-3-2-defn}. The expression \eqref{6-1-cor-2-final} can be found in \cite[Eq.\ (5.5)]{Hickerson-Mortenson} with a typo that $J_4$ should be $J_{4,12}$. Mortenson \cite[Eq.\ (2.14)]{Mortenson-2013} gave a different Hecke-type series representation for $\phi^{(3)}(q)$.
\end{rem}

\subsection{Representations for $\psi^{(6)}(q)$}\label{sub-6-2}
\begin{theorem}\label{thm-ab-6-2}
For $\max \{|ab|, |aq^2|, |bq^2|\}<1$, we have
\begin{align}
&\frac{(q^2,ab;q^2)_\infty}{(aq^2,bq^2;q^2)_\infty}
{}_3\phi_2\bigg(\genfrac{}{}{0pt}{}{q^2/a,q^2/b,q^2}{-q^2,-q^3};q^2,ab\bigg) \nonumber\\
=&~ (1+q)\sum_{n=0}^{\infty}\sum_{j=-n}^{n}(1-q^{2n+1})\frac{(q^2/a,q^2/b;q^2)_n}{(aq^2,bq^2;q^2)_n}(ab)^n(-1)^jq^{2n^2+n-j^2}.
\end{align}
\end{theorem}
\begin{proof}
Taking $(q,c,d,u,v) \rightarrow (q^2,-q^2,-q^3,q^2,q^2)$ in Theorem \ref{thm-main}, we deduce that
\begin{align}
&\frac{(q^2,ab;q^2)_\infty}{(aq^2,bq^2;q^2)_\infty}
{}_3\phi_2\bigg(\genfrac{}{}{0pt}{}{q^2/a,q^2/b,q^2}{-q^2,-q^3};q^2,ab\bigg) \nonumber\\
=&\sum_{n=0}^{\infty}(1-q^{4n+2})\frac{(q^2/a,q^2/b;q^2)_n}{(aq^2,bq^2;q^2)_n}(-ab)^nq^{n^2-n}
{}_3\phi_2\bigg(\genfrac{}{}{0pt}{}{q^{-2n},q^{2n+2},q^2}{-q^2,-q^3};q^2,q^2\bigg). \label{ab-6-2-proof}
\end{align}
Taking $(q, \alpha, \beta, c, d) \rightarrow (q^2, q^2, q^2, -q^2, -q^3)$ in \eqref{sears:32}, we deduce that
\begin{align}
{}_{3}\phi_{2}\bigg(\genfrac{}{}{0pt}{}{q^{-2n}, q^{2n+2},q^2}{-q^2,-q^{3}};q^2,q^2\bigg)=q^{n^2+n}{}_{3}\phi_{2}\bigg(\genfrac{}{}{0pt}{}{q^{-2n}, q^{2n+2},-q}{-q^2,-q^{3}};q^2,-q^2\bigg). \label{psi-proof-2}
\end{align}
Taking $(q, \alpha, c, d) \rightarrow (q^2, 1, -q, -q^2)$ in Lemma \ref{meq:1}, we deduce that
\begin{align}
{}_{3}\phi_{2}\bigg(\genfrac{}{}{0pt}{}{q^{-2n}, q^{2n+2},-q}{-q^2,-q^{3}};q^2,-q^2\bigg)=(-1)^nq^n\frac{1+q}{1+q^{2n+1}}\sum_{j=-n}^{n}(-1)^jq^{-j^2}. \label{psi-proof-3}
\end{align}
The theorem follows after substituting \eqref{psi-proof-3} into \eqref{psi-proof-2} and then combining with \eqref{ab-6-2-proof}.
\end{proof}

\begin{corollary}
Identity \eqref{mock-6-2} holds. In addition, we have
\begin{align}
&\sum_{n=0}^\infty \frac{(-1)^nq^{n(n+1)}(q^2;q^2)_n}{(-q;q)_{2n+1}} \nonumber \\
=&\sum_{n=0}^\infty\sum_{j=-n}^n(-1)^{n+j}(1-q^{2n+1})q^{3n^2+2n-j^2} \label{6-2-cor-1} \\
=&\overline{f}_{1,2,1}(-q^4,-q^4,q^4)-q^5\overline{f}_{1,2,1}(-q^{10},-q^{10},q^4) \label{6-2-cor-1-H} \\
=&\overline{f}_{1,2,1}(q^{3/2},-q^{3/2},q). \label{6-2-cor-1-combine} \\
&\sum_{n=0}^\infty \frac{q^{n(n+1)}}{(-q;q^2)_{n+1}} \nonumber \\
=&\frac{(-q^2;q^2)_\infty}{(q^2;q^2)_\infty}\sum_{n=0}^\infty (1-q^{2n+1})q^{3n^2+2n}\sum_{j=-n}^n(-1)^jq^{-j^2} \label{6-2-cor-2-nu} \\
=&\frac{J_4}{J_{2}^2}\left(f_{1,2,1}(q^4,q^4,q^4)+q^5f_{1,2,1}(q^{10},q^{10},q^4) \right) \label{6-2-cor-2-nu-H}\\
=&2q^{-1}m(q^2,q^{12},-1)-q^{-1}\frac{J_1J_3J_8^2J_{12}^2}{J_2J_4J_6J_{24}^2}. \label{6-2-cor-2-nu-A}
\end{align}
Furthermore, \eqref{6-2-cor-2-nu} is equivalent to \eqref{2-2-cor-1}.
\end{corollary}
\begin{proof}
Taking $(a,b)\rightarrow (0,q)$, $(0,1)$ and $(0,-1)$ in Theorem \ref{thm-ab-6-2}, we obtain \eqref{mock-6-2}, \eqref{6-2-cor-1} and \eqref{6-2-cor-2-nu}, respectively.

Setting $(a,b,c,x,y,q)\rightarrow (1,2,1,q^{3/2},-q^{3/2},q)$ in \eqref{barf-id-1},  we get \eqref{6-2-cor-1-combine}  from \eqref{6-2-cor-1-H}.

The expression \eqref{6-2-cor-2-nu-A} follows from \eqref{6-2-cor-2-nu-H} and Theorem \ref{thm-fg-1} with $(n,p)=(1,1)$.

Note that the left side of \eqref{6-2-cor-2-nu} is in fact the third order mock theta function $\nu^{(3)}(q)$. We now show that \eqref{6-2-cor-2-nu} can be deduced from \eqref{2-2-cor-1}, i.e., \eqref{Mortenson-nu} and vice versa.

Taking $(x,q,z,z',n)\rightarrow (q^2,q^{12},-1,q^3,2)$ in Lemma \ref{lem-m-decompose} and using \eqref{m-id-2}, we deduce that
\begin{align}
m(q^2,q^{12},-1)=& m(-q^{16},q^{48},q^3)+q^{-2}m(-q^{8},q^{48},q^{-3})\nonumber \\
&+q^3\frac{J_{24}^3}{\overline{J}_{2,12}J_{3,48}}\left(\frac{J_{19,24}J_{-3,48}}{\overline{J}_{19,24}\overline{J}_{0,24}}+q^2\frac{J_{31,24}J_{21,48}}{\overline{J}_{19,24}\overline{J}_{12,24}} \right). \label{revise-equiv-nu-1}
\end{align}
Setting $(x,q,z_1,z_0)\rightarrow (-q^{16},q^{48},q^3,q^{-3})$ in Lemma \ref{lem-m-minus}, we obtain
\begin{align}\label{revise-equiv-nu-2}
m(-q^{16},q^{48},q^3)-m(-q^{16},q^{48},q^{-3})=q^{-3}\frac{J_{48}^3J_{6,48}\overline{J}_{16,48}}{J_{-3,48}J_{3,48}\overline{J}_{13,48}\overline{J}_{19,48}}.
\end{align}
Setting $(x,q,z_1,z_0)\rightarrow (-q^8,q^{48},q^3,q^{-3})$ in Lemma \ref{lem-m-minus}, we obtain
\begin{align}\label{revise-equiv-nu-3}
m(-q^8,q^{48},q^3)-m(-q^8,q^{48},q^{-3})=q^{-3}\frac{J_{48}^3J_{6,48}\overline{J}_{8,48}}{J_{-3,48}J_{3,48}\overline{J}_{5,48}\overline{J}_{11,48}}.
\end{align}
Replacing $q$ by $-q$ in \eqref{6-2-cor-2-nu-A} and applying \eqref{revise-equiv-nu-1}--\eqref{revise-equiv-nu-3}, we deduce that
\begin{align}\label{revise-equiv-nu-4}
\nu^{(3)}(-q)=&-q^{-1}m(-q^{16},q^{48},q^3)-q^{-3}m(-q^{8},q^{48},q^{-3})-q^{-1}m(-q^{16},q^{48},q^{-3})\nonumber \\
&-q^{-3}m(-q^8,q^{48},q^{-3})
-q^2\frac{J_{24}^3}{\overline{J}_{2,12}J_{3,48}}\left(\frac{J_{19,24}J_{-3,48}}{\overline{J}_{19,24}\overline{J}_{0,24}}+q^2\frac{J_{31,24}J_{21,48}}{\overline{J}_{19,24}\overline{J}_{12,24}} \right)\nonumber \\
&-q^{-4}\frac{J_{48}^3J_{6,48}\overline{J}_{16,48}}{J_{-3,48}J_{3,48}\overline{J}_{13,48}\overline{J}_{19,48}}
-q^{-6}\frac{J_{48}^3J_{6,48}\overline{J}_{8,48}}{J_{-3,48}J_{3,48}\overline{J}_{5,48}}\nonumber \\
&+q^{-1}\frac{J_2^2J_6^2J_8^2J_{12}^2}{J_1J_3J_4^2J_{12}J_{24}^2}.
\end{align}
Comparing this with \eqref{2-2-cor-1-A}, it suffices to show that the sum of the products in \eqref{revise-equiv-nu-4} equals the product in \eqref{2-2-cor-1-A}. This is a routine exercise and can be verified using the method in \cite{Garvan-Liang}.
\end{proof}
\begin{rem}
As noted in the proof,  \eqref{6-2-cor-2-nu} provides a Hecke-type series representation for the third order mock theta function $\nu^{(3)}(q)$. See \eqref{nu-Hecke-pre}. This representation appears to be new. Using Lemma \ref{lem-m-minus} with $(x,q,z_1,z_0)\rightarrow (q^2,q^{12},-1,-q^3)$, we can show that \eqref{6-2-cor-2-nu-A} is equivalent to \eqref{HM-nu3-m}.

We also observe that \eqref{6-2-cor-2-nu-A} differs from \eqref{6-1-cor-1-A} by 1, which is obvious from their Eulerian forms.
\end{rem}

\subsection{Representations for $\rho^{(6)}(q)$, $\sigma^{(6)}(q)$ and $\lambda^{(6)}(q)$}
Identities \eqref{mock-6-3}, \eqref{mock-6-4} and  \eqref{mock-6-5} have already been established in \eqref{2-2-cor-3-rho},  \eqref{2-1-cor-add-sigma} and \eqref{5-5-cor-2-lambda}, respectively.

\subsection{Representations for $\mu^{(6)}(q)$}
\begin{theorem}\label{thm-ab-6-6}
For $\max \{|abq|, |aq^2|, |bq^2|\}<1$, we have
\begin{align}
&\frac{(q,abq;q)_\infty}{(aq^2,bq^2;q)_\infty}{}_3\phi_2\bigg(\genfrac{}{}{0pt}{}{q/a,q/b,q}{0,-q};q,abq\bigg) \nonumber \\
=&~ \sum_{n=0}^{\infty}(1-q^{2n+2})\frac{(q/a,q/b;q)_n}{(aq^2,bq^2;q)_n}(ab)^nq^{(3n^2+5n)/2} \nonumber \\
&\quad \quad \times \left(2(-1)^nq^{-n^2-n}-(1+q^{n+1})\sum_{j=-n}^{n}(-1)^jq^{-j^2}  \right).
\end{align}
\end{theorem}
\begin{proof}
Taking $(c,d,u,v)\rightarrow (0, -q, q^2, q)$ in Theorem \ref{thm-main}, we deduce that
\begin{align}
&\frac{(q,abq;q)_\infty}{(aq^2,bq^2;q)_\infty}{}_3\phi_2\bigg(\genfrac{}{}{0pt}{}{q/a,q/b,q}{0,-q};q,abq\bigg) \nonumber \\
=&\sum_{n=0}^{\infty}(1-q^{2n+2})(1-q^{n+1})\frac{(q/a,q/b;q)_n}{(aq^2,bq^2;q)_n}(-abq^2)^nq^{(n^2-3n)/2}{}_3\phi_2\bigg(\genfrac{}{}{0pt}{}{q^{-n},q^{n+2},q}{0,-q};q,q\bigg). \label{ab-6-6-proof}
\end{align}
Setting $(\alpha, c)=(q,-q)$ in Lemma \ref{meq:2}, we obtain
\begin{align}
{}_3\phi_2\bigg(\genfrac{}{}{0pt}{}{q^{-n},q^{n+2},q}{0,-q};q,q\bigg)
=2(-1)^n\frac{1+q^{n+1}}{1-q^{n+1}}q^{n^2+2n}\sum_{j=0}^{n}\frac{1-q^{2j+1}}{(1+q^j)(1+q^{j+1})}(-1)^jq^{-j^2-j}.
\end{align}
Note that
\begin{align}
&\sum_{j=0}^{n}\frac{1-q^{2j+1}}{(1+q^j)(1+q^{j+1})}(-1)^jq^{-j^2-j} \nonumber \\
=&~ \sum_{j=0}^{n}\frac{(1+q^j)-q^j(1+q^{j+1})}{(1+q^j)(1+q^{j+1})}(-1)^jq^{-j^2-j} \nonumber \\
=&~ \sum_{j=0}^{n}\frac{(-1)^jq^{-j^2-j}}{1+q^{j+1}}-\sum_{j=0}^{n}\frac{(-1)^jq^{-j^2}}{1+q^j} \nonumber \\
=&~ \sum_{j=1}^{n}\frac{(-1)^{j-1}q^{-j(j-1)}}{1+q^j}+\frac{(-1)^nq^{-n^2-n}}{1+q^{n+1}}-\frac{1}{2}-\sum_{j=1}^{n}\frac{(-1)^jq^{-j^2}}{1+q^j} \nonumber \\
=&~ \frac{(-1)^nq^{-n^2-n}}{1+q^{n+1}}-\frac{1}{2}\sum_{j=-n}^{n}(-1)^jq^{-j^2}.
\end{align}
Therefore,
\begin{align}
&{}_3\phi_2\bigg(\genfrac{}{}{0pt}{}{q^{-n},q^{n+2},q}{0,-q};q,q\bigg)\nonumber \\
=&~ (-1)^n\frac{q^{n^2+2n}}{1-q^{n+1}}\left(2(-1)^nq^{-n^2-n}-(1+q^{n+1})\sum_{j=-n}^{n}(-1)^jq^{-j^2}  \right). \label{ab-6-6-proof-1}
\end{align}
The theorem then follows after substituting \eqref{ab-6-6-proof-1} into \eqref{ab-6-6-proof}.
\end{proof}

\begin{corollary}
Identity \eqref{mock-6-6} holds. In addition, we have
\begin{align}
&\sum_{n=0}^{\infty}\frac{(-1)^nq^{2n+1}(q;q^2)_n}{(-q;q)_n} \nonumber \\
=&~ -2+\frac{(q;q^2)_\infty}{(q^2;q^2)_\infty}\sum_{n=0}^\infty (-1)^nq^{(3n^2+n)/2}(1+q^n+q^{2n+1})\sum_{j=-n}^{n}(-1)^jq^{-j^2} \label{revise-H(q)} \\
=&-2+\frac{J_1}{J_{2}^2}\left(f_{1,5,1}(-q,-q,q)+f_{1,5,1}(-q^2,-q^2,q)+qf_{1,5,1}(-q^3,-q^3,q) \right) \label{revise-H(q)-H} \\
=&-2+4m(-q^{10},q^{24},-1)+4q^{-1}m(-q^6,q^{24},-1)+4q^{-2}m(-q^2,q^{24},-1) \nonumber \\
& +\frac{J_1}{J_2^2}\left(\Phi_{1,4}(-q,-q,q)+\Phi_{1,4}(-q^2,-q^2,q)+q\Phi_{1,4}(-q^3,-q^3,q) \right) \label{revise-H(q)-A}\\
=&2q^{-1}m(1,q^6,-q^2)+4m(q^2,q^6,-1)-2+\frac{J_{1,2}\overline{J}_{3,12}}{\overline{J}_{1,4}}-\frac{J_{1,2}\overline{J}_{1,3}}{\overline{J}_{1,4}}.
\label{revise-H(q)-final}
\end{align}
\end{corollary}
\begin{proof}
For convenience, we denote
\begin{align}
H(q):=\sum_{n=0}^{\infty}\frac{(-1)^nq^{2n+1}(q;q^2)_n}{(-q;q)_n}.
\end{align}
Taking $(a,b)\rightarrow (q^{1/2},-q^{1/2})$ in Theorem \ref{thm-ab-6-6}, we deduce that
\begin{align}
H(q)=& q\frac{(q;q^2)_\infty}{(q^2;q^2)_\infty}\sum_{n=0}^\infty
\frac{(1+q)(1-q^{2n+2})}{(1-q^{2n+1})(1-q^{2n+3})}(-1)^nq^{(3n^2+7n)/2}\nonumber \\
& \times \left(2(-1)^nq^{-n^2-n}-(1+q^{n+1})\sum_{j=-n}^{n}(-1)^jq^{-j^2}  \right). \label{6-6-proof-eq-1}
\end{align}
This gives \cite[Eq.\ (2.28)]{Andrews-Hickerson}.

Let $S_n=\sum_{j=-n}^{n}(-1)^jq^{-j^2}$. Since $S_{n+1}-S_{n}=2(-1)^{n+1}q^{-n^2-2n-1}$, we have
\begin{align}
&2(-1)^nq^{-n^2-n}-(1+q^{n+1})\sum_{j=-n}^{n}(-1)^jq^{-j^2} \nonumber \\
 = &~ -q^{n+1}(S_{n+1}-S_n)-(1+q^{n+1})S_n \nonumber \\
 =&~ -q^{n+1}S_{n+1}-S_n. \label{6-6-proof-add-eq-1}
\end{align}
Substituting \eqref{6-6-proof-add-eq-1} into \eqref{6-6-proof-eq-1}, and noting that
\begin{align}
\frac{(1+q)(1-q^{2n+2})}{(1-q^{2n+1})(1-q^{2n+3})}=\frac{1}{1-q^{2n+1}}+\frac{q}{1-q^{2n+3}},
\end{align}
we deduce that
\begin{align}
H(q)=&~ -\frac{(q;q^2)_\infty}{(q^2;q^2)_\infty}\Big(\sum_{n=0}^\infty (-1)^n\frac{q^{(3n^2+7n+2)/2}S_n}{1-q^{2n+1}}+\sum_{n=0}^\infty (-1)^n\frac{q^{(3n^2+9n+4)/2}S_{n+1}}{1-q^{2n+1}} \nonumber \\
&\quad \quad +q\sum_{n=0}^\infty (-1)^n\frac{q^{(3n^2+7n+2)/2}S_n}{1-q^{2n+3}} +q\sum_{n=0}^\infty (-1)^n \frac{q^{(3n^2+9n+4)/2}S_{n+1}}{1-q^{2n+3}}\Big). \label{6-6-proof-add-2}
\end{align}
Replacing $n$ by $n-1$ in the last second and last sum in \eqref{6-6-proof-add-2}, after rearrangements and simplifications, we obtain
\begin{align}
H(q)=&-\frac{(q;q^2)_\infty}{(q^2;q^2)_\infty} \Big(\frac{q}{1-q}- \sum_{n=1}^\infty (-1)^nq^{(3n^2+3n)/2}S_n \nonumber \\
& +\sum_{n=0}^\infty (-1)^n\frac{q^{(3n^2+9n+4)/2}}{1-q^{2n+1}}S_{n+1} +\sum_{n=1}^\infty (-1)^{n-1}\frac{q^{(3n^2+n)/2}}{1-q^{2n+1}}S_{n-1}   \Big). \label{6-6-cor-1}
\end{align}
Now we substitute the relations
\begin{align}
S_{n-1}=S_n-2(-1)^nq^{-n^2}, \quad S_{n+1}=S_n+2(-1)^{n+1}q^{-n^2-2n-1}
\end{align}
into \eqref{6-6-cor-1} to deduce that
\begin{align}
&H(q)= \frac{(q;q^2)_\infty}{(q^2;q^2)_\infty}\sum_{n=0}^\infty (-1)^n\Big(q^{(3n^2+3n)/2}S_n \nonumber \\
&\quad \quad \quad +q^{(3n^2+n)/2}(1+q^{2n+1})S_n -2(-1)^nq^{(n^2+n)/2}  \Big),
\end{align}
from which \eqref{revise-H(q)} follows.

Andrews and Hickerson \cite{Andrews-Hickerson} showed that
\begin{align}
2\mu^{(6)}(q)=2+H(q)-\lambda^{(6)}(q). \label{mu-lambda}
\end{align}
Combining \eqref{mu-lambda} with \eqref{6-6-cor-1} and \eqref{mock-6-5}, we get \eqref{mock-6-6}.

Now we convert the Hecke-type series in \eqref{revise-H(q)} to a form in terms of $f_{a,b,c}(x,y,q)$. We have
\begin{align}
&\sum_{n=0}^\infty (-1)^nq^{(3n^2+n)/2}\sum_{j=-n}^n(-1)^jq^{-j^2} \nonumber \\
=& \frac{1}{2}\left(\sum_{n=0}^\infty (-1)^nq^{(3n^2+n)/2}\sum_{j=-n}^n(-1)^jq^{-j^2}-\sum_{n=-\infty}^{-1}\sum_{j=n+1}^{-n-1}(-1)^{n+j}q^{\frac{3}{2}n^2+\frac{5}{2}n+1-j^2} \right). \label{add-H-simplify-1}
\end{align}
Similarly, we have
\begin{align}
&\sum_{n=0}^\infty (-1)^nq^{(3n^2+n)/2+2n+1}\sum_{j=-n}^n(-1)^jq^{-j^2} \nonumber \\
=& \frac{1}{2}\left(\sum_{n=0}^\infty \sum_{j=-n}^{n} (-1)^{n+j}q^{\frac{3}{2}n^2+\frac{5}{2}n+1-j^2}-\sum_{n=-\infty}^{-1}\sum_{j=n+1}^{-n-1}(-1)^{n+j}q^{\frac{3}{2}n^2+\frac{1}{2}n-j^2} \right). \label{add-H-simplify-2}
\end{align}
Combining \eqref{add-H-simplify-1} with \eqref{add-H-simplify-2}, we deduce that
\begin{align}
&\sum_{n=0}^\infty (-1)^nq^{(3n^2+n)/2}(1+q^{2n+1})\sum_{j=-n}^n(-1)^jq^{-j^2} \nonumber \\
=&\frac{1}{2}\left(\sum_{\substack{n+j\geq 0 \\ n-j\geq 0}}-\sum_{\substack{n+j<0\\n-j<0}} \right)(-1)^{n+j}q^{\frac{3}{2}n^2+\frac{1}{2}n-j^2} \nonumber \\
&  +\frac{1}{2}\left(\sum_{\substack{n+j\geq 0 \\ n-j\geq 0}}-\sum_{\substack{n+j<0\\n-j<0}} \right)(-1)^{n+j}q^{\frac{3}{2}n^2+\frac{5}{2}n+1-j^2}. \label{add-H-simplify-3}
\end{align}
Letting $n+j=r$ and $n-j=s$ and then splitting the sums according to the parity of $r$ and $s$, after simplifications, we arrive at
\begin{align}
&\sum_{n=0}^\infty (-1)^nq^{(3n^2+n)/2}(1+q^{2n+1})\sum_{j=-n}^n(-1)^jq^{-j^2} \nonumber \\
=&\frac{1}{2}\left(f_{1,5,1}(-q,-q,q)-q^2f_{1,5,1}(-q^4,-q^4,q)+qf_{1,5,1}(-q^3,-q^3,q)\right. \nonumber \\
&\quad \quad \left.-q^5f_{1,5,1}(-q^6,-q^6,q)\right). \label{add-H-simplify-4}
\end{align}
Now we treat the remaining sum in the right side of \eqref{revise-H(q)}. In a way similar to the above, we deduce that
\begin{align}
&\sum_{n=0}^\infty (-1)^nq^{(3n^2+n)/2+n}\sum_{j=-n}^n(-1)^jq^{-j^2} \nonumber \\
=&\frac{1}{2}\left( \sum_{n=0}^\infty (-1)^nq^{(3n^2+3n)/2}\sum_{j=-n}^n(-1)^jq^{-j^2} -\sum_{n=-\infty}^{-1}\sum_{j=n+1}^{-n-1}(-1)^{n+j}q^{(3n^2+3n)/2-j^2} \right) \nonumber \\
=&\frac{1}{2}\left(\sum_{\substack{n+j\geq 0 \\ n-j\geq 0}}-\sum_{\substack{n+j<0\\n-j<0}}  \right)(-1)^{n+j}q^{(3n^2+3n)/2-j^2} \nonumber \\
=&\frac{1}{2}\left(f_{1,5,1}(-q^2,-q^2,q)-q^3f_{1,5,1}(-q^5,-q^5,q) \right). \label{add-H-simplify-5}
\end{align}
Using \eqref{f-id-2} we deduce that
\begin{align}
f_{1,5,1}(-q,-q,q)&=-q^5f_{1,5,1}(-q^6,-q^6,q), \label{add-H-use-1} \\
f_{1,5,1}(-q^2,-q^2,q)&=-q^3f_{1,5,1}(-q^5,-q^5,q), \label{add-H-use-2} \\
f_{1,5,1}(-q^3,-q^3,q)&=-qf_{1,5,1}(-q^4,-q^4,q). \label{add-H-use-3}
\end{align}
Now adding \eqref{add-H-simplify-4} and \eqref{add-H-simplify-5} up and employing the identities in \eqref{add-H-use-1}--\eqref{add-H-use-3}, we arrive at \eqref{revise-H(q)-H}.

The expression \eqref{revise-H(q)-A} follows from \eqref{revise-H(q)-H} and Theorem \ref{thm-fg-1} with $(n,p)=(1,4)$.

Taking $(x,q,z,z',n)\rightarrow (1,q^6,-q^2,-1,2)$ and $(q^2,q^6,-1,-1,2)$ in Lemma \ref{lem-m-decompose}, then substituting the results into \eqref{revise-H(q)-A}, and using the method in \cite{Garvan-Liang} to prove theta function identities, we arrive at \eqref{revise-H(q)-final}.
\end{proof}
\begin{rem}
From \cite[Eq.\ (5.28)]{Hickerson-Mortenson} we find
\begin{align}\label{HM-mu6-m}
\mu^{(6)}(q)=2m(q^2,q^6,-1)-\frac{1}{2}\frac{J_{1,2}\overline{J}_{1,3}}{\overline{J}_{1,4}}.
\end{align}
The expression \eqref{revise-H(q)-final} can also be deduced from \eqref{HM-mu6-m} and \eqref{HM-lambda6-m} upon using \eqref{mu-lambda}.
\end{rem}

\subsection{Representations for $\gamma^{(6)}(q)$}
\begin{theorem}\label{thm-ab-6-7}
For $\max \{|ab/q|, |a|, |b|\}<1$, we have
\begin{align}
&\frac{(q,ab/q;q)_\infty}{(a,b;q)_\infty}
{}_3\phi_2\bigg(\genfrac{}{}{0pt}{}{q/a,q/b,q}{\zeta_3q,\zeta_3^2q};q,ab/q\bigg) \nonumber\\
=&~ 1+3\sum_{n=1}^{\infty}\frac{1+q^{n}}{1+q^n+q^{2n}}\frac{(q/a,q/b;q)_n}{(a,b;q)_n}(-ab)^nq^{\frac{n^2-n}{2}}.
\end{align}
\end{theorem}
\begin{proof}
Taking $(c,d,u,v)\rightarrow (\zeta_3q, \zeta_3^2q, 1, q)$ in Theorem \ref{thm-main}, we deduce that
\begin{align}
&\frac{(q,ab/q;q)_\infty}{(a,b;q)_\infty}
{}_3\phi_2\bigg(\genfrac{}{}{0pt}{}{q/a,q/b,q}{\zeta_3q,\zeta_3^2q};q,ab/q\bigg) \nonumber\\
=&1+\sum_{n=1}^{\infty}(1+q^{n})\frac{(q/a,q/b;q)_n}{(a,b;q)_n}(-ab)^nq^{\frac{n^2-3n}{2}}
{}_3\phi_2\bigg(\genfrac{}{}{0pt}{}{q^{-n},q^{n},q}{\zeta_3q,\zeta_3^2q};q,q\bigg). \label{ab-6-7-proof}
\end{align}
Setting $(a,b,c)=(1, \zeta_3, \zeta_3^2)$ in \eqref{Pfaff}, we obtain
\begin{align}
{}_3\phi_2\bigg(\genfrac{}{}{0pt}{}{q^{-n},q^{n},q}{\zeta_3q,\zeta_3^2q};q,q\bigg)=\frac{3q^n}{1+q^n+q^{2n}}. \label{ab-6-7-proof-1}
\end{align}
The theorem follows after substituting \eqref{ab-6-7-proof-1} into \eqref{ab-6-7-proof}.
\end{proof}
\begin{corollary}
Identity \eqref{mock-6-7} holds. In addition, we have
\begin{align}
&\sum_{n=0}^\infty \frac{q^{(n^2-n)/2}(q^2;q^2)_n}{(q^3;q^3)_n} \nonumber \\
=&\frac{3}{2}\frac{(-q;q)_\infty}{(q;q)_\infty}\sum_{n=-\infty}^\infty \frac{(-1)^nq^{n^2}(1+q^n)^2}{1+q^n+q^{2n}} \label{6-7-cor-1}\\
=&\frac{3}{2}-\frac{3(1+\zeta_3)}{2(1-\zeta_3)}\left(m(\zeta_3^2q,q^2,\zeta_3^2)-m(\zeta_3q,q^2,\zeta_3) \right), \label{6-7-cor-1-simplify} \\
&\sum_{n=0}^\infty \frac{(-1)^nq^{n^2}(q;q)_{2n}}{(q^6;q^6)_{n}} \nonumber \\
=&\frac{3}{2}\frac{(q;q^2)_\infty}{(q^2;q^2)_\infty}\sum_{n=-\infty}^\infty \frac{q^{2n^2+n}(1+q^{2n})}{1+q^{2n}+q^{4n}} \label{6-7-cor-2} \\
=&\frac{3}{2}+\frac{1}{2}\frac{J_{1}^4}{J_{2}^2J_{3}}+\frac{3(1+\zeta_3)}{2(1-\zeta_3)}q^{-1}\left(m(-\zeta_3q^{2},q,\zeta_3^2)-m(-\zeta_3^2q^2,q,\zeta_3) \right). \label{6-7-cor-2-simplify}
\end{align}
\end{corollary}
\begin{proof}
Taking $(a,b)\rightarrow (0,0)$ and $(0, -1)$ in Theorem \ref{thm-ab-6-7}, we obtain \eqref{mock-6-7} and \eqref{6-7-cor-1}, respectively.

Taking $(q,a,b)\rightarrow (q^2, 0, q)$ in Theorem \ref{thm-ab-6-7}, we obtain \eqref{6-7-cor-2}.

Note that
\begin{align}
\sum_{n=-\infty}^\infty \frac{(-1)^nq^{n^2}(1+q^n)^2}{1+q^n+q^{2n}}=\sum_{n=-\infty}^\infty (-1)^nq^{n^2}+\sum_{n=-\infty}^\infty \frac{(-1)^nq^{n^2+n}}{1+q^n+q^{2n}}. \label{revise-h-use-0}
\end{align}
By \eqref{h-defn} we have
\begin{align}
\sum_{n=-\infty}^\infty \frac{(-1)^nq^{n^2+n}}{1+q^n+q^{2n}}=& \frac{1}{1-\zeta_3}\left(\sum_{n=-\infty}^\infty \frac{(-1)^nq^{n^2+n}}{1-\zeta_3q^n}- \zeta_3\sum_{n=-\infty}^\infty \frac{(-1)^nq^{n^2+n}}{1-\zeta_3^2q^n}\right) \nonumber \\
=&\frac{\zeta_3^2}{1-\zeta_3}j(q;q^2)\left(h(\zeta_3^2q,q)-h(\zeta_3q,q) \right). \label{revise-h-use-1}
\end{align}
Substituting \eqref{revise-h-use-1} into \eqref{revise-h-use-0} and using \eqref{h-m}, we will eventually arrive at \eqref{6-7-cor-1-simplify}.

Next, we have
\begin{align}
\sum_{n=-\infty}^\infty \frac{q^{2n^2+n}(1+q^{2n})}{1+q^{2n}+q^{4n}}=\sum_{n=-\infty}^\infty q^{2n^2+n}-\sum_{n=-\infty}^\infty \frac{q^{2n^2+5n}}{1+q^{2n}+q^{4n}}. \label{revise-k-use-0}
\end{align}
Using \eqref{k-defn} we deduce that
\begin{align}
\sum_{n=-\infty}^\infty \frac{q^{2n^2+5n}}{1+q^{2n}+q^{4n}} =& \frac{1}{1-\zeta_3}\sum_{n=-\infty}^\infty \left(\frac{q^{2n^2+5n}}{1-\zeta_3q^{2n}}-\frac{\zeta_3q^{2n^2+5n}}{1-\zeta_3^2q^{2n}} \right) \nonumber \\
=&\frac{1}{1-\zeta_3}q^{-3}\sum_{n=-\infty}^\infty \left(\frac{q^{2n^2+n}}{1-\zeta_3q^{2n-2}}-\frac{\zeta_3q^{2n^2+n}}{1-\zeta_3^2q^{2n-2}} \right) \nonumber \\
=&\frac{1}{1-\zeta_3}q^{-4}j(-q;q^4)\left(\zeta_3^{-1}k(\zeta_3^{-1}q^{-1},q)-\zeta_3k(\zeta_3q^{-1},q)  \right).
\end{align}
Here for the second equality we replaced $n$ by $n-1$. Now using \eqref{k-id} we will eventually arrive at \eqref{6-7-cor-2-simplify}.
\end{proof}

\subsection{Representations for $\phi_{-}^{(6)}(q)$}
\begin{theorem}\label{thm-ab-6-8}
For $\max \{|abq^2|, |aq^4|, |bq^4|\}<1$, we have
\begin{align}
&\frac{(q^2,abq^2;q^2)_\infty}{(aq^4,bq^4;q^2)_\infty}
{}_3\phi_2\bigg(\genfrac{}{}{0pt}{}{q^2/a,q^2/b,q^2}{0,q^3};q^2,abq^2\bigg) \nonumber\\
=&(1-q)\sum_{n=0}^{\infty}(1-q^{4n+4})\frac{(q^2/a,q^2/b;q^2)_n}{(aq^4,bq^4;q^2)_n}(-ab)^nq^{3n^2+6n}
\sum_{j=0}^{n}(1+q^{2j+1})q^{-2j^2-3j}.
\end{align}
\end{theorem}
\begin{proof}
Taking $(q,c,d,u,v) \rightarrow (q^2, 0, q^3, q^4, q^2)$ in Theorem \ref{thm-main}, we deduce that
\begin{align}
&\frac{(q^4,abq^2;q^2)_\infty}{(aq^4,bq^4;q^2)_\infty}
{}_3\phi_2\bigg(\genfrac{}{}{0pt}{}{q^2/a,q^2/b,q^2}{0,q^3};q^2,abq^2\bigg) \nonumber\\
=&\sum_{n=0}^{\infty}(1-q^{4n+4})\frac{(q^4,q^2/a,q^2/b;q^2)_n}{(q^2,aq^4,bq^4;q^2)_n}(-ab)^nq^{n^2+n}
{}_3\phi_2\bigg(\genfrac{}{}{0pt}{}{q^{-2n},q^{2n+4},q^2}{0,q^3};q^2,q^2\bigg). \label{ab-6-8-proof}
\end{align}
Taking $(q,\alpha, c)\rightarrow (q^2, q^2, q)$ in Lemma \ref{meq:2}, we obtain
\begin{align}
{}_3\phi_2\bigg(\genfrac{}{}{0pt}{}{q^{-2n},q^{2n+4},q^2}{0,q^3};q^2,q^2\bigg)
=(1-q)\frac{q^{2n^2+5n}}{1-q^{2n+2}}\sum_{j=0}^{n}(1+q^{2j+1})q^{-2j^2-3j}. \label{ab-6-8-proof-1}
\end{align}
The theorem follows after substituting \eqref{ab-6-8-proof-1} into \eqref{ab-6-8-proof}.
\end{proof}
\begin{corollary}
Identity \eqref{mock-6-8} holds. In addition, we have
\begin{align}
&\sum_{n=0}^\infty \frac{(-1)^nq^{n^2+3n}(q^2;q^2)_n}{(q;q^2)_{n+1}}=\sum_{n=0}^\infty (1+q^{2n+2})q^{4n^2+7n}\sum_{j=0}^n(1+q^{2j+1})q^{-2j^2-3j} \label{6-8-cor-1} \\
=& -2q^{-3}+\overline{f}_{1,3,1}(-q^6,-q^{12},q^4)+q^{11}\overline{f}_{1,3,1}(-q^{14},-q^{20},q^4), \label{6-8-cor-1-H} \\
&\sum_{n=0}^\infty \frac{q^{n^2+3n}(-q^2;q^2)_n}{(q;q^2)_{n+1}} \nonumber \\
=& \frac{(-q^2;q^2)_\infty}{(q^2;q^2)_\infty}\sum_{n=0}^\infty (-1)^n(1-q^{2n+2})q^{4n^2+7n}\sum_{j=0}^n(1+q^{2j+1})q^{-2j^2-3j} \label{6-8-cor-2} \\
=& \frac{J_4}{J_2^2}\left(f_{1,3,1}(q^6,q^{12},q^4)-q^{11}f_{1,3,1}(q^{14},q^{20},q^4) \right) \label{6-8-cor-2-H} \\
 =& -q^{-2}m(q^3,q^8,q^{-3/2})+q^{-2}\frac{J_{\frac{7}{2},8}^2J_{8}^3}{J_{\frac{3}{2},8}^2J_{2,8}J_{3,8}}. \label{6-8-cor-2-A}
\end{align}
\end{corollary}
\begin{proof}
Taking $(a,b)\rightarrow (-1,-q^{-1})$, $(0, 1)$ and $(0,-1)$ in Theorem \ref{thm-ab-6-8}, we obtain \eqref{mock-6-8}, \eqref{6-8-cor-1} and \eqref{6-8-cor-2}, respectively.

We have
\begin{align}
&\sum_{j=0}^n (1+q^{2j+1})q^{-2j^2-3j}\nonumber \\
=& \sum_{j=0}^nq^{-2j^2-3j}+\sum_{j=0}^nq^{-2j^2-j+1} \nonumber \\
=& \sum_{j=0}^nq^{-2j^2-3j}+\sum_{j=-n-1}^{-1}q^{-2j^2-3j} \nonumber \\
=&\sum_{j=-n}^nq^{-2j^2-3j}+q^{-2n^2-n+1}. \label{revise-6-8-cor-add}
\end{align}
Here for the second equality we have replaced $j$ by $-j-1$ in the second sum. Now substituting \eqref{revise-6-8-cor-add} into \eqref{6-8-cor-1} and \eqref{6-8-cor-2}, after routine arguments and simplifications, we get \eqref{6-8-cor-1-H} and \eqref{6-8-cor-2-H}, respectively.

The expression \eqref{6-8-cor-2-A} follows from \eqref{6-8-cor-2-H} and Theorem \ref{thm-fg-2} with $n=1$.
\end{proof}
\begin{rem}
The left side of \eqref{6-8-cor-2} is equal to $q^{-2}T_0^{(8)}(-q)$. See \eqref{mock-8-3-defn}. If we take $(x,q,z_1,z_0)\rightarrow (q^3,q^8,q^{-3/2},q^2)$ in Lemma \ref{lem-m-minus}, we see that \eqref{6-8-cor-2-A} is equivalent to \eqref{HM-T0-m}.
\end{rem}
\subsection{Representations for $\psi_{-}^{(6)}(q)$}
Identity \eqref{mock-6-9} has already been established in \eqref{5-8-cor-3-psi}.

\section{Mock theta functions of order 8}\label{sec-order-8}
Gordon and McIntosh \cite{Gordon-McIntosh} studied four mock theta functions of order 8. They are defined as
\begin{align}
&S_0^{(8)}(q):=\sum_{n=0}^{\infty}\frac{q^{n^2}(-q;q^2)_{n}}{(-q^2;q^2)_{n}}, \label{mock-8-1-defn} \\
&S_1^{(8)}(q):=\sum_{n=0}^{\infty}\frac{q^{n(n+2)}(-q;q^2)_{n}}{(-q^2;q^2)_{n}}, \label{mock-8-2-defn} \\
&T_0^{(8)}(q):=\sum_{n=0}^{\infty}\frac{q^{(n+1)(n+2)}(-q^2;q^2)_{n}}{(-q;q^2)_{n+1}}, \label{mock-8-3-defn} \\
&T_1^{(8)}(q):=\sum_{n=0}^{\infty}\frac{q^{n(n+1)}(-q^2;q^2)_{n}}{(-q;q^2)_{n+1}}. \label{mock-8-4-defn}
\end{align}
The following functions appear in the transformation laws of the eighth order mock theta functions:
\begin{align}
&U_0^{(8)}(q):=\sum_{n=0}^{\infty}\frac{q^{n^2}(-q;q^2)_{n}}{(-q^4;q^4)_n}=S_0(q^2)+qS_1(q^2),  \label{mock-8-5-defn} \\
&U_1^{(8)}(q):=\sum_{n=0}^{\infty}\frac{q^{(n+1)^2}(-q;q^2)_n}{(-q^2;q^4)_{n+1}}=T_0(q^2)+qT_1(q^2),  \label{mock-8-6-defn}\\
&V_0^{(8)}(q):=-1+2\sum_{n=0}^{\infty}\frac{q^{n^2}(-q;q^2)_{n}}{(q;q^2)_n}=-1+2\sum_{n=0}^{\infty}\frac{q^{2n^2}(-q^2;q^4)_{n}}{(q;q^2)_{2n+1}}, \label{mock-8-7-defn} \\ &V_1^{(8)}(q):=\sum_{n=0}^{\infty}\frac{q^{(n+1)^2}(-q;q^2)_n}{(q;q^2)_{n+1}}=\sum_{n=0}^{\infty}\frac{q^{2n^2+2n+1}(-q^4;q^4)_n}{(q;q^2)_{2n+2}}=\sum_{n=0}^{\infty}\frac{q^{n+1}(-q;q)_{2n}}{(-q^2;q^4)_{n+1}}. \label{mock-8-8-defn}
\end{align}
Appell-Lerch series representations of $U_0^{(8)}(q), U_1^{(8)}(q), V_0^{(8)}(q)$ and $V_1^{(8)}(q)$ were obtained in \cite{Gordon-McIntosh}.
\begin{theorem}\label{thm-ord-8}
\begin{align}
&U_0^{(8)}(q)=\frac{(-q;q^2)_{\infty}}{(q^2;q^2)_{\infty}}\sum_{n=-\infty}^{\infty}\frac{1+q^{2n}}{1+q^{4n}}(-1)^{n}q^{2n^2+n}, \label{mock-8-5} \\
&U_1^{(8)}(q)=q\frac{(-q;q^2)_{\infty}}{(q^2;q^2)_{\infty}}\sum_{n=-\infty}^{\infty}(-1)^n\frac{q^{2n^2+3n}}{1+q^{4n+2}}, \label{mock-8-6} \\
&V_0^{(8)}(q)=-1+2\frac{(-q^2;q^4)_{\infty}}{(q^4;q^4)_{\infty}}\sum_{n=-\infty}^{\infty}\frac{(-1)^nq^{4n^2+2n}}{1-q^{4n+1}}, \label{mock-8-7-b} \\
&V_1^{(8)}(q)=\frac{(-q^4;q^4)_\infty}{(q^4;q^4)_\infty}\sum_{n=-\infty}^{\infty}\frac{(-1)^nq^{(2n+1)^2}}{1-q^{4n+1}}. \label{mock-8-8-b}
\end{align}
\end{theorem}
From the relations
\begin{align}
&S_0^{(8)}(q)=\frac{1}{2}\left(U_0^{(8)}(q^{1/2})+U_0(-q^{1/2})\right), \quad S_1^{(8)}(q)=\frac{1}{2}q^{-1/2}\left(U_0(q^{1/2})-U_0(-q^{1/2})\right), \\
&T_0^{(8)}(q)=\frac{1}{2}\left(U_1^{(8)}(q^{1/2})+U_1(-q^{1/2})\right), \quad S_1^{(8)}(q)=\frac{1}{2}q^{-1/2}\left(U_1(q^{1/2})-U_1(-q^{1/2})\right),
\end{align}
It is clear that Appell-Lerch series representations for $S_0^{(8)}(q)$, $S_1^{(8)}(q)$, $T_0^{(8)}(q)$ and $T_{1}^{(8)}(q)$ can also be deduced from Theorem \ref{thm-ord-8}, and so here we omit them.

We will give new proofs to Theorem \ref{thm-ord-8}. Moreover, we also find the following  representations for the eighth order mock theta functions.
\begin{theorem}\label{thm-ord-8-add}
We have
\begin{align}
S_0^{(8)}(q)&=\frac{(-q;q^2)_{\infty}}{(q^2;q^2)_{\infty}}\sum_{n=0}^{\infty}q^{4n^2+n}(1-q^{6n+3})\sum_{j=-n}^{n}(-1)^jq^{-2j^2}, \label{mock-8-1} \\
S_1^{(8)}(q)&=\frac{(-q;q^2)_{\infty}}{(q^2;q^2)_{\infty}}\sum_{n=0}^{\infty}q^{4n^2+3n}(1-q^{2n+1})\sum_{j=-n}^{n}(-1)^jq^{-2j^2}, \label{mock-8-2} \\
T_0^{(8)}(q)&=q^2\frac{(-q^2;q^2)_{\infty}}{(q^2;q^2)_{\infty}}\sum_{n=0}^{\infty}q^{4n^2+7n}(1-q^{2n+2})\sum_{j=-n-1}^{n}(-1)^jq^{-2j^2-3j}, \label{mock-8-3} \\
T_1^{(8)}(q)&=\frac{(-q^2;q^2)_{\infty}}{(q^2;q^2)_{\infty}}\sum_{n=0}^{\infty}q^{4n^2+3n}(1-q^{2n+1})\sum_{j=-n}^{n}(-1)^jq^{-2j^2-j}, \label{mock-8-4} \\
V_0^{(8)}(q)&=-1+2\frac{(-q;q^2)_\infty}{(q^2;q^2)_\infty}\sum_{n=0}^{\infty}(-1)^nq^{4n^2+2n}\sum_{j=0}^{2n}q^{-\binom{j+1}{2}}(1+q^{4n+2}),  \label{mock-8-7-a-thm} \\
V_1^{(8)}(q)&=q\frac{(-q;q^2)_{\infty}}{(q^2;q^2)_{\infty}}\sum_{n=0}^{\infty}(-1)^nq^{4n^2+4n}\sum_{j=-n}^{n}q^{-2j^2-j} \label{mock-8-8-a-thm} \\
&=q\frac{(-q;q)_\infty}{(q;q)_\infty}\sum_{n=0}^\infty (-1)^n\frac{1-q^{2n+1}}{1+q^{4n+2}}q^{n^2+2n}. \label{mock-8-8-c}
\end{align}
Furthermore, \eqref{mock-8-8-c} is equivalent to \eqref{mock-8-8-b}.
\end{theorem}
From the known relations \eqref{mock-8-5-defn} and \eqref{mock-8-6-defn}, it is clear that Hecke-type series representations for $U_0^{(8)}(q)$ and $U_1^{(8)}(q)$ can be deduced from \eqref{mock-8-1}--\eqref{mock-8-4}, and so we omit them.

The Hecke-type series representations in Theorem \ref{thm-ord-8} were earlier found by Cui, Gu and Hao \cite{CGH} by Bailey pairs. They also write them in terms of $f_{a,b,c}(x,y,q)$. See \cite[Theorem 1.4]{CGH}.

\subsection{Representations of $S_0^{(8)}(q)$ and $S_1^{(8)}(q)$}

Identity \eqref{mock-8-1} has already been established in \eqref{5-1-cor-1}.

Identity \eqref{mock-8-2} has already been established in \eqref{5-5-cor-3}.

\subsection{Representations of $T_0^{(8)}(q)$}
\begin{theorem}\label{thm-ab-8-3}
For $\max \{|abq^2|, |aq^4|, |bq^4|\}<1$, we have
\begin{align}
&\frac{(q^2,abq^2;q^2)_\infty}{(aq^4,bq^4;q^2)_\infty}
{}_3\phi_2\bigg(\genfrac{}{}{0pt}{}{q^2/a,q^2/b,q^2}{0,-q^3};q^2,abq^2\bigg) \nonumber\\
=& (1+q)\sum_{n=0}^{\infty}(1-q^{4n+4})\frac{(q^2/a,q^2/b;q^2)_n}{(aq^4,bq^4;q^2)_n}(ab)^nq^{3n^2+6n}\sum_{j=0}^n(-1)^j(1-q^{2j+1})q^{-2j^2-3j}.
\end{align}
\end{theorem}
\begin{proof}
Taking $(q,c,d,u,v)\rightarrow (q^2, 0,-q^3, q^4, q^2)$ in Theorem \ref{thm-main}, we deduce that
\begin{align}
&\frac{(q^4,abq^2;q^2)_\infty}{(aq^4,bq^4;q^2)_\infty}
{}_3\phi_2\bigg(\genfrac{}{}{0pt}{}{q^2/a,q^2/b,q^2}{0,-q^3};q^2,abq^2\bigg)
\nonumber \\
=&\sum_{n=0}^{\infty}(1-q^{4n+4})\frac{1-q^{2n+2}}{1-q^2}\frac{(q^2/a,q^2/b;q^2)_n}{(aq^4,bq^4;q^2)_n}(-ab)^nq^{n^2+n}
{}_3\phi_2\bigg(\genfrac{}{}{0pt}{}{q^{-2n},q^{2n+4},q^2}{0,-q^3};q^2,q^2\bigg). \label{ab-8-3-proof}
\end{align}
Taking $(q, \alpha, c) \rightarrow (q^2, q^2, -q)$ in Lemma \ref{meq:2},  we deduce that
\begin{align}
{}_{3}\phi_{2}\bigg(\genfrac{}{}{0pt}{}{q^{-2n}, q^{2n+4},q^2}{0,-q^3};q^2,q^2\bigg)=\frac{1+q}{1-q^{2n+2}}(-1)^nq^{2n^2+5n}\sum_{j=0}^{n}(-1)^j(1-q^{2j+1})q^{-2j^2-3j}. \label{8-T0-2}
\end{align}
The theorem follows after substituting \eqref{8-T0-2} into \eqref{ab-8-3-proof}.
\end{proof}
\begin{corollary}
Identity \eqref{mock-8-3} holds. In addition, we have
\begin{align}
&\sum_{n=1}^\infty \frac{(-1)^nq^{n^2}(q;q^2)_n}{(-q;q^2)_n}\nonumber \\
=& -q\frac{(q;q^2)_\infty}{(q^2;q^2)_\infty}\times \sum_{n=0}^\infty (-1)^n(1-q^{4n+4})q^{4n^2+6n}\sum_{j=0}^n (-1)^j(1-q^{2j+1})q^{-2j^2-3j} \label{8-3-cor-1} \\
=& -q\frac{J_1}{J_2^2}\left(f_{1,3,1}(-q^5,-q^{11},q^4)-q^{10}f_{1,3,1}(-q^{13},-q^{19},q^4) \right)+q^{-1} \label{8-3-cor-1-H} \\
=& -m(q^{4},q^8,q^{-3/2})-q^{-1}m(q^{8},q^8,q^{-3/2})+q^{-1}+\frac{J_{1,8}^2J_{3,8}J_{4,8}J_{8}^3}{J_{\frac{1}{2},8}J_{\frac{3}{2},8}^2J_{2,8}^2J_{\frac{5}{2},8}} \label{8-3-cor-1-A} \\
=&\frac{1}{2}+q^{-1}m(1,q^8,-q)-\frac{1}{2}\frac{J_1^2J_4^3}{J_2^3J_8}. \label{8-3-cor-1-final}
\end{align}
\end{corollary}
\begin{proof}
Taking $(a,b)\rightarrow (0,-1)$ and $(0, q^{-1})$ in Theorem \ref{thm-ab-8-3}, we obtain \eqref{mock-8-3} and \eqref{8-3-cor-1}, respectively.

The expression \eqref{8-3-cor-1-A} follows from \eqref{8-3-cor-1-H} and Theorem \ref{thm-fg-2} with $n=1$.

Using \eqref{eq-fact} and Lemma \ref{lem-m-minus} we see that
\begin{align}
m(q^{4},q^8,q^{-3/2})=m(q^4,q^8,-1)+\frac{1}{2}\frac{J_8^3J_{3,8}J_{4,8}}{J_{\frac{3}{2},8}^2J_{\frac{5}{2},8}^2} =\frac{1}{2}+\frac{1}{2}\frac{J_8^3J_{3,8}J_{4,8}}{J_{\frac{3}{2},8}^2J_{\frac{5}{2},8}^2}. \label{revise-sec8-T08-cor-neweq-1}
\end{align}
Using \eqref{m-id-3} we obtain
\begin{align}
m(q^8,q^8,q^{-3/2})=1-m(1,q^8,q^{-3/2}). \label{revise-sec8-T08-cor-neweq-2}
\end{align}
By Lemma \ref{lem-m-minus} we have
\begin{align}
m(1,q^8,q^{-3/2})-m(1,q^8,-q)
=-q\frac{J_{1,16}^3J_{5,16}J_{7,16}^2J_{8}J_{16}^4}{J_{\frac{1}{2},16}J_{\frac{3}{2},16}^2J_{2,16}^2J_{\frac{5}{2},16}J_{\frac{11}{2},16}J_{\frac{13}{2},16}^2J_{\frac{15}{2},16}}.   \label{revise-sec8-T08-cor-neweq-3}
\end{align}
Substituting \eqref{revise-sec8-T08-cor-neweq-1}--\eqref{revise-sec8-T08-cor-neweq-3} into \eqref{8-3-cor-1-A} and using the method in \cite{Garvan-Liang} to prove theta function identities, we arrive at \eqref{8-3-cor-1-final}.
\end{proof}
\begin{rem}
It is not difficult to see that the left side of \eqref{8-3-cor-1} equals $\frac{1}{2}V_0^{(8)}(-q)-\frac{1}{2}$.  Thus \eqref{8-3-cor-1-final} is equivalent to \eqref{HM-V08-m}.
\end{rem}
\subsection{Representations for $T_1^{(8)}(q)$, $U_0^{(8)}(q)$, $U_1^{(8)}(q)$}

Identity \eqref{mock-8-4} has already been established in \eqref{5-8-cor-2}.

Identity \eqref{mock-8-5} has already been established in \eqref{3-2-cor-3}.

Identity \eqref{mock-8-6} has already been established in \eqref{3-6-cor-3}.

\subsection{Representations for $V_0^{(8)}(q)$}
Identity \eqref{mock-8-7-a-thm} follows from \eqref{5-4-cor-1}.

To get \eqref{mock-8-7-b}, we first establish the following identity.
\begin{theorem}\label{thm-ab-8-7}
For $\max \{|ab/q^4|, |a|, |b|\}<1$, we have
\begin{align}
&\frac{(q^4,ab/q^4;q^4)_\infty}{(a,b;q^4)_\infty}
{}_3\phi_2\bigg(\genfrac{}{}{0pt}{}{q^4/a,q^4/b,q^4}{q^3,q^5};q^4,ab/q^4\bigg) \nonumber\\
=& 1-(1-q)^2\sum_{n=1}^{\infty}\frac{1+q^{4n}}{(1-q^{4n-1})(1-q^{4n+1})}\frac{(q^4/a,q^4/b;q^4)_n}{(a,b;q^4)_n}(-ab)^nq^{2n^2-2n-1}.
\end{align}
\end{theorem}
\begin{proof}
Taking $(q,c,d,u,v) \rightarrow (q^4, q^3, q^5, 1,q^4)$ in Theorem \ref{thm-main}, we deduce that
\begin{align}
&\frac{(q^4,ab/q^4;q^4)_\infty}{(a,b;q^4)_\infty}
{}_3\phi_2\bigg(\genfrac{}{}{0pt}{}{q^4/a,q^4/b,q^4}{q^3,q^5};q^4,ab/q^4\bigg) \nonumber\\
=&1+\sum_{n=1}^{\infty}(1+q^{4n})\frac{(q^4/a,q^4/b;q^4)_n}{(a,b;q^4)_n}(-ab)^nq^{2n^2-6n}
{}_3\phi_2\bigg(\genfrac{}{}{0pt}{}{q^{-4n},q^{4n},q^4}{q^3,q^5};q^4,q^4\bigg). \label{ab-8-7-2nd}
\end{align}
Replacing $q$ by $q^4$ and setting $(a,b,c)=(1, q, q^{-1})$ in \eqref{Pfaff}, we obtain
\begin{align}
{}_{3}\phi_{2}\bigg(\genfrac{}{}{0pt}{}{q^{-4n}, q^{4n},q^4}{q^3,q^5};q^4,q^4\bigg)=\frac{(1-q)(1-q^{-1})}{(1-q^{4n+1})(1-q^{4n-1})}q^{4n}. \label{mock-8-7-3}
\end{align}
The theorem then follows after substituting \eqref{mock-8-7-3} into \eqref{ab-8-7-2nd}.
\end{proof}
\begin{corollary}
Identity \eqref{mock-8-7-b} holds. In addition, we have
\begin{align}
&\sum_{n=0}^\infty \frac{(-1)^nq^{2n^2}(q^2;q^4)_n}{(q;q^2)_{2n+1}} \nonumber \\
=&\frac{(q^2;q^4)_\infty}{(q^4;q^4)_\infty}\sum_{n=-\infty}^\infty \frac{q^{4n^2+2n}}{1-q^{4n+1}} \label{8-7-cor-1} \\
=&m(-q,q^2,q)+\frac{1}{2}\frac{J_{2}^5}{J_{1}^2J_{4}^2} \label{8-7-cor-1-m} \\
=&\frac{1}{2}+\frac{1}{2}\frac{J_{2}^5}{J_{1}^2J_{4}^2}. \label{8-7-cor-1-simplify}
\end{align}
\end{corollary}
\begin{proof}
Taking $(a,b)\rightarrow (0, -q^2)$ in Theorem \ref{thm-ab-8-7}, we deduce that
\begin{align}
V_0^{(8)}(q)=-1+2\frac{(-q^2;q^4)_\infty}{(q^4;q^4)_\infty}\left(\frac{1}{1-q}+(1-q^{-1})\sum_{n=1}^\infty \frac{(1+q^{4n})(-1)^nq^{4n^2+2n}}{(1-q^{4n-1})(1-q^{4n+1})} \right).
\end{align}
Note that
\begin{align}
(1-q^{-1})(1+q^{4n})=(1-q^{4n-1})-q^{-1}(1-q^{4n+1}).
\end{align}
We obtain
\begin{align}
V_0^{(8)}(q)=-1+2\frac{(-q^2;q^4)_\infty}{(q^4;q^4)_\infty}\left(\frac{1}{1-q}+\sum_{n=1}^\infty \frac{(-1)^nq^{4n^2+2n}}{1-q^{4n+1}}-\sum_{n=1}^\infty \frac{(-1)^nq^{4n^2+2n-1}}{1-q^{4n-1}} \right).
\end{align}
Upon replacing $n$ by $-n$ in the second sum, we obtain \eqref{mock-8-7-b}.

Similarly, taking $(a,b)\rightarrow (0, q^2)$ in Theorem \ref{thm-ab-8-7}, we obtain \eqref{8-7-cor-1}.

From \eqref{k-defn} we have
\begin{align}
\sum_{n=-\infty}^\infty \frac{q^{4n^2+2n}}{1-q^{4n+1}} =q^{1/2}j(-q^2;q^8)k(q^{1/2};q^2).
\end{align}
By \eqref{k-id} we obtain \eqref{8-7-cor-1-m}.

To get \eqref{8-7-cor-1-simplify}, we need to prove that
\begin{align}
m(-q,q^2,q)=\frac{1}{2}. \label{m-special-value-1}
\end{align}
In fact, setting $(x,q,z)\rightarrow (-q^{-1},q^2,q)$ in \eqref{m-id-3}, we deduce that
\begin{align}
m(-q,q^2,q)=1+q^{-1}m(-q^{-1},q^2,q)=1-m(-q,q^2,q).
\end{align}
The equality in \eqref{m-special-value-1} then follows.
\end{proof}
\begin{rem}
(1) Comparing \eqref{3-2-cor-1-simplify} with \eqref{8-7-cor-1-simplify}, we deduce the following interesting identity:
\begin{align}
\sum_{n=0}^{\infty}\frac{(-q;q)_nq^{(n^2-n)/2}}{(-q^2;q^2)_n}=2\sum_{n=0}^\infty \frac{(-1)^nq^{2n^2}(q^2;q^4)_n}{(q;q^2)_{2n+1}}. \label{revise-relation-1}
\end{align}
(2) If we set $x=-q$ in \cite[Eq.\ (2.16)]{Mortenson-2014AIM}, we get \eqref{8-7-cor-1-m}.
\end{rem}

\subsection{Representations for $V_1^{(8)}(q)$}\label{subsec-mock-V18}
The Hecke-type series representation \eqref{mock-8-8-a-thm} follows from  \eqref{5-8-cor-4} and the first expression of $V_1^{(8)}(q)$ in \eqref{mock-8-8-defn}.

The Appell-Lerch series representation \eqref{mock-8-8-c} follows from \eqref{3-6-cor-4} and the third expression of $V_1^{(8)}(q)$ in \eqref{mock-8-8-defn}. Here we show that \eqref{mock-8-8-c} is equivalent to \eqref{mock-8-8-b}.

To get another Appell-Lerch series representation, namely \eqref{mock-8-8-b}, we establish the following result.
\begin{theorem}\label{thm-ab-8-8}
For $\max \{|ab|, |aq^4|, |bq^4|\}<1$, we have
\begin{align}
&\frac{(q^4,ab;q^4)_\infty}{(aq^4,bq^4;q^4)_\infty}
{}_3\phi_2\bigg(\genfrac{}{}{0pt}{}{q^4/a,q^4/b,q^4}{q^5,q^7};q^4,ab\bigg) \nonumber\\
=&(1-q)(1-q^3)\sum_{n=0}^{\infty}\frac{1-q^{8n+4}}{(1-q^{4n+1})(1-q^{4n+3})}\frac{(q^4/a,q^4/b;q^4)_n}{(aq^4,bq^4;q^4)_n}(-ab)^nq^{2n^2+2n}.
\end{align}
\end{theorem}
\begin{proof}
Taking $(q,c,d,u,v) \rightarrow (q^4, q^5, q^7, q^4, q^4)$ in Theorem \ref{thm-main}, we deduce that
\begin{align}
&\frac{(q^4,ab;q^4)_\infty}{(aq^4,bq^4;q^4)_\infty}
{}_3\phi_2\bigg(\genfrac{}{}{0pt}{}{q^4/a,q^4/b,q^4}{q^5,q^7};q^4,ab\bigg) \nonumber\\
&=\sum_{n=0}^{\infty}(1-q^{8n+4})\frac{(q^4/a,q^4/b;q^4)_n}{(aq^4,bq^4;q^4)_n}(-ab)^nq^{2n^2-2n}
{}_3\phi_2\bigg(\genfrac{}{}{0pt}{}{q^{-4n},q^{4n+4},q^4}{q^5,q^7};q^4,q^4\bigg). \label{ab-8-8-2nd-proof}
\end{align}
Replacing $q$ by $q^4$ and setting $(a,b,c)=(q^4,q,q^3)$ in \eqref{Pfaff}, we deduce that
\begin{align}
{}_{3}\phi_{2}\bigg(\genfrac{}{}{0pt}{}{q^{-4n}, q^{4n+4},q^4}{q^5,q^7};q^4,q^4\bigg)=\frac{(1-q)(1-q^3)q^{4n}}{(1-q^{4n+1})(1-q^{4n+3})}. \label{mock-8-8-b-3}
\end{align}
The theorem then follows after substituting \eqref{mock-8-8-b-3} into \eqref{ab-8-8-2nd-proof}.
\end{proof}
\begin{corollary}
Identity \eqref{mock-8-8-b} holds. In addition, we have
\begin{align}
&\sum_{n=0}^\infty \frac{(-1)^nq^{2n(n+1)}(q^4;q^4)_n}{(q;q^2)_{2n+2}}\nonumber \\
=& \sum_{n=0}^\infty \frac{q^{4n^2+4n}}{1-q^{4n+1}}-\sum_{n=-\infty}^{-1}\frac{q^{4n^2+4n}}{1-q^{4n+1}} \label{8-8-cor-1} \\
=& \overline{m}(-q^2,q^8,-q^8)+q\overline{m}(-q^{-2},q^8,-q^{12}), \label{8-8-cor-1-simplify} \\
&\sum_{n=0}^\infty \frac{(-1)^nq^{2n^2+4n}(q^2;q^4)_n}{(q^3;q^2)_{2n+1}}\nonumber \\
=&\frac{(q^2;q^4)_\infty}{(q^4;q^4)_\infty}\sum_{n=-\infty}^\infty \frac{q^{4n^2+6n}}{1-q^{4n+1}} \label{8-8-cor-2} \\
=&q^{-2}m(-1,q^8,-q^2)-q^{-1}m(-q^4,q^8,-q^2) \label{8-8-cor-2-simplify} \\
=&-\frac{1}{2}q^{-1}+\frac{1}{2}q^{-1}\frac{J_2^5}{J_1^2J_4^2}. \label{8-8-cor-2-final}
\end{align}
\end{corollary}
\begin{proof}
We use the second expression of $V_1^{(8)}(q)$ to prove \eqref{mock-8-8-b}. We first write
\begin{align}
V_1^{(8)}(q)=\frac{q}{(1-q)(1-q^3)}\sum_{n=0}^{\infty}\frac{(-q^4;q^4)_nq^{2n^2+2n}}{(q^5,q^7;q^4)_n}. \label{mock-8-8-b-1}
\end{align}
Taking $(a,b)\rightarrow (0,-1)$ in Theorem \ref{thm-ab-8-8}, we deduce that
\begin{align}
V_1^{(8)}(q)&=q\sum_{n=0}^{\infty}\frac{(-q^4;q^4)_nq^{2n^2+2n}}{(q;q^2)_{2n+2}} \nonumber \\
&=q\frac{(-q^4;q^4)_{\infty}}{(q^4;q^4)_{\infty}}\sum_{n=0}^{\infty}\frac{(1-q^{8n+4})(-1)^nq^{4n^2+4n}}{(1-q^{4n+1})(1-q^{4n+3})} \nonumber \\
&=q\frac{(-q^4;q^4)_{\infty}}{(q^4;q^4)_{\infty}}\sum_{n=0}^{\infty}\frac{((1-q^{4n+3})+q^{4n+3}(1-q^{4n+1}))(-1)^nq^{4n^2+4n}}{(1-q^{4n+1})(1-q^{4n+3})} \nonumber \\
&=q\frac{(-q^4;q^4)_{\infty}}{(q^4;q^4)_{\infty}} \left(\sum_{n=0}^{\infty}\frac{(-1)^nq^{4n^2+4n}}{1-q^{4n+1}}+\sum_{n=0}^{\infty}\frac{(-1)^nq^{4n^2+8n+3}}{1-q^{4n+3}} \right)\nonumber \\
&=q\frac{(-q^4;q^4)_{\infty}}{(q^4;q^4)_{\infty}}\sum_{n=-\infty}^{\infty}\frac{(-1)^nq^{4n^2+4n}}{1-q^{4n+1}}.
\end{align}
This proves \eqref{mock-8-8-b}.

In the same way, taking $(a,b)\rightarrow (0,1)$ in Theorem \ref{thm-ab-8-8}, we obtain \eqref{8-8-cor-1}.

Taking $(a,b)\rightarrow (0,q^2)$ in Theorem \ref{thm-ab-8-8}, we deduce that
\begin{align}
\sum_{n=0}^\infty \frac{(-1)^nq^{2n^2+4n}(q^2;q^4)_n}{(q^3;q^2)_{2n+1}}=(1-q)\frac{(q^2;q^4)_\infty}{(q^4;q^4)_\infty}\sum_{n=0}^\infty \frac{(1+q^{4n+2})q^{4n^2+6n}}{(1-q^{4n+1})(1-q^{4n+3})}. \label{8-8-cor-2-proof}
\end{align}
Note that
\begin{align*}
(1-q)(1+q^{4n+2})=(1-q^{4n+3})-q(1-q^{4n+1}).
\end{align*}
Applying this identity to \eqref{8-8-cor-2-proof}, after rearrangements and simplifications, we obtain \eqref{8-8-cor-2}.

Now writing \eqref{8-8-cor-2} as \eqref{8-8-cor-2-simplify} is a routine exercise. If we take $(x,q,z,z',n)\rightarrow (q,q^2,-1,-q^2,2)$ in Lemma \ref{lem-m-decompose} and use \eqref{eq-fact}, we get \eqref{8-8-cor-2-final} from \eqref{8-8-cor-2-simplify}.
\end{proof}
\begin{rem}
Comparing \eqref{8-8-cor-2-final} with \eqref{8-7-cor-1-simplify}, we obtain the following interesting identity:
\begin{align}\label{revise-sec8-new-1}
\sum_{n=0}^\infty \frac{(-1)^nq^{2n^2}(q^2;q^4)_n}{(q;q^2)_{2n+1}} =\sum_{n=0}^\infty \frac{(-1)^nq^{2n^2+4n+1}(q^2;q^4)_n}{(q^3;q^2)_{2n+1}}+1.
\end{align}
\end{rem}

Though \eqref{mock-8-8-b} and \eqref{mock-8-8-c} look different, they can be deduced from each other.
\begin{proof} [Proof of the equivalence of \eqref{mock-8-8-b} and \eqref{mock-8-8-c}]
We can rewrite \eqref{mock-8-8-c} as
\begin{align}\label{New-V18}
V_1^{(8)}(q)=&q\frac{J_2}{J_1^2}\sum_{n=-\infty}^\infty \frac{(-1)^nq^{n^2+2n}}{1+q^{4n+2}} \nonumber \\
=&q\frac{J_2}{J_1^2}\left(\sum_{n=-\infty}^\infty \frac{q^{4n^2+4n}}{1+q^{8n+2}}-\sum_{n=-\infty}^\infty \frac{q^{4n^2+8n+3}}{1+q^{8n+6}} \right)\nonumber \\
=& \frac{J_2}{J_1^2}\left(2q\frac{J_{16}^2}{J_8}m(q^2,q^8,-1)-\frac{J_8^5}{J_4^2J_{16}^2}m(q^2,q^8,-q^4)\right).
\end{align}
Here the first equality follows by splitting the sum in \eqref{mock-8-8-c} into two sums and then replacing $n$ by $-n-1$ in the second sum. The second equality follows by decomposing the sum in the first line according to the parity of $n$.

We can also rewrite \eqref{mock-8-8-b} as \cite[Eq.\ (5.42)]{Hickerson-Mortenson}
\begin{align}
V_1^{(8)}(q)=-m(q^2,q^8,q). \label{HM-V18}
\end{align}
Taking $(x,q,z_1,z_0)\rightarrow (q^2,q^8,-1,q)$ in Lemma \ref{lem-m-minus}, we deduce that
\begin{align}\label{V18-equiv-1}
m(q^2,q^8,-1)-m(q^2,q^8,q)=q\frac{J_8^3\overline{J}_{-1,8}\overline{J}_{3,8}}{J_{1,8}\overline{J}_{0,8}J_{3,8}\overline{J}_{2,8}}
=\frac{1}{2}\frac{J_2^4J_8^4}{J_1^2J_4^3J_{16}^2}.
\end{align}
Taking $(x,q,z_1,z_0)\rightarrow (q^2,q^8,-q^4,q)$ in Lemma \ref{lem-m-minus}, we deduce that
\begin{align}\label{V18-equiv-2}
m(q^2,q^8,-q^4)-m(q^2,q^8,q)=q\frac{J_8^3\overline{J}_{3,8}\overline{J}_{7,8}}{J_{1,8}\overline{J}_{4,8}J_{3,8}\overline{J}_{6,8}}
=q\frac{J_2^4J_{16}^2}{J_1^2J_4J_8^2}.
\end{align}
Note that
\begin{align}
\frac{J_8^5}{J_4^2J_{16}^2}-2q\frac{J_{16}^2}{J_8}=& \sum_{n=-\infty}^\infty q^{4n^2}-\sum_{n=-\infty}^\infty q^{(2n+1)^2}\nonumber \\
&=\sum_{n=-\infty}^\infty (-1)^nq^{n^2}=\frac{J_1^2}{J_2}. \label{V18-equiv-3}
\end{align}
Now substituting \eqref{V18-equiv-1} and \eqref{V18-equiv-2} into \eqref{New-V18}, and then applying \eqref{V18-equiv-3} for simplifications, we arrive at \eqref{HM-V18}.
\end{proof}
\section{Concluding Remarks}
There are many other problems which can be investigated in the future. For example, a natural question arises if we look again at our proofs of the representations of the mock theta functions. Below in  Table \ref{tab-list} we list the theorems that can be applied to give representations to at least two mock theta functions:
\begin{table}[h]
\caption{} \label{tab-list}
\begin{tabular}{l|l}
  \hline
  Theorem \ref{thm-ab-2-1-1st} & $A^{(2)}(q)$, $B^{(2)}(q)$, $\omega^{(3)}(q)$ \\
  \hline
  Theorem \ref{thm-ab-2-1-2nd} &$A^{(2)}(q)$, $\psi^{(3)}(q)$, $\sigma^{(6)}(q)$ \\
  \hline
  Theorem \ref{thm-ab-2-2} & $B^{(2)}(q)$, $\nu^{(3)}(q)$, $\rho^{(6)}(q)$ \\
  \hline
  Theorem \ref{thm-ab-2-3} & $\mu^{(2)}(q)$, $f^{(3)}(q)$ \\
  \hline
  Theorem \ref{thm-ab-3-2} & $\phi^{(3)}(q)$, $U_0^{(8)}(q)$ \\
  \hline
  Theorem \ref{thm-ab-3-3} & $\psi^{(3)}(q)$, $\sigma^{(6)}(q)$ \\
  \hline
  Theorem \ref{thm-ab-3-6} & $\nu^{(3)}(q)$, $U_1^{(8)}(q)$, $V_1^{(8)}(q)$ \\
  \hline
  Theorem \ref{thm-ab-5-1} & $f_0^{(5)}(q)$, $S_0^{(8)}(q)$ \\
  \hline
  Theorem \ref{thm-ab-5-4} & $F_0^{(5)}(q)$, $V_0^{(8)}(q)$, $T_0^{(8)}(q)$ \\
  \hline
  Theorem \ref{meq:2.1} & $f_1^{(5)}(q)$, $\lambda^{(6)}(q)$, $S_1^{(8)}(q)$ \\
  \hline
  Theorem \ref{thm-ab-5-6} & $\phi_1^{(5)}(q)$, $\psi_1^{(5)}(q)$ \\
  \hline
  Theorem \ref{thm-ab-5-8} & $F_1^{(5)}(q)$, $\psi_{-}^{(6)}(q)$, $T_1^{(8)}(q)$, $V_1^{(8)}(q)$ \\
  \hline
  Theorem \ref{thm-ab-6-1} & $\phi^{(3)}(q)$, $\nu^{(3)}(q)$, $\phi^{(6)}(q)$  \\
  \hline
  Theorem \ref{thm-ab-6-2} & $\nu^{(3)}(q)$, $\psi^{(6)}(q)$ \\
  \hline
  Theorem \ref{thm-ab-6-8} & $\phi_{-}^{(6)}(q)$, $T_0^{(8)}(q)$ \\
  \hline
\end{tabular}
\end{table}

\noindent Clearly, the mock theta functions in the same row of this table share similar Hecke-type series representations. However, it is not clear to us whether they have further deep relations.

Another question we can ask is about mock theta functions of orders 7 and 10. The alert readers may have already noted that we have skipped them. The reason is that we do face difficulties in applying the previous procedure to these functions.

In his last letter to Hardy, Ramanujan gave three mock theta functions of order 7:
\begin{align}
&\mathcal{F}_0^{(7)}(q):=\sum_{n=0}^\infty \frac{q^{n^2}}{(q^{n+1};q)_n}, \label{mock-7-1-defn} \\
&\mathcal{F}_1^{(7)}(q):=\sum_{n=1}^\infty \frac{q^{n^2}}{(q^n;q)_n}, \label{mock-7-2-defn} \\
&\mathcal{F}_2^{(7)}(q):=\sum_{n=0}^\infty \frac{q^{n^2+n}}{(q^{n+1};q)_{n+1}}. \label{mock-7-3-defn}
\end{align}
In his lost notebook \cite{lostnotebook}, Ramanujan recorded four mock theta functions of order 10:
\begin{align}
\phi^{(10)}(q)&:=\sum_{n=0}^\infty \frac{q^{n(n+1)/2}}{(q;q^2)_{n+1}}, \label{mock-10-1-defn} \\
\psi^{(10)}(q)&:=\sum_{n=1}^\infty \frac{q^{n(n+1)/2}}{(q;q^2)_n}, \label{mock-10-2-defn} \\
X^{(10)}(q)&:=\sum_{n=0}^\infty \frac{(-1)^nq^{n^2}}{(-q;q)_{2n}}, \label{mock-10-3-defn} \\
\chi^{(10)}(q)&:=\sum_{n=1}^\infty \frac{(-1)^{n-1}q^{n^2}}{(-q;q)_{2n-1}}. \label{mock-10-4-defn}
\end{align}
Using Bailey pairs, Andrews \cite{Andrews-TAMS} gave the following Hecke-type series representations for the seventh order functions:
\begin{align}
\mathcal{F}_0^{(7)}(q)=&~\frac{1}{(q;q)_\infty}\Big(\sum_{n=0}^\infty\sum_{|j|\leq n}q^{7n^2+n-j^2}(1-q^{12n+6}) \nonumber \\
&-2q\sum_{n=0}^\infty\sum_{j=0}^nq^{7n^2+8n-j^2-j}(1-q^{12n+13})  \Big), \label{mock-7-1} \\
\mathcal{F}_1^{(7)}(q)=&~\frac{1}{(q;q)_\infty}\Big(-2\sum_{n=0}^\infty\sum_{j=0}^{n-1}q^{7n^2-2n-j^2-j}(1-q^{4n})\nonumber \\
&+\sum_{n=0}^\infty\sum_{|j|\leq n}q^{7n^2+5n+1-j^2}(1-q^{4n+2})  \Big), \label{mock-7-2} \\
\mathcal{F}_2^{(7)}(q)=&~\frac{1}{(q;q)_\infty}\Big(\sum_{n=0}^\infty\sum_{|j|\leq n}q^{7n^2+n-j^2}(1-q^{8n+3})\nonumber \\
&-2q^2\sum_{n=0}^\infty\sum_{j=0}^nq^{7n^2+8n-j^2-j}(1-q^{8n+7}) \Big). \label{mock-7-3}
\end{align}
Using Bailey pairs, Garvan \cite{Garvan-arXiv} found new Hecke-type series for these three functions.

Again by Bailey's Lemma, Choi \cite{Choi-1,Choi-2} gave Hecke-type series representations for the tenth order mock theta functions:
\begin{align}
\phi^{(10)}(q)=&~\frac{(q^2;q^2)_\infty}{(q;q)_\infty^2}\Big(\sum_{n=0}^\infty\sum_{|j|\leq n}q^{5n^2+2n-j^2}(1-q^{6n+3})\nonumber \\
&-2\sum_{n=0}^\infty\sum_{j=0}^nq^{5n^2+7n+2-j^2-j}(1-q^{6n+6}) \Big), \label{mock-10-1} \\
\psi^{(10)}(q)=&~\frac{(q^2;q^2)_\infty}{(q;q)_\infty^2}\Big(\sum_{n=0}^\infty\sum_{|j|\leq n}q^{5n^2+4n+1-j^2}(1-q^{2n+1})\nonumber \\
&-2\sum_{n=0}^\infty\sum_{j=0}^nq^{5n^2+9n+4-j^2-j}(1-q^{2n+2}) \Big), \label{mock-10-2} \\
X^{(10)}(q)=&~\frac{(q;q)_\infty}{(q^2;q^2)_\infty^2}\Big(\sum_{n=0}^\infty\sum_{|j|\leq n}q^{10n^2+2n-2j^2}(1-q^{16n+8}) \nonumber \\
&+2\sum_{n=0}^\infty\sum_{j=0}^nq^{10n^2+12n+3-2j^2-2j}(1-q^{16n+16}) \Big), \label{mock-10-3} \\
\chi^{(10)}(q)=&~\frac{(q;q)_\infty}{(q^2;q^2)_\infty^2}\Big( 2\sum_{n=0}^\infty\sum_{j=0}^nq^{10n^2+16n+6-2j^2-2j}(1-q^{8n+8}) \nonumber \\ &+\sum_{n=0}^\infty\sum_{|j|\leq n}q^{10n^2+6n+1-2j^2}(1-q^{8n+4})\Big). \label{mock-10-4}
\end{align}
Unfortunately we are not able to provide new proofs for the representations in \eqref{mock-7-1}--\eqref{mock-10-4}. The major difficulties arise from evaluating certain terminated ${}_4\phi_{3}$ and ${}_{3}\phi_{2}$ series. For example, for the seventh order mock theta function $\mathcal{F}_{0}^{(7)}(q)$, we first write it as
\begin{align}
\mathcal{F}_0^{(7)}(q)=\sum_{n=0}^\infty\frac{q^{n^2}}{(-q,q^{1/2}, -q^{1/2};q)_n}.
\end{align}
Taking $(m, \alpha, a, b, b_1, b_2, b_3, c_1, c_2, c_3, z)\rightarrow (3, 1, 0, 0, q, 0, 0, q^{1/2}, -q^{1/2}, -q, 1)$ in Theorem \ref{thm-key}, we deduce that
\begin{align}
\mathcal{F}_0^{(7)}(q)=\frac{1}{(q;q)_\infty}\left(1+\sum_{n=1}^\infty (-1)^n(1+q^n)q^{(3n^2-n)/2}{}_{4}\phi_{3}\bigg(\genfrac{}{}{0pt}{}{q^{-n}, q^{n},q, 0}{-q,q^{1/2}, -q^{1/2}};q,q\bigg)\right).
\end{align}
We do not know how to evaluate the terminated ${}_{4}\phi_{3}$ series on the right side. The same problems exist for $\mathcal{F}_{1}^{(7)}(q)$ and $\mathcal{F}_2^{(7)}(q)$.

As for the tenth order functions, the situation is quite similar. Taking $\phi^{(10)}(q)$ as an example, we first write it as
\begin{align}
\phi^{(10)}(q)=\frac{1}{1-q}\sum_{n=0}^\infty\frac{q^{n(n+1)/2}}{(q^{3/2},-q^{3/2};q)_{n}}.
\end{align}
Taking $(c,d,u,v)=(q^{3/2},-q^{3/2},-q,0)$ in Theorem \ref{thm-main}, we deduce that
\begin{align}
&\frac{(-q,-ab;q)_\infty}{(-aq,-bq;q)_\infty}
{}_3\phi_2\bigg(\genfrac{}{}{0pt}{}{q/a,q/b,0}{q^{3/2},-q^{3/2}};q,-ab\bigg) \nonumber\\
=&\sum_{n=0}^{\infty}(1+q^{2n+1})\frac{(-q,q/a,q/b;q)_n}{(q,-aq,-bq;q)_n}(ab)^nq^{(n^2-n)/{2}}
{}_3\phi_2\bigg(\genfrac{}{}{0pt}{}{q^{-n},-q^{n+1},0}{q^{3/2},-q^{3/2}};q,q\bigg). \label{mock-10-proof}
\end{align}
Setting $(a,b)\rightarrow (0,1)$ in \eqref{mock-10-proof} would yield a representation for $\phi^{(10)}(q)$. However, we cannot find useful formulas for evaluating the terminated ${}_3\phi_2$ series on the right side of \eqref{mock-10-proof}. It would be quite interesting if one could settle the difficulties in calculating these kinds of terminated sums.

\subsection*{Acknowledgements}
We thank Prof.\ George Andrews for pointing out the reference \cite{Andrews-GJM} where the identity \eqref{2-3-cor-3-unusual} appeared.  We are also grateful to the referee for his/her valuable comments. The second author was supported by the National Natural Science Foundation of China (11801424), the Fundamental Research Funds for the Central Universities (Project No. 2042018kf0027, Grant No. 1301--413000053) and a start-up research grant (No. 1301--413100048) of the Wuhan University.

\end{document}